\documentclass[a4paper]{amsart}
\usepackage{amsmath,amssymb}
\usepackage{stmaryrd}
\usepackage{tikz}
\usepackage{hyperref}

\title[Probabilistic Computability and Choice]
{Probabilistic Computability and Choice}

\author{Vasco Brattka}
\address{Faculty of Computer Science, Universit\"at der Bundeswehr M\"unchen, Germany and 
             Department of Mathematics \& Applied Mathematics, University of Cape Town, South Africa\footnote{Vasco Brattka is supported by the National Research Foundation of South Africa.}} 
\email{Vasco.Brattka@cca-net.de}

\author{Guido Gherardi}
\address{Faculty of Computer Science, Universit\"at der Bundeswehr M\"unchen, Germany} 
\email{Guido.Gherardi@unibw.de}

\author{Rupert H\"olzl}
\address{Department of Mathematics, Faculty of Science, National University of Singapore\footnote{Rupert H\"olzl was supported by a Feodor Lynen postdoctoral research fellowship by the Alexander von Humboldt Foundation.}} 
\email{r@hoelzl.fr}

\date{\today}

\def\AA{{\mathcal A}}

\def\CC{{\mathcal C}}

\def\FF{{\mathcal F}}

\def\OO{{\mathcal O}}

\def\UU{{\mathcal U}}

\def\IN{{\mathbb{N}}}

\def\IR{{\mathbb{R}}}
\def\IS{{\mathbb{S}}}

\def\TO{\Longrightarrow}
\def\In{\subseteq}

\def\prefix{\sqsubseteq}

\def\mto{\rightrightarrows}

\def\id{{\rm id}}

\def\dom{{\rm dom}}
\def\range{{\rm range}}

\def\diam{{\rm diam}}
\def\dist{{\rm dist}}

\def\sgn{{\rm sgn}}

\def\Baire{{\IN^\IN}}

\def\Tr{{\rm Tr}}

\newcommand{\SO}[1]{{{\bf\Sigma}^0_{#1}}}

\newcommand{\PO}[1]{{{\bf\Pi}^0_{#1}}}

\newcommand{\pO}[1]{{\Pi^0_{#1}}}

\def\LPO{\text{\rm\sffamily LPO}}
\def\LLPO{\text{\rm\sffamily LLPO}}
\def\WKL{\text{\rm\sffamily WKL}}
\def\WWKL{\text{\rm\sffamily WWKL}}

\def\IVT{\text{\rm\sffamily IVT}}

\def\B{\text{\rm\sffamily B}}

\def\P{\mbox{\rm\sffamily P}}
\def\PC{\mbox{\rm\sffamily PC}}
\def\C{\mbox{\rm\sffamily C}}
\def\ConC{\mbox{\rm\sffamily CC}}
\def\UC{\mbox{\rm\sffamily UC}}
\def\PCC{\mbox{\rm\sffamily PCC}}
\def\AUC{\mbox{\rm\sffamily AUC}}

\def\LPO{\mbox{\rm\sffamily LPO}}
\def\LLPO{\mbox{\rm\sffamily LLPO}}

\def\K{\text{\rm\sffamily K}}

\def\Low{\text{\rm\sffamily L}}
\def\CL{\text{\rm\sffamily CL}}

\def\PCL{\text{\rm\sffamily PCL}}

\def\RDIV{\text{\rm\sffamily RDIV}}
\def\NASH{\text{\rm\sffamily NASH}}

\def\LDL{\text{\rm\sffamily LDL}}

\def\leqr{\mathop{\leq_{\mathrm{r}}}}

\def\leqT{\mathop{\leq_{\mathrm{T}}}}

\def\leqW{\mathop{\leq_{\mathrm{W}}}}
\def\equivW{\mathop{\equiv_{\mathrm{W}}}}

\def\leqSW{\mathop{\leq_{\mathrm{sW}}}}

\def\equivSW{\mathop{\equiv_{\mathrm{sW}}}}

\def\nleqW{\mathop{\not\leq_{\mathrm{W}}}}
\def\nleqSW{\mathop{\not\leq_{\mathrm{sW}}}}

\def\lW{\mathop{<_{\mathrm{W}}}}
\def\lSW{\mathop{<_{\mathrm{sW}}}}
\def\nW{\mathop{|_{\mathrm{W}}}}
\def\nSW{\mathop{|_{\mathrm{sW}}}}

\DeclareMathOperator*{\bigtimes}{\vartimes}

\def\stars{*_{\rm s}\;\!}

\newcommand{\dash}{\mbox{-}}

\newtheorem{theorem}{Theorem}[section]
\newtheorem{proposition}[theorem]{Proposition}
\newtheorem{lemma}[theorem]{Lemma}

\newtheorem{corollary}[theorem]{Corollary}

\theoremstyle{definition}
\newtheorem{definition}[theorem]{Definition}

\begin{document}

\ \\[-1cm]
\begin{abstract}
We study the computational power of randomized computations on infinite objects, such as real numbers.
In particular, we introduce the concept of a Las Vegas computable multi-valued function, which is a function that
can be computed on a probabilistic Turing machine that receives a random binary sequence as
auxiliary input. The machine can take advantage of this random sequence, but it always has to produce
a correct result or to stop the computation after finite time if the random advice is not successful. 
With positive probability the random advice has to be successful. 
We characterize the class of Las Vegas computable functions in the Weihrauch lattice with the help of 
probabilistic choice principles and Weak Weak K\H{o}nig's Lemma. 
Among other things we prove an Independent Choice Theorem that implies that Las Vegas
computable functions are closed under composition.
In a case study we show that Nash equilibria are Las Vegas computable, while zeros of continuous functions
with sign changes cannot be computed on Las Vegas machines. However, we show that the
latter problem admits randomized algorithms with weaker failure recognition mechanisms.
The last mentioned results can be interpreted such that the Intermediate Value Theorem is reducible
to the jump of Weak Weak K\H{o}nig's Lemma, but not to Weak Weak K\H{o}nig's Lemma itself.
These examples also demonstrate that Las Vegas computable functions form a proper superclass of the class of 
computable functions and a proper subclass of the class of non-deterministically computable functions. 
 We also study the impact of specific lower bounds on the success probabilities, which leads
to a strict hierarchy of classes. In particular, the classical technique of probability amplification fails
for computations on infinite objects. We also investigate the dependency on the underlying probability space.
Besides Cantor space, we study the natural numbers, the Euclidean space and Baire space. \ \bigskip 
\end{abstract}

\maketitle


\begin{footnotesize}
\setcounter{tocdepth}{1}
\tableofcontents
\end{footnotesize}

\section{Introduction}
\label{sec:introduction}
 
What is the computational power of a sequence of coin flips for computations on real numbers?
While the power of randomized algorithms has been studied in the discrete setting for a long time (for a survey see for instance the text book by Motwani and Raghavan \cite{MR95}), 
very little is known for computations on real numbers.\footnote{See however the work of Hertling and Weihrauch \cite{HW03}, G{\'{a}}cs \cite{Gac05}, Bosserhoff \cite{Bos08b,Bos08f} and Hoyrup and Rojas \cite{HR09},
Freer and Roy \cite{FR12} for some results on randomness and probabilistic computability over topological spaces in this direction and for further references.}

In the discrete setting of decision problems randomization has no impact on what can be computed in principle, which follows from the Theorem
of Sacks and its predecessors (see the discussion of Theorem~\ref{thm:Sacks}), but it might have
an impact on the computational complexity (whether it does or not is still not known for polynomial time complexity).

As we will see, randomization actually increases the computational power in general for computations of multi-valued functions in the infinite setting;
and the question is to which extent it does so.

The purpose of this study is to analyze this question from the following perspective. Given a problem (a partial multi-valued function) $f:\In X\mto Y$:
\begin{itemize}
\item
Imagine that a Turing machine upon input of $x\in X$ receives a second auxiliary input $r\in R$
and is supposed to produce a result $y\in f(x)$ with the help of this additional ``random advice'' $r$.
\item
We will require that such a computation is successful for every fixed $x$ 
with a certain probability, i.e., the set $S_x$ of successful advices $r$ for input $x$ has to have a certain measure.
\item
Additional conditions can be imposed on how the machine has to recognize the possible failure of an advice
during the course of the computation.
\end{itemize}

Hence, this scenario can be seen as a way to formalize randomized algorithms over infinite objects $x,y$,
where the computation is performed using some additional ``random'' input $r$.
Now we can distinguish several ways in which this scenario can be refined:

\begin{enumerate}
\item {\bf Probability space.} The probability space $R$ can be chosen in different ways:
\begin{enumerate}
\item $R=2^\IN$ reflects the situation where the computation depends on a sequence of ``coin tosses'' (i.e., a sequence of zeros and ones),
\item $R=\IN$ reflects the situation where the computation depends on a randomly chosen natural number,
\item $R=\IN\times2^\IN$ reflects the situation where the computation depends on a randomly chosen natural number and a randomly chosen sequence of zeros and ones 
         (as we will see, one can more or less equivalently work with $R=\IR$ and choose a real number $x\in\IR$),
\item $R=\IN^\IN$ reflects the situation where the computation depends on a randomly chosen sequence of natural numbers.
\end{enumerate}
In each case $R$ is equipped with some natural canonical (probability) measure $\mu$. We also allow measures which are not probability measures.
\item {\bf Success probability.} Different success probabilities can be imposed:
\begin{enumerate}
\item $\mu(S_x)>0$ for all admissible inputs $x$ reflects positive success probability, which is the weakest meaningful requirement in this regard,
\item $\mu(S_x)\in I$ reflects more generally a success probability in some fixed interval $I\In\IR$ for all admissible inputs $x$.
\end{enumerate}
\item {\bf Failure recognition.} Finally, we can impose different conditions on how the machine has to recognize the failure of an advice $r$:
\begin{enumerate}
\item Las Vegas algorithms\footnote{Our understanding of Las Vegas algorithms for infinite computations is very close to Babai's original understanding of this concept, see \cite{Bab79}.}
require that the machine always produces a correct result and otherwise recognizes at some
        finite stage that the advice is unsuccessful and stops the computation in this case.
\item Higher order probabilistic algorithms are defined by weaker failure recognition
        mechanisms of the machine.
\end{enumerate} 
\end{enumerate}

The most important and most natural scenario for us is the one with a sequence of coin tosses $R=2^\IN$, 
with positive success probability in $I=(0,1]$ and a Las Vegas failure recognition mechanism. 
In this setting we will simply speak of {\em Las Vegas computability}.

We briefly summarize some major results that we are going to provide.
After the introduction of some preliminaries in Section~\ref{sec:preliminaries} we formally introduce the
concept of Las Vegas computability in Section~\ref{sec:Las-Vegas} and we characterize it with the help
of a probabilistic choice operation $\P_I\C_X$.
Intuitively, the problem $\P_I\C_X$ is the problem of finding
a point in a given closed set $A\In X$ of measure $\mu(A)\in I$ (we assume that $X$ is a topological space
equipped with a suitable measure $\mu$ and $I\In\IR$ is an interval). 
In case of $I=(0,\infty]$ we briefly write $\PC_X$ instead of $\P_I\C_X$ and in this form this problem
was already introduced and studied under the name {\em positive choice} by the first author and Arno Pauly in \cite{BP10}.
Using the concept of Weihrauch reducibility $\leqW$ we show that 
\[\mbox{$f$ is Las Vegas computable $\iff f\leqW\PC_{2^\IN}$.}\]
We assume that Cantor space $2^\IN$ is equipped with the usual uniform measure.
Intuitively, the above characterization means that a function $f$ is Las Vegas computable if and only if it can be computed using the
resource of probabilistic choice $\PC_{2^\IN}$ exactly once during the course of the computation.
We also prove that Weihrauch reducibility $\leqW$ cannot be replaced by strong Weihrauch reducibility $\leqSW$ here (the strong reducibility
is different in that the only information that is available after usage of the oracle is the answer of the oracle; in particular the original
input is not available afterwards). 

In Section~\ref{sec:products} we prove a general Independent Choice Theorem~\ref{thm:products} with the help of the Theorem of Fubini
(that generalizes a corresponding result of the first author, de Brecht and Pauly \cite{BBP12} on non-deterministic computations).
This theorem implies, in particular, that Las Vegas computable functions are closed under composition, i.e.,
\[\mbox{$f$ and $g$ Las Vegas computable $\TO g\circ f$ Las Vegas computable.}\]
Closure under composition is in a certain sense a minimal requirement for a reasonable class of functions
from a ``practical programming'' perspective.

The Independent Choice Theorem~\ref{thm:products} also shows that a similar result cannot just be obtained for randomized
computations with positive probabilities,
but more generally whenever the intervals $I$ of probabilities are closed under multiplication. In Section~\ref{sec:intervals}
we show that over the probability space $R=2^\IN$ we can only get three classes of functions in this way:
the computable functions, the Las Vegas computable functions and the non-deterministically computable functions.\footnote{A multi-valued function
is {\em non-deterministically computable} if it can be computed using an infinite sequence of coin tosses with failure recognition, but without any
further restrictions on the success probability; however, we still require that there has to be at least one successful guess.}

In Section~\ref{sec:PCN} we study the setting of the probability space $R=\IN\times2^\IN$ and we call the functions
below $\PC_{\IN\times2^\IN}$ {\em Las Vegas computable with finitely many mind changes}.
We prove that these functions are exactly those functions that one can obtain if one composes a Las Vegas computable
function $f$ with a function $g$ that is computable with finitely many mind changes in either order.
In particular, the class of functions that are Las Vegas computable with finitely many mind changes is closed under composition as well. 
This result is interesting, since functions that are computable with finitely many mind changes are of independent interest
and have been used for instance in learning theory (see for instance \cite{dBY10}).

In Section~\ref{sec:probability-space} we use the Smith-Volterra Cantor set construction in order to show
that the probability space $R=2^\IN$ can almost (up to some arbitrarily small measure) be replaced by the unit interval $R=[0,1]$, equipped with the Lebesgue measure.
In fact, the corresponding classes are even exactly equivalent if arbitrary positive probabilities are allowed and,
in particular, we obtain
\[\mbox{$f$ is Las Vegas computable $\iff f\leqW\PC_{[0,1]}$.}\]
We also prove that $R=\IN\times2^\IN$ and $R=\IR$ lead exactly to the same classes of probabilistically computable functions,
no matter what kind of intervals are imposed on the probabilities. In particular,
\[\mbox{$f$ is Las Vegas computable with finitely many mind changes $\iff f\leqW\PC_\IR$.}\]

In Section~\ref{sec:WWKL} we collect some definitions and observations regarding Weak Weak K\H{o}nig's Lemma $\WWKL$,
in particular we have $\PC_{2^\IN}\equivW\WWKL$ and thus
\[\mbox{$f$ is Las Vegas computable $\iff f\leqW\WWKL$.}\]
In Section~\ref{sec:jumps} we briefly discuss jumps $\WWKL'$ and discrete jumps $\WWKL^\Delta$ of $\WWKL$
and we show that
\[\mbox{$f$ is Las Vegas computable with finitely many mind changes $\iff f\leqW\WWKL^\Delta.$}\]
The functions $f\leqW\WWKL'$ form an even larger class and they can be seen as probabilistically computable with weaker
failure recognition mechanisms. For all these functions failure of the random advice can, in particular, be recognized ``in the limit''.

In Section~\ref{sec:probability} we study the problem $\varepsilon\dash\WWKL$ that was introduced by Dorais et al.\ \cite{DDH+12}
and that is Weihrauch equivalent to $\P_{(\varepsilon,1]}\C_{2^\IN}$, the probabilistic choice problem for closed sets $A\In2^\IN$ with $\mu(A)>\varepsilon$.
We prove that the lower bounds that are imposed on the probability lead to a strict hierarchy of problems, i.e.,
\[\varepsilon\dash\WWKL\leqW\delta\dash\WWKL\iff\varepsilon\geq\delta.\]
A similar result has independently been proved by Dorais et al.\ \cite[Proposition~4.7]{DDH+12}.

The aforementioned result can be interpreted such that probability amplification fails for Las Vegas computable functions.
Intuitively, this is because we are dealing with infinite computations and even if we perform two randomized computations in parallel
we need to start producing some definite output possibly before we might know that one of the computations fails. 
Using a version of the Lebesgue Density Lemma $\LDL$ we prove in Section~\ref{sec:LDL} that 
probability amplification works for Las Vegas computable functions with finitely many mind changes, i.e.,
\[\mbox{$\PC_{\IN\times2^\IN}\equivW\PC_\IR\equivW\P_{(\varepsilon,\infty]}\C_\IR$ for every $\varepsilon>0$.}\]

This also shows that there is a trade-off between the different aspects of probabilistic computations that we consider: the underlying 
probability space, the imposed probability and the failure recognition mechanism. If we want to achieve a certain guaranteed probability,
then this might be feasible for the price of changing the underlying probability space (for instance from $\PC_{2^\IN}$ to $\PC_{\IN\times2^\IN}$)
or for the price of allowing a weaker failure recognition mechanism (for instance from $\PC_{2^\IN}$ to $\PC_{2^\IN}'$).

As a side result we obtain that the Lebesgue Density Lemma $\LDL$ itself is equivalent to choice on natural numbers, i.e., $\LDL\equivW\C_\IN$.
In Section~\ref{sec:probability-amplification} we briefly discuss an algebraic operation $f+g$ that mimics parallel computations as they occur in the usual probability amplification method. 

In Section~\ref{sec:majority} we prove that single-valued functions $f:X\to Y$ on computable metric spaces $X,Y$,
which are below $\frac{1}{2}\dash\WWKL^{(n)}$ (the $n$--fold jump of $\frac{1}{2}\dash\WWKL$) are always computable, i.e.,
for all $n\in\IN$ we have
\[\mbox{$f\leqW\frac{1}{2}\dash\WWKL^{(n)}\TO f$ computable.}\]
The underlying idea of a ``majority-vote'' is the same that has been used for the classical proof of the Theorem of Sacks~\ref{thm:Sacks},
which we discuss in Section~\ref{sec:probabilistic-degrees}.

More generally, we discuss {\em probabilistic functions} $f$ in Section~\ref{sec:probabilistic-degrees}, which have been called functions that
are computable with {\em random advice} in an earlier study of the first author and Arno Pauly~\cite{BP10}.
Intuitively, a probabilistic $f$ is an $f$ that can be computed with random advice irrespectively of any kind of uniformity or
failure recognition method in this regard. In \cite{BP10} it was already proved that $\WKL$ is not probabilistic.
We show that all functions below $\PC_R^{(n)}$ for $R$ among $\IN,2^\IN,\IN\times2^\IN$ and $\IN^\IN$ are probabilistic.
Hence, a proof that some function is not probabilistic shows that it cannot be computed with any of the mentioned resources.
In particular, we obtain
\[\WKL\nleqW\WWKL^{(n)}\]
for all $n\in\IN$. We prove another result that generalizes the Theorem of Sacks~\ref{thm:Sacks} in a certain sense, namely
for suitable single-valued $f:X\to Y$ we have that
\[\mbox{$f$ probabilistic $\TO f$ maps computable inputs $x$ to computable outputs $f(x)$.}\]

In the remaining sections of the paper we present case studies in which we investigate certain computational problems with regards to the question
of whether they admit a Las Vegas algorithm.
In Section~\ref{sec:IVT} we start with the problem of finding a zero $x\in[0,1]$ of a continuous function
$f:[0,1]\to\IR$ that changes its sign, i.e., $f(0)\cdot f(1)<0$. This problem is also known as the Intermediate Value Theorem $\IVT$. 
Using the finite extension method we prove that there is no Las Vegas algorithm for this problem (not even with finitely many mind changes), 
however it is probabilistic
and even admits a uniform probabilistic algorithm with a weaker failure recognition mechanism in the sense that
$\IVT\leqW\WWKL'$. Altogether, we obtain
\[\IVT\nleqW\WWKL\mbox{ and }\IVT\leqW\WWKL'.\]

The second case study concerns the computation of Nash equilibria $\NASH$. This problem was studied by Arno Pauly \cite{Pau10,Pau11}
who proved that it is equivalent to the idempotent closure $\RDIV^*$ of robust division $\RDIV$, i.e., $\NASH\equivW\RDIV^*$.
Intuitively, robust divisions can be used to solve linear equations (and inequalities) in a compact domain and by using this
operation repeatedly, one can determine Nash equilibria. 
In Section~\ref{sec:RDIV} we first prove that there is a Las Vegas algorithm for robust division
and since $\WWKL$ is idempotent (which means $\WWKL^*\leqW\WWKL$) we can conclude in Section~\ref{sec:NASH} that 
\[\NASH\leqW\WWKL.\]
This implies that there is a Las Vegas algorithm to compute Nash equilibria.
As a side result we prove that robust division (and hence Nash equilibria $\NASH$) cannot be computed with any fixed
positive success probability and that Nash equilibria cannot be reduced to the Intermediate Value Theorem, i.e.,
\[\NASH\nleqW\IVT.\]
Hence, the two problems $\IVT$ and $\NASH$ are incomparable by the above results.
The exact relation between most of the problems studied in this paper is shown in the diagram in Figure~\ref{fig:diagram}
in the concluding Section~\ref{sec:conclusions}.
 
Altogether, our case study proves that there are problems of practical interest that are Las Vegas computable,
but not computable (such as Nash equilibria) and that there are other problems of practical interest
which are probabilistically and non-deterministically computable, but not Las Vegas computable, such 
as the problem of finding zeros of continuous functions with sign changes. 
The latter problem also illustrates that there are problems of practical interest, which admit
randomized algorithms with weaker failure recognition methods (here the failure is recognizable only in the limit),
but not Las Vegas algorithms. This proves that the distinctions that we have made are meaningful and
can be illustrated with problems of practical importance.

\section{Preliminaries}
\label{sec:preliminaries}

In this section we give a brief introduction into the Weihrauch lattice and we provide some basic notions from
probability theory. 

\subsection*{Pairing Functions}

We are going to use some standard pairing functions in the following that we briefly summarize.
As usual, we denote by 
$\langle n,k\rangle :=\frac{1}{2}(n+k+1)(n+k)+k$
the Cantor pair of two natural numbers $n,k\in\IN$ and by 
$\langle p,q\rangle(n):=p(k)$ if $n=2k$ and $\langle p,q\rangle(n)=q(k)$, if $n=2k+1$,
the pairing of two sequences $p,q\in\IN^\IN$. By 
$\langle k,p\rangle(n):=kp$ 
we denote the obvious pairing of a number $k\in\IN$ with a sequence $p\in\IN^\IN$.
We also define a pairing function
$\langle p_0,p_1\rangle:=\langle\langle p_0(0),p_1(0)\rangle,\langle \overline{p_0},\overline{p_1}\rangle\rangle$,
for $p_0,p_1\in\IN\times2^\IN$, where $\overline{p_i}(n)=p_i(n+1)$.
Finally, we use the pairing function
$\langle p_0,p_1,p_2,...\rangle\langle i,j\rangle:=p_i(j)$
for sequences $(p_i)_i$ in $\IN^\IN$.

\subsection*{The Weihrauch Lattice}

The original definition of Weihrauch reducibility is due to Klaus Weihrauch
and has been studied for many years (see \cite{Ste89,Wei92a,Wei92c,Her96,Bra99,Bra05}).
More recently it has been noticed that a certain variant of this reducibility yields
a lattice that is very suitable for the classification of the computational content of mathematical theorems
(see  \cite{GM09,Pau10,Pau10a,BG11,BG11a,BBP12,BGM12}). The basic reference for all notions
from computable analysis is Weihrauch's textbook \cite{Wei00}.
The Weihrauch lattice is a lattice of multi-valued functions on represented
spaces. 

A {\em representation} $\delta$ of a set $X$ is just a surjective partial
map $\delta:\In\IN^\IN\to X$. In this situation we call $(X,\delta)$ a {\em represented space}.
By $\IN:=\{0,1,2,...\}$ we denote the set of {\em natural numbers}.
In general we use the symbol ``$\In$'' in order to indicate that a function is potentially partial.
We work with partial multi-valued functions $f:\In X\mto Y$ where $f(x)\In Y$ denotes the set of possible
values upon input $x\in\dom(f)$. If $f$ is single-valued, then for the sake of simplicity we identify $f(x)$ with the single element $y$ in it.
We denote the {\em composition} of
two (multi-valued) functions $f:\In X\mto Y$ and $g:\In Y\mto Z$ either by $g\circ f$ or by $gf$.
It is defined by 
\[g\circ f(x):=\{z\in Z:(\exists y\in Y)(z\in g(y)\mbox{ and }y\in f(x))\},\]
where $\dom(g\circ f):=\{x\in X:f(x)\In\dom(g)\}$.
Using represented spaces we can define the concept of a realizer. 

\begin{definition}[Realizer]
Let $f : \In (X, \delta_X) \mto (Y, \delta_Y)$ be a multi-valued function on represented spaces.
A function $F:\In\IN^\IN\to\IN^\IN$ is called a {\em realizer} of $f$, in symbols $F\vdash f$, if
$\delta_YF(p)\in f\delta_X(p)$ for all $p\in\dom(f\delta_X)$.
\end{definition}

Realizers allow us to transfer the notions of computability
and continuity and other notions available for Baire space to any represented space;
a function between represented spaces will be called {\em computable}, if it has a computable realizer, etc.
Now we can define Weihrauch reducibility.

\begin{definition}[Weihrauch reducibility]
Let $f,g$ be multi-valued functions on represented spaces. 
Then $f$ is said to be {\em Weihrauch reducible} to $g$, in symbols $f\leqW g$, if there are computable
functions $K,H:\In\IN^\IN\to\IN^\IN$ such that $H\langle \id, GK \rangle \vdash f$ for all $G \vdash g$.
Moreover, $f$ is said to be {\em strongly Weihrauch reducible} to $g$, in symbols $f\leqSW g$,
if an analogous condition holds, but with the property $HGK\vdash f$ in place of  $H\langle \id, GK \rangle \vdash f$.
\end{definition}

We can always tacitly assume that $K, H$ are defined on the minimal necessary domains that consist of those names that are actually involved. More precisely, 
let $f:\In X\mto Y$ and $g:\In W\mto Z$ be multi-valued functions on represented spaces $(X,\delta_X)$, $(Y,\delta_Y)$, $(W,\delta_W)$ and $(Z,\delta_Z)$,
and let $f\leqW g$ hold according to the above definition;
then we say that $H$ and $K$ have {\em minimal domains} if 
\begin{enumerate}
\item $\dom(K)=\{p\in\IN^\IN:\delta_X(p)\in\dom(f)\}$ and 
\item $\dom(H)=\{\langle p,q\rangle:p\in\dom(K)$ and $\delta_Z(q)\in g(\delta_W(K(p)))\}$. 
\end{enumerate}
We use an analogous terminology in case of $f\leqSW g$.

The difference between ordinary and strong Weihrauch reducibility is that the ``output modificator'' $H$ has
direct access to the original input in case of ordinary Weihrauch reducibility, but not in case of strong Weihrauch reducibility. 
There are algebraic and other reasons to consider ordinary Weihrauch reducibility as the more natural variant. 
For instance, one can characterize the reduction $f\leqW g$ as follows: $f\leqW g$ holds if and only if a Turing machine can compute $f$ in such a way that
it evaluates the ``oracle'' $g$ exactly on one (usually infinite) input during the course of its computation (see \cite[Theorem~7.2]{RTW11}).
We will use the strong variant $\leqSW$ of Weihrauch reducibility mostly for technical purposes, for instance
it is better suited to study jumps (see below).

We note that the relations $\leqW$, $\leqSW$ and $\vdash$ implicitly refer to the underlying representations, which
we will only mention explicitly if necessary. It is known that these relations only depend on the underlying equivalence
classes of representations, but not on the specific representatives (see Lemma~2.11 in \cite{BG11}).
The relations $\leqW$ and $\leqSW$ are reflexive and transitive, thus they induce corresponding partial orders on the sets of 
their equivalence classes (which we refer to as {\em Weihrauch degrees} or {\em strong Weihrauch degrees}, respectively).
These partial orders will be denoted by $\leqW$ and $\leqSW$ as well. The induced lattice and semi-lattice, respectively, are distributive
(for~details~see~\cite{Pau10a}~and~\cite{BG11}).
We use $\equivW$ and $\equivSW$ to denote the respective equivalences regarding $\leqW$ and $\leqSW$, 
by $\lW$ and $\lSW$ we denote strict reducibility and by $\nW,\nSW$ we denote incomparability in the respective sense.

\subsection*{The Algebraic Structure}

The partially ordered structures induced by the two variants of Weihrauch reducibility are equipped with a number of useful algebraic operations that we summarize in the next definition.
We use $X\times Y$ to denote the ordinary set-theoretic {\em product}, $X\sqcup Y:=(\{0\}\times X)\cup(\{1\}\times Y)$ in order
to denote {\em disjoint sums} or {\em coproducts}, by $\bigsqcup_{i=0}^\infty X_i:=\bigcup_{i=0}^\infty(\{i\}\times X_i)$ we denote the 
{\em infinite coproduct}. By $X^i$ we denote the $i$--fold product of a set $X$ with itself, where $X^0=\{()\}$ is some canonical singleton (i.e., we identify $()$ with the empty word $\varepsilon$).
By $X^*:=\bigsqcup_{i=0}^\infty X^i$ we denote the set of all {\em finite sequences over $X$}
and by $X^\IN$ the set of all {\em infinite sequences over $X$}. 
All these constructions have parallel canonical constructions on representations and the corresponding representations
are denoted by $[\delta_X,\delta_Y]$ for the product of $(X,\delta_X)$ and $(Y,\delta_Y)$, 
by $\delta_X^n$ for the $n$--fold product of $(X,\delta_X)$ with itself, where $n\in\IN$ and $\delta_X^0$ is a representation of the one-point set $\{()\}=\{\varepsilon\}$.
By $\delta_X\sqcup\delta_Y$ we denote the representation of the coproduct, by $\delta^*_X$ the representation of $X^*$ and by $\delta_X^\IN$ the representation
of $X^\IN$. For instance, $(\delta_X\sqcup\delta_Y)$ can be defined by $(\delta_X\sqcup\delta_Y)\langle n,p\rangle:=(0,\delta_X(p))$
if $n=0$ and $(\delta_X\sqcup\delta_Y)\langle n,p\rangle:=(1,\delta_Y(p))$, otherwise.
Likewise, $\delta^*_X\langle n,p\rangle:=(n,\delta_X^n(p))$.
See \cite{Wei00} or \cite{BG11,Pau10a,BBP12} for details of the definitions of the other representations. We will always assume that these canonical representations
are used, if not mentioned otherwise. 

\begin{definition}[Algebraic operations]
\label{def:algebraic-operations}
Let $f:\In X\mto Y$ and $g:\In Z\mto W$ be multi-valued functions. Then we define
the following operations:
\begin{enumerate}
\itemsep 0.2cm
\item $f\times g:\In X\times Z\mto Y\times W, (f\times g)(x,z):=f(x)\times g(z)$ \hfill (product)
\item $f\sqcap g:\In X\times Z\mto Y\sqcup W, (f\sqcap g)(x,z):=f(x)\sqcup g(z)$ \hfill (sum)
\item $f\sqcup g:\In X\sqcup Z\mto Y\sqcup W$, with $(f\sqcup g)(0,x):=\{0\}\times f(x)$ and\\
        $(f\sqcup g)(1,z):=\{1\}\times g(z)$ \hfill (coproduct)
\item $f^*:\In X^*\mto Y^*,f^*(i,x):=\{i\}\times f^i(x)$ \hfill (finite parallelization)
\item $\widehat{f}:\In X^\IN\mto Y^\IN,\widehat{f}(x_n):=\bigtimes\limits_{i\in\IN} f(x_i)$ \hfill (parallelization)
\end{enumerate}
\end{definition}

In this definition and in general we denote by $f^i:\In X^i\mto Y^i$ the $i$--th fold product
of the multi-valued map $f$ with itself ($f^0$ is the constant function on the canonical singleton).
It is known that $f\sqcap g$ is the {\em infimum} of $f$ and $g$ with respect to strong as well as
ordinary Weihrauch reducibility (see \cite{BG11}, where this operation was denoted by $f\oplus g$).
Correspondingly, $f\sqcup g$ is known to be the {\em supremum} of $f$ and $g$ with respect to ordinary Weihrauch reducibility $\leqW$ (see \cite{Pau10a}).
This turns the partially ordered structure of Weihrauch degrees (induced by $\leqW$) into a lattice,
which we call the {\em Weihrauch lattice}.
The two operations $f\mapsto\widehat{f}$ and $f\mapsto f^*$ are known to be closure operators
in this lattice (see \cite{BG11,Pau10a}).

There is some useful terminology related to these algebraic operations. 
We say that $f$ is a {\em a cylinder} if $f\equivSW\id\times f$ where $\id:\Baire\to\Baire$ always
denotes the identity on Baire space, if not mentioned otherwise. 
Cylinders $f$ have the property that $g\leqW f$ is equivalent to $g\leqSW f$ (see \cite{BG11}).
We say that $f$ is {\em idempotent} if $f\equivW f\times f$ and {\em strongly idempotent}, if $f\equivSW f\times f$. 
We say that a multi-valued function on represented spaces is {\em pointed}, if it has a computable
point in its domain. For pointed $f$ and $g$ we obtain $f\sqcup g\leqSW f\times g$. 
The properties of pointedness and idempotency are both preserved under
equivalence and hence they can be considered as properties of the respective degrees.
For a pointed $f$ the finite prallelization $f^*$ can also be considered as {\em idempotent closure} since one
can easily show that idempotency is equivalent to $f\equivW f^*$ in this case.
We call $f$ {\em parallelizable} if $f\equivW\widehat{f}$ and it is easy to see that $\widehat{f}$ is always idempotent.
Analogously, we call $f$ {\em strongly parallelizable} if $f\equivSW\widehat{f}$.

More generally, we define {\em countable coproducts} $\bigsqcup_{i\in\IN} f_i:\In\bigsqcup_{i\in\IN} X_i\mto\bigsqcup_{i\in\IN} Y_i$ 
for a sequence $(f_i)$ of multi-valued functions $f_i:\In X_i\mto Y_i$ on represented spaces and then it denotes
the operation given by $(\bigsqcup_{i\in\IN} f_i)(i,u):=\{i\}\times f_i(u)$. 
Using this notation we obtain $f^*=\bigsqcup_{i\in\IN} f^i$. 
In \cite{BBP12} a multi-valued function on represented spaces has been called 
{\em join-irreducible} if $f\equivW\bigsqcup_{n\in\IN}f_n$ implies that there
is some $n$ such that $f\equivW f_n$. Analogously, we can define {\em strong join-irreducibility}
using strong Weihrauch reducibility in both instances. 
We can also define a {\em countable sum} $\bigsqcap_{i\in\IN} f_i:\In \bigtimes_{i\in\IN} X_i\mto\bigsqcup_{i\in\IN} Y_i$,
defined by $\left(\bigsqcap_{i\in\IN} f_i\right)(x_i)_i:=\bigsqcup_{i\in\IN} f_i(x_i)$.

One should note however, that $\bigsqcap$ and $\bigsqcup$ do not provide infima and
suprema of sequences. By a result of Higuchi and Pauly \cite[Proposition~3.15]{HP13} the Weihrauch lattice has no non-trivial suprema (i.e., a sequence
$(s_n)$ has a supremum $s$ if and only if $s$ is already the supremum of a finite prefix of the sequence $(s_n)_n$)
and likewise by \cite[Corollary~3.18]{HP13} the pointed Weihrauch degrees do not have non-trivial infima. 
In particular, the Weihrauch lattice is not complete.\footnote{We note, however, that for the continuous variant of Weihrauch reducibility
the objects $\bigsqcup_{n\in\IN}f_n$ and $\bigsqcap_{n\in\IN}f_n$ are suprema and infima of the sequence $(f_n)_n$, respectively, and the corresponding
continuous version of the Weihrauch lattice is actually countably complete (see \cite{HP13}).}

\subsection*{Compositional Products}

While the Weihrauch lattice is not completed, some suprema and some infima exist in general.
The following result was proved by the first author and Pauly in \cite{BP14a} and ensures the existence of an important maximum.

\begin{proposition}[Compositional products]
Let $f,g$ be multi-valued functions on represented spaces. Then the following Weihrauch degree exists:
\begin{itemize}
\item[] $f *g:=\max\{f_0\circ g_0:f_0\leqW f\mbox{ and }g_0\leqW g\}$ \hfill (compositional product)
\end{itemize}
\end{proposition}

Here $f*g$ is defined over all $f_0\leqW f$ and $g_0\leqW g$ which can actually be composed (i.e., the target space of $g_0$ and the source space of $f_0$ have to coincide).
In this way $f*g$ characterizes the most complicated Weihrauch degree that can be obtained by first performing a computation with the help of $g$ and then another one with the help of $f$.
Since $f*g$ is a maximum in the Weihrauch lattice, we can consider $f*g$ as some fixed representative of the corresponding degree.
It is easy to see that $f\times g\leqW f*g$ holds. 
We can also define the {\em strong compositional product} by 
\[f\stars g:=\sup\{f_0\circ g_0:f_0\leqSW f\mbox{ and }g_0\leqSW g\}\]
(but we neither claim that it exists in general nor that it is a maximum). 
The compositional products were originally introduced in \cite{BGM12}.

\subsection*{Jumps}

In \cite{BGM12} the first two authors and Marcone introduced {\em jumps} or {\em derivatives} $f'$ of multi-valued functions $f$ on represented spaces.
We recall that the {\em jump} $f':\In (X,\delta_X')\mto (Y,\delta_Y)$ of a multi-valued function $f:\In (X,\delta_X)\mto (Y,\delta_Y)$ on represented
spaces is obtained by replacing the input representation $\delta_X$ by its jump $\delta'_X:=\delta_X\circ\lim$. This leads to $f'\equivSW f\stars\lim$
(see \cite[Corollary~5.16]{BGM12}). By $f^{(n)}$ we denote the $n$--fold jump. 
Here
\[\lim:\In\IN^\IN\to\IN^\IN,\langle p_0,p_1,p_2,...\rangle\mapsto\lim_{n\to\infty}p_n\] 
is the limit operation on Baire space $\IN^\IN$ with respect to the product topology on $\IN^\IN$. 
Hence, a $\delta_X'$--name $p$ of a point $x\in X$ is a sequence that converges to a $\delta_X$--name of $x$.
This means that a $\delta_X'$--name typically contains significantly less accessible information on $x$ than a $\delta_X$--name. 
Hence, $f'$ is typically harder to compute than $f$, since less input information is available for $f'$.

The jump operation $f\mapsto f'$ plays a similar role in the Weihrauch lattice as the Turing jump operation
does in the Turing semi-lattice. In a certain sense $f'$ is a version of $f$ on the ``next higher'' level of complexity
(which can be made precise using the Borel hierarchy \cite{BGM12}).
It was proved in \cite{BGM12} that the jump operation $f\mapsto f'$ is monotone with respect to strong Weihrauch 
reducibility $\leqSW$, but not with respect to ordinary Weihrauch reducibility $\leqW$. This is one reason
why it is beneficial to extend the study of the Weihrauch lattice to strong Weihrauch reducibility.

\subsection*{Closed Choice}

A particularly useful multi-valued function in the Weihrauch lattice is closed choice (see \cite{GM09,BG11,BG11a,BBP12})
and it is known that many notions of computability can be calibrated using the right version of choice. 
We recall that a subset $U\In X$ of a represented space $X$ is open with respect to the final topology of its representation,\footnote{If $\delta_X$ is the representation
of $X$, then $\{U\In X:\delta_X^{-1}(U)$ open in $\dom(\delta_X)\}$ is called the {\em final topology of $\delta_X$}.}
if and only if its characteristic function
\[\chi_U:X\to\IS,x\mapsto\left\{\begin{array}{ll}
  1 & \mbox{if $x\in U$}\\
  0 & \mbox{otherwise}
\end{array}\right.\]
is continuous, where $\IS=\{0,1\}$ is Sierpi{\'n}ski space (equipped with the topology $\{\emptyset,\{1\},\IS\}$).
Analogously, $U$ is {\em c.e.\ open} if $\chi_U$ is computable, where $\IS$ is equipped with its Standard representation $\delta_\IS$
defined by $\delta_\IS(p):=0$ if $p(n)=0$ for all $n$ and $\delta_\IS(p):=1$, otherwise.
Closed and {\em co-c.e.\ closed} sets $A\In X$ are sets whose complement $U:=X\setminus U$ is open and c.e.\ open, respectively (see \cite{BBP12,BP03}
for more details).
For subsets $A\In\IN$ of natural numbers this leads to the usual notion of c.e.\ and co-c.e.\ sets.
The co-c.e.\ closed subsets $A\In2^\IN$ of Cantor space are exactly the usual $\pO{1}$--classes.

In general, if $(X,\delta_X)$ is a represented space then we always assume that this space is endowed with the final topology
of its representation.
We denote by $\AA_-(X)$ the set of closed subsets of $X$
represented with respect to negative information. More precisely, we can define a representation $\psi_-$ of $\AA_-(X)$ by
\[\psi_-(p)=A:\iff[\delta_X\to\delta_\IS](p)=\chi_{X\setminus A},\]
where $[\delta_X\to\delta_\IS]$ is the canonical function space representation in the category of represented spaces (see \cite{Wei00}).
This means that a $\psi_-$--name $p$ of a closed set $A\In X$ is a name for the characteristic function $\chi_{X\setminus A}$ of its complement.

We are mostly interested in closed choice for computable metric spaces $X$, which are separable metric spaces
such that the distance function is computable on the given dense subset.
We assume that computable metric spaces are represented via their Cauchy representations
(see \cite{Wei00} for details). In this special case a computably equivalent definition of $\psi_-$ can be obtained by
\[\psi_-(p):=X\setminus\bigcup_{i=0}^\infty B_{p(i)},\]
where $B_n$ is some standard enumeration of the open balls of $X$ with center in the dense subset and rational radius.
Here a $\psi_-$--name $p$ of a closed set $A\In X$ is a list of sufficiently many open rational balls whose union exhausts exactly the complement of $A$.
We are now prepared to define closed choice.

\begin{definition}[Closed Choice]
Let $X$ be a represented space. The {\em closed choice} problem 
of the space $X$ is defined by
\[\C_X:\In\AA_-(X)\mto X,A\mapsto A\]
with $\dom(\C_X):=\{A\in\AA_-(X):A\not=\emptyset\}$.
\end{definition}

Intuitively, $\C_X$ takes as input a non-empty closed set in negative description (i.e., given by $\psi_-$) 
and it produces an arbitrary point of this set as output.
Hence, $A\mapsto A$ means that the multi-valued map $\C_X$ maps
the input $A\in\AA_-(X)$ to the set $A\In X$ as a set of possible outputs.
We mention some classes of functions that can be characterized
by closed choice. The following results have mostly been proved in \cite{BBP12}:

\begin{proposition}
\label{prop:classes}
Let $f$ be a multi-valued function on represented spaces. Then:
\begin{enumerate}
\item $f\leqW\C_1\iff f$ is computable,
\item $f\leqW\C_\IN\iff f$ is computable with finitely many mind changes,
\item $f\leqW\C_{2^\IN}\iff f$ is non-deterministically computable,
\item $f\leqW\C_{\IN^\IN}\iff f$ is effectively Borel measurable.
\end{enumerate}
In the latter case (4) we have to assume that $f:X\to Y$ is single-valued and defined
on computable complete metric spaces $X,Y$.
\end{proposition}

Here and in general we identify each natural number $n\in\IN$ with the corresponding finite subset $n=\{0,1,...,n-1\}$.
The problem $\C_0$, i.e., closed choice for the empty set $0=\emptyset$, is the bottom element of the Weihrauch lattice.
Also $\C_2$ plays a significant role, since it is equivalent to $\LLPO$, the so-called {\em lesser limited principle of omniscience} as it is known from constructive mathematics.
In \cite{BGM12} we have characterized the jumps $\C_X'\equivW\CL_X$ for computable metric spaces $X$ using
the cluster point problem $\CL_X$ of $X$.
We also use the {\em limited principle of omniscience} $\LPO:\IN^\IN\to\{0,1\}$, which is simply the characteristic function
of the constant zero sequence $\widehat{0}\in\IN^\IN$.

\subsection*{Some Measure Theory}

We now introduce some notation from measure theory. 
We consider measures as non-negative functions into $[0,\infty]$. 
On the set $\IN=\{0,1,2,...\}$ we can use the {\em geometric probability measure} $\mu$,
induced by $\mu(\{n\})=2^{-n-1}$ for all $n\in\IN$.
This leads to a product measure $\mu_{\IN^\IN}$ on Baire space $\IN^\IN$ with 
$\mu_{\IN^\IN}(w\IN^\IN)=\prod_{i=0}^{|w|-1}2^{-w(i)-1}$ for all words $w\in\IN^*$.
Similarly, we use the uniform measure $\mu_{2}$ on $2=\{0,1\}$ with $\mu_{2}(\{i\})=\frac{1}{2}$
for $i\in\{0,1\}$ and the induced product measure $\mu_{2^\IN}$ on Cantor space $2^\IN$
with $\mu_{2^\IN}(w2^\IN)=2^{-|w|}$ for all words $w\in2^*=\{0,1\}^*$.
Often, we rather use the {\em counting measure} $\mu_\IN$ on $\IN$ that is induced by
$\mu_\IN(\{n\})=1$ for all $n\in\IN$. If not mentioned otherwise, we assume that $\IN$ is endowed
with the counting measure.
On $\IR$ and $[0,1]$ we use the Lebesgue measure~$\lambda$.
We recall that a measure $\mu$ on $X$ is called {\em finite} if $\mu(X)<\infty$ and {\em $\sigma$--finite}
if there exists a sequence $(X_n)_n$ of measurable sets $X_n\In X$ with $X=\bigcup_{n=0}^\infty X_n$ and $\mu(X_n)<\infty$ for all $n\in\IN$.
All the measures mentioned here are $\sigma$--finite, in fact all except the counting measure
on $\IN$ and the Lebesgue measure on $\IR$ are even probability measures.
We only use Borel measures, i.e., measures for which exactly all sets in the Borel $\sigma$--algebra
generated by the underlying topology are measurable sets. 
This is because we want to assume that all closed sets are measurable.
If not mentioned otherwise, we always assume in the following that $\IN$ is endowed with the counting measure $\mu_\IN$
and $2^\IN,\IN^\IN$ are endowed with the product measures $\mu_{2^\IN},\mu_{\IN^\IN}$ as standard measures, respectively. 
Likewise, $\IR$ and $[0,1]$ are always endowed with the Lebesgue measure $\lambda$.
If we just write $\mu$ for a measure on one of these spaces, then this refers to the corresponding standard measure. 

Given two $\sigma$--finite measures $\mu_X$ on $X$ and $\mu_Y$ on $Y$,
we obtain a unique {\em product measure} $\mu_X\otimes\mu_Y$ on $X\times Y$
with respect to the corresponding product $\sigma$--algebra and this measure is $\sigma$--finite too (see \cite[Theorem~23.3]{Bau01a}).
The product measure $\mu_X\otimes\mu_Y$ satisfies
\[(\mu_X\otimes\mu_Y)(A\times B)=\mu_X(A)\cdot\mu_Y(B)\]
for all corresponding measurable sets $A\In X$ and $B\In Y$.
In the following we assume that we always use this product measure on product spaces.
For instance, $\IN\times2^\IN$ is endowed with the measure
$\mu_\IN\otimes\mu_{2^\IN}$ and so forth.

A basic computability theoretic property of all measures used here is that they
are upper semi-computable on closed sets.

\begin{lemma}[Semi-computability of measures]
\label{lem:semi-computable}
The measures $\mu_{2^\IN}:\AA_-(2^\IN)\to\IR$ and $\lambda:\AA_-([0,1])\to\IR$ are upper-semi computable.
\end{lemma}

We will also occasionally use the fact that two measures $\mu_1,\mu_2$ on $2^\IN$ coincide on closed sets
if they coincide on all open balls $w2^\IN$.

\begin{lemma}[Identity of measures]
\label{lem:identity}
Let $\mu_1,\mu_2$ be two Borel measures on a subspace $R\In\IN^\IN$ of Baire space, at least one of which is finite. 
If
\[\mu_1(w\IN^\IN\cap R)=\mu_2(w\IN^\IN\cap R)\]
for all $w\in\IN^*$, then $\mu_1=\mu_2$.
\end{lemma}
\begin{proof}
Firstly, if one of the involved measures is finite, then the other is finite too, since $\mu_1(R)=\mu_2(R)$.
It follows by $\sigma$--additivity that $\mu_1(U)=\mu_2(U)$ for every open $U\In R$, 
since every such open set $U$ can be written as a disjoint union of balls $w\IN^\IN\cap R$.
Finally, all finite Borel measures on Polish spaces are outer regular (see \cite[Lemma~26.2]{Bau01a}) and hence $\mu_1,\mu_2$ even coincide completely
under the given conditions. 
\end{proof}

By an {\em interval} $I$ we mean any interval of real numbers
with open or closed end points and including $\infty$ as a possible right end point.
An interval $[a,\infty]$ can be used to accommodate a measure that is infinite.
We note that $[\infty,\infty]$ is not considered as an interval here, but singletons $[a,a]=\{a\}$ for $a\in\IR$ are allowed.
If we want to exclude the case of the closed right end point $\infty$, then we speak about an interval $I\In\IR$.
If we want to exclude the case of the open right end point $\infty$, then we speak about intervals $I$ without the open endpoint $\infty$.
We will need the following statement on the monotonicity of Lebesgue integrals. 
We adopt the usual convention in measure theory (see \cite{Bau01a}) that
\[0\cdot\infty=0 \mbox{ and } x\cdot\infty=\infty\mbox{ for $x>0$}.\]
We also assume that these products commute and we can set similar conventions for negative numbers (that we are not going to use).

\begin{lemma}[Monotonicity]
\label{lem:monotonicity}
Let $X$ be a topological space with a $\sigma$--finite Borel measure $\mu$ and let $I$ be a non-empty interval that does not have $\infty$ as an open endpoint.
Let $A\In X$ be measurable and let $f:X\to\IR$ be a non-negative measurable function such that $f(x)\in I$ for all $x\in A$.
Then
\[\int_Af\,{\rm d}\mu\in\mu(A)\cdot I.\]
\end{lemma}
\begin{proof}
If $\mu(A)=0$, then $\mu(A)\cdot I=\{0\}$ since $I$ is non-empty and $\int_Af\,{\rm d}\mu=0$. 
Let now $\mu(A)>0$ and $a,b\in\IR$. Then we obtain
\begin{enumerate}
\item $a\leq f(x)$ for all $x\in A \TO \mu(A)a\leq\int_Af\,{\rm d}\mu$,
\item $f(x)\leq b$ for all $x\in A \TO \int_Af\,{\rm d}\mu\leq \mu(A)b$,
\item $a< f(x)$ for all $x\in A$ $\TO \mu(A)a<\int_Af\,{\rm d}\mu$,
\item $f(x)< b$ for all $x\in A$ $\TO \int_Af\,{\rm d}\mu< \mu(A)b$.
\end{enumerate}
Here the first two properties (1) and (2) are just consequences of the monotonicity of the Lebesgue integral (see \cite[Theorem~12.4]{Bau01a}),
whereas the other two strong monotonicity properties (3) and (4) follow from the first two properties together with the
following observation: for any constant $c\in\IR$ we have that
$\int_Af\,{\rm d}\mu-\mu(A)c=\int_Af-c\,{\rm d}\mu=0$ implies $f=c$ almost everywhere by \cite[Theorem~13.2]{Bau01a}
and hence $f(x)=c$ for some $x\in A$ since $\mu(A)>0$.
Altogether, suitable combinations of the above statements prove the claim for all bounded intervals $I$.
It clearly also holds if $I$ is of the form $I=(a,\infty]$ or $I=[a,\infty]$.
\end{proof}

We mention that the above result cannot be extended to the cases $(a,\infty)$ or $[a,\infty)$, 
since for instance $\int_{(0,1]}\frac{1}{x}\,{\rm d}x=\infty$, even though $\frac{1}{x}<\infty$ for all $x\in(0,1]$.

\section{Las Vegas Computability and Probabilistic Choice}
\label{sec:Las-Vegas}

In this section we would like to formalize the notion of a Las Vegas computable multi-valued function $f:\In X\mto Y$ as it was 
intuitively described in the introduction and we will show that this notion can be characterized
by a suitably defined probabilistic choice operation. Since we do not want to formalize probabilistic Turing
machines on infinite sequences in a technical way here, we will just use the notion of an ordinary computable function
$F:\In\IN^\IN\to\IN^\IN$ in order to introduce our concept of randomized computations. 
Essentially, we will use two such functions $F_1,F_2$, which
play the following roles:
\begin{enumerate}
\item $F_2$ is a {\em failure recognition function} that takes a name $p$ of the input $x\in X$ and a ``random advice'' $r\in R$
         and indicates whether $r$ is successful on input $p$. Here $\delta_\IS F_2\langle p,r\rangle=0$ indicates
         success, i.e., the set
         \[S_p:=\{r\in R:\delta_\IS F_2\langle p,r\rangle=0\}\] 
         (which is closed in $R$) is the {\em set of successful advices} on input $p$. 
\item $F_1$ is the {\em computation function}, i.e., it also takes a name $p$ of the input and the ``random advice'' $r$ 
         and it actually computes the correct result for $f:\In (X,\delta_X)\mto (Y,\delta_Y)$
         in the sense that 
         \[\delta_YF_1\langle p,r\rangle\in f\delta_X(p)\]
         for all successful advices $r\in S_p$.
\end{enumerate}

By the nature of Sierpi\'nski space $(\IS,\delta_\IS)$ the failure event $\delta_\IS F_2\langle p,r\rangle=1$ is the one that can be recognized
in finite time, while success $\delta_\IS F_2\langle p,r\rangle=0$ just means absence of failure in the long run.
If no failure occurs, then $F_1\langle p,r\rangle$ will be a correct result on input $p$ with advice $r$ in the long run.
Having these interpretations in mind we are now prepared to give the formal definition, which is actually a refined version of non-deterministic computability
as defined in \cite{BBP12}.

\begin{definition}[Las Vegas computability]
\label{def:Las-Vegas}
Let $(X,\delta_X)$ and $(Y,\delta_Y)$ be represented spaces, let $R\In\IN^\IN$, let $\mu_R$ be a
Borel measure on $R$ and let $I$ be some interval. A multi-valued function $f:\In X\mto Y$ is said to be 
{\em Las Vegas computable over $R$ with measure in $I$},
if there exist two computable functions $F_1,F_2:\In\IN^\IN\to\IN^\IN$ such that
$\langle \dom(f\delta_X)\times R\rangle\In\dom(F_2)$ and for each $p\in\dom(f\delta_X)$
the following hold:
\begin{enumerate}
\item $S_p:=\{r\in R:\delta_\IS F_2\langle p,r\rangle=0\}$ is non-empty and $\mu_R(S_p)\in I$,
\item $\delta_YF_1\langle p,r\rangle\in f\delta_X(p)$ for all $r\in S_p$.
\end{enumerate}
\end{definition}

If $f$ is Las Vegas computable over $R=2^\IN$ with measure $\mu_{2^\IN}$ and values in $I=(0,1]$, then we say for short that $f$
is {\em Las Vegas computable}. If the same holds over $R=\IN\times2^\IN$ with $\mu_{\IN}\otimes\mu_{2^\IN}$ and $I=(0,\infty]$, then we say for short that $f$
is {\em Las Vegas computable with finitely many mind changes}. The latter terminology will become clearer in 
Section~\ref{sec:PCN} and in particular by Corollary~\ref{cor:PCN2N}.

We mention that Las Vegas computability over $2^\IN$ with probabilities in $I=[0,1]$ is the
same as non-deterministic computability as originally introduced by Martin Ziegler \cite{Zie07,Zie07a}
and further studied in \cite{BBP12}. The above definition is just an adaption of \cite[Definition~7.1]{BBP12}.
In \cite[Theorem~7.2]{BBP12} non-deterministic computations over $R$ were characterized with
the help of the closed choice principle $\C_R$ and here we transfer this result to the probabilistic setting.

We introduce corresponding probabilistic choice principles by generalizing the corresponding definition in \cite{BP10}.
By $\P_I\C_X$ we denote the closed choice operation restricted to closed subsets of $X$
whose measure is in the interval $I$.

\begin{definition}[Probabilistic choice]
Let $X$ be a represented space together with a Borel measure $\mu_X$ on $X$
and let $I$ be an interval.
By 
\[\P_I\C_X:\In\AA_-(X)\mto X,A\mapsto A\] 
we denote the choice operation restricted to $\dom(\P_I\C_X):=\{A:\mu_X(A)\in I\}$.
We call $\P_I\C_X$ {\em probabilistic choice} of $X$ with respect to $I$.
\end{definition}

We usually abbreviate the interval $I$, for instance, by writing ``$>0$'' instead of $(0,\infty]$ and
we use analogous abbreviations for other intervals.
We also write $\PC_X:=\P_{>0}\C_X=\P_{(0,\infty]}\C_X$, which was already studied under the name {\em positive choice} in \cite{BP10}.
We also obtain ordinary closed choice as s special case:
$\C_X=\P_{\geq 0}\C_X=\P_{[0,\infty]}\C_X$.
Theorem~7.2 from \cite{BBP12} can now directly be transferred to our setting.

\begin{theorem}[Las Vegas computability]
\label{thm:probabilistic-computability}
Let $X$ and $Y$ be represented spaces, let $R\In\IN^\IN$ be endowed with a Borel measure, let $I$ be an interval
and let $f:\In X\mto Y$ be a multi-valued function. Then the following are equivalent to each other:
\begin{enumerate}
\item $f\leqW \P_I\C_R$,
\item $f$ is Las Vegas computable over $R$ with measure in $I$.
\end{enumerate}
\end{theorem}

The proof is literally the same as the proof of Theorem~7.2 in \cite{BBP12},
with the extra observation that the success sets $S_p$ in case of Las Vegas computability have to satisfy analogous
measure requirements as the sets in the domain of $\P_I\C_R$.
As a special case of Theorem~\ref{thm:probabilistic-computability} we obtain the following corollary.

\begin{corollary}[Las Vegas computability]
\label{cor:Las-Vegas}
Let $f$ be a multi-valued function on represented spaces. Then
\begin{enumerate}
\item $f\leqW\PC_{2^\IN}$ if and only if $f$ is Las Vegas computable,
\item $f\leqW\PC_{\IN\times2^\IN}$ if and only if $f$ is Las Vegas computable with finitely many mind changes.
\end{enumerate}
\end{corollary}

Theorem~\ref{thm:probabilistic-computability} raises the question of whether there is a difference
between ordinary and strong Weihrauch reducibility to $\P_I\C_R$, in other words, whether
$\P_I\C_R$ is a cylinder. The following result shows that this is typically not the case.
We first introduce a notation.

\begin{definition}[Cardinality]
Let $f:\In X\mto Y$ be a multi-valued function. Then we denote by $\# f$ the supremum 
of the cardinalities $|M|$ of sets $M\In\dom(f)$ such that $\{f(x):x\in M\}$ contains only pairwise disjoint sets.
We call $\# f$ the {\em cardinality} of $f$.
\end{definition}

Obviously, the cardinality $\# f$ is always bounded from above by the cardinality of $\dom(f)$.
It is clear that strong reductions preserve cardinalities in the following sense.

\begin{proposition}[Cardinality]
\label{prop:cardinality}
$f\leqSW g\TO \# f\leq \#g$.
\end{proposition}

It is folklore that for a $\sigma$--finite measure there cannot be an uncountable number of
pairwise disjoint sets of positive measure. For completeness we include the proof.

\begin{proposition}
Let $X$ be a space that is equipped with a $\sigma$--finite measure. Then there can be at most
countably many pairwise disjoint measurable sets $A\In X$ of positive measure.
\end{proposition}
\begin{proof}
Since $X$ is $\sigma$--finite, there is a sequence $(X_i)$ of measurable sets $X_i\In X$ with 
$X=\bigcup_{i=0}^\infty X_i$ and $\mu(X_i)<\infty$ for all $i\in\IN$.
Let $\FF$ be a family of pairwise disjoint sets $A\In X$ of positive measure and let 
\[\FF_{n,k}:=\{A\in\FF:\mu(A\cap X_n)>2^{-k}\}\]
for all $n,k\in\IN$. Then $\FF=\bigcup_{n,k\in\IN}\FF_{n,k}$. There cannot be more than $\mu(X_n)\cdot 2^k$ many
pairwise disjoint sets in $\FF_{n,k}$. Hence any set $\FF_{n,k}$ is finite and hence $\FF$ is a countable union of
finite sets and countable itself.  
\end{proof}

Hence, we obtain the following.

\begin{proposition}[Cardinality of probabilistic choice]
\label{prop:cardPICR}
If $X$ is a represented space that is equipped with a $\sigma$--finite measure and $I$ is an interval
with $0\not\in I$, then $\#\P_I\C_X\leq|\IN|$.
\end{proposition}

Since $\#\id_{\IN^\IN}=|\IN^\IN|$, we get $\id_{\IN^\IN}\nleqSW\P_I\C_X$ in this situation.
This yields the following corollary.

\begin{corollary}[Probabilistic choice is not a cylinder]
\label{cor:cylinder}
If $X$ is a represented space that is equipped with a $\sigma$--finite measure and $I$ is an interval
with $0\not\in I$, then $\P_I\C_X$ is not a cylinder.
\end{corollary}

This means that typically (namely under the conditions given in Corollary~\ref{cor:cylinder}) 
we cannot replace Weihrauch reducibility $\leqW$ by strong Weihrauch reducibility $\leqSW$
in Theorem~\ref{thm:probabilistic-computability}.

\section{Products of Probabilistic Choice}
\label{sec:products}

We now want to compare different probabilistic choice operations with each other
and in particular we want to consider products of probabilistic choice.
As a first obvious observation we note that probabilistic choice is always monotone in the interval of probabilities (or measure values).

\begin{proposition}[Monotonicity]
\label{prop:interval-monotonicity}
Let $I,J$ be intervals and let $X$ be some represented space endowed with some Borel measure.
Then
\[I\In J\TO\P_I\C_X\leqSW\P_J\C_X.\]
\end{proposition}

We mention another obvious result on products.
For two intervals $I,J\In\IR$ we denote by $I\cdot J:=\{x\cdot y:x\in I,y\in J\}$
the {\em arithmetical product} of the two sets. 
It is not too difficult to see that products of non-negative intervals are always intervals (the case $[0,a]\cdot[\infty,\infty]=\{0,\infty\}$ cannot occur
since $[\infty,\infty]$ is not considered as an interval here.)
We say that an interval $I$ is {\em closed under product}, if $I\cdot I\In I$.

\begin{proposition}[Products]
\label{prop:products}
Let $X$ and $Y$ be represented spaces, both endowed with $\sigma$--finite Borel measures. 
Let $I,J$ be intervals. Then we obtain
\[\P_I\C_X\times\P_J\C_Y\leqSW\P_{I\cdot J}\C_{X\times Y}.\]
\end{proposition}

The proof is straightforward, noting that the map 
\[\times:\AA_-(X)\times\AA_-(Y)\to\AA_-(X\times Y),(A,B)\mapsto A\times B\]
is computable and satisfies the property that the measure of the output
is the product of the measures of the inputs.

In \cite{BBP12} an Independent Choice Theorem~7.3 was proved that allows to
conclude that non-deterministically computable functions are closed under composition.
With an additional invocation of the Theorem of Fubini we can transfer this theorem
and its proof to the probabilistic setting.
This theorem can also be seen as a generalization of Proposition~\ref{prop:products} for ordinary Weihrauch reducibility in the
case of $R,S\In\IN^\IN$, since $f\times g\leqW f*g$. 

\begin{theorem}[Independent Choice]
\label{thm:products}
Let $R,S\In\IN^\IN$ both be endowed with $\sigma$--finite Borel measures and let $I,J$ be intervals, such that $\infty$ is not an open endpoint of $I$.
Then
\[\P_I\C_R*\P_J\C_S\leqW \P_{I\cdot J}\C_{R\times S}.\]
\end{theorem}
\begin{proof}
If one of the intervals $I,J$ is empty, then $I\cdot J$ is empty and the claim holds.
Hence, we can assume that $I,J$ are both non-empty.
We consider represented spaces $(X,\delta_X)$, $(Y,\delta_Y)$ and $(Z,\delta_Z)$.
Let now $f:\In Y\mto Z$ and $g:\In X\mto Y$ be Las Vegas computable over $R,S$, respectively
with measures in $I,J$, respectively. Let $\mu_R,\mu_S$ denote the $\sigma$--finite measures of $R,S$, respectively.
Due to Theorem~\ref{thm:probabilistic-computability} it suffices to show that
$f\circ g$ is Las Vegas computable over $R\times S$ with measure in $I\cdot J$.
Intuitively, we can choose an advice $(r,s)\in R\times S$ and use advice $r$
for $f$ and advice $s$ for $g$. More precisely, let $f$ and $g$ be Las Vegas
computable using computable functions $F_1,F_2$ and $G_1,G_2$ according to Definition~\ref{def:Las-Vegas},
respectively.
We define $H_1$ and $H_2$ that witness Las Vegas computability of $f\circ g$
over $R\times S$ with measure in $I\cdot J$.
We can define a computable $H_1$ by
\[H_1\langle p,\langle r,s\rangle\rangle:=F_1\langle G_1\langle p,s\rangle,r\rangle\]
and there exists a computable $H_2$ such that
\[\delta_\IS H_2\langle p,\langle r,s\rangle\rangle
  =\left\{\begin{array}{ll}
  1 & \mbox{if $\delta_\IS G_2\langle p,s\rangle=1$}\\
  \delta_\IS F_2\langle G_1\langle p,s\rangle,r\rangle & \mbox{otherwise}
\end{array}\right.\]
for all $p\in\dom(fg\delta_X)$ and all $(r,s)\in R\times S$.
Such a computable $H_2$ exists, since $\delta_\IS G_2\langle p,s\rangle=0$ implies
that $\delta_Y G_1\langle p,s\rangle\in g(\delta_X(p))\In\dom(f)$.
Now we verify that $H_1$ and $H_2$ satisfy conditions (1) and (2) of
Definition~\ref{def:Las-Vegas} for $f\circ g$.
To this end, let $p\in\dom(fg\delta_X)$.

Let $(r,s)\in R\times S$ be such that $\delta_\IS H_2\langle p,\langle r,s\rangle\rangle=0$.
Then $\delta_\IS G_2\langle p,s\rangle=0$ and $\delta_\IS F_2\langle G_1\langle p,s\rangle,r\rangle=0$.
Hence by conditions (2) for $g$ and $f$ we obtain $\delta_YG_1\langle p,s\rangle\in g\delta_X(p)$
and hence $\delta_ZF_1\langle G_1\langle p,s\rangle,r\rangle\in fg\delta_X(p)$, which proves
condition (2) for $f\circ g$.

It remains to prove that condition (1) holds for $f\circ g$. For this purpose we consider for our fixed $p$
the following sets (which are closed in $S$, $R$, and $R\times S$, respectively):
\begin{itemize}
\item $S_p:=\{s\in S:\delta_\IS G_2\langle p,s\rangle=0\}$,
\item $R_{p,s}:=\{r\in R:\delta_\IS F_2\langle G_1\langle p,s\rangle,r\rangle=0\}$ for all $s\in S_p$,
\item $T_p:=\{(r,s)\in R\times S:\delta_\IS H_2\langle p,\langle r,s\rangle\rangle=0\}$.
\end{itemize}
Intuitively, the set $S_p$ is the set of successful advices for the Las Vegas computation of $g$ on input $p$,
$R_{p,s}$ is the set of successful advices of the Las Vegas computation of $f$ on input $G_2\langle p,s\rangle$,
provided $s\in S_p$ and $T_p$ is the set of all successful advices $(r,s)$ for the Las Vegas computation of $f\circ g$ on input $p$.
By condition (2) for $f$ and $g$ we know that $\mu_S(S_p)\in J$ and $\mu_R(R_{p,s})\in I$ for all $s\in S_p$.
By definition of $H_2$ we obtain
\[T_p=\{(r,s)\in R\times S:s\in S_p\mbox{ and }r\in R_{p,s}\}.\]
By the Theorem of Fubini for measurable sets (see \cite[Theorem~23.3]{Bau01a}) and Lemma~\ref{lem:monotonicity}
this yields
\[(\mu_R\otimes\mu_S)(T_p)=\int_{S_p}\mu_R(R_{p,s})\,{\rm d}\mu_S\in I\cdot\mu_S(S_p) \In I\cdot J,\]
where the integrand is understood to be the function $s\mapsto \mu_R(R_{p,s})$.
This shows that condition (1) also holds for $f\circ g$.
\end{proof}

In order to complete our results on products we need another notion.
A function $f:X\to Y$ on spaces $X,Y$ that are equipped with measures $\mu_X,\mu_Y$,
respectively, is called {\em measure preserving} if $\mu_X(f^{-1}(A))=\mu_Y(A)$ for all closed $A\In Y$.
A function $f:X\to Y$ is called a {\em computable isomorphism} if it is bijective and $f$
as well as $f^{-1}$ are computable. 
We will exploit the fact that the usual Cantor pairing functions (as introduced in Section~\ref{sec:preliminaries})
are computable and measure preserving. 

\begin{lemma}[Pairing functions]
\label{lem:pairing}
The following functions are computable isomorphisms and measure preserving. 
We assume that $\IN$ is endowed with the counting measure.
\begin{enumerate}
\item $2^\IN\times2^\IN\to2^\IN,(p,q)\mapsto\langle p,q\rangle$,
\item $\IN^\IN\times\IN^\IN\to\IN^\IN,(p,q)\mapsto\langle p,q\rangle$,
\item $\IN\times\IN\to\IN,(n,k)\mapsto\langle n,k\rangle$,
\item $(\IN\times2^\IN)\times(\IN\times2^\IN)\to\IN\times2^\IN,(p,q)\mapsto\langle p,q\rangle$.
\end{enumerate}
\end{lemma}

In the special case of spaces $R$ that are equipped with a pairing mechanism that
is a computable isomorphism and measure preserving (as for the spaces given in Lemma~\ref{lem:pairing})
and $I$ is closed under product, 
we can exploit the fact that we obtain
\[\P_I\C_{R\times R}\equivSW\P_I\C_{R}.\]
This yields the following corollary.

\begin{corollary}[Products and pairing]
\label{cor:products-pairing}
If $I$ is an interval that is closed under product and does not contain $\infty$ as an open endpoint
and $R\In\IN^\IN$ is equipped with a $\sigma$--finite measure and a corresponding pairing function that is a computable
isomorphism as well as measure preserving, then we obtain
\[\P_I\C_{R\times R}\equivSW\P_I\C_{R}\equivSW\P_I\C_{R}\times\P_I\C_{R}\equivW\P_I\C_{R}*\P_I\C_{R}.\]
In particular, $\P_I\C_{R}$ is strongly idempotent and closed under composition.
\end{corollary}

We note that by Lemma~\ref{lem:pairing} the assumption on the pairing function applies to all the spaces
$R$ among $\IN,2^\IN,\IN^\IN,\IN\times2^\IN$ with the respective
canonical measures. In particular, we get the following corollary.

\begin{corollary}[Closure under composition]
The classes of multi-valued functions that are Las Vegas computable 
and Las Vegas computable with finitely many mind changes, respectively,
are both closed under composition.
\end{corollary}

\section{Intervals that are Closed under Product}
\label{sec:intervals}

In Corollary~\ref{cor:products-pairing} we have seen that intervals that are closed
under product lead to very natural notions of probabilistic computability, since
the corresponding classes of functions are closed under composition.
For the case of coin tosses, i.e., for the space $2^\IN$, we will see that we 
only obtain three distinct classes in this way: $\C_1,\C_{2^\IN}$ and $\PC_{2^\IN}$.
These classes correspond exactly to the following classes of multi-valued functions: computable,
non-deterministically computable and Las Vegas computable ones, respectively.

We start by considering the types of intervals that are closed under product.
In case that we are using a probability measure, we only have to
consider intervals $I\In[0,1]$ and we can easily see which of those are closed
under product.

\begin{lemma}[Intervals closed under product]
An interval $I\In[0,1]$ is closed under product if and only if
one of the following cases holds:
\begin{enumerate}
\item $I=\{1\}$,
\item $0\in I$,
\item $I=(0,b)$ or $I=(0,b]$ for some $b\in(0,1]$.
\end{enumerate}
\end{lemma}

These three cases lead exactly to the three choice principles $\C_1,\C_{2^\IN}$ and $\PC_{2^\IN}$, respectively. 

\begin{proposition}[Choice for intervals that are closed under product]
Let $I\In[0,1]$ be an interval that is closed under product.
We obtain:
\[\P_I\C_{2^\IN}\equivSW\left\{\begin{array}{ll}
   \C_1 & \mbox{if $I=\{1\}$}\\
   \C_{2^\IN} & \mbox{if $0\in I$}\\
   \PC_{2^\IN} & \mbox{if $I=(0,b)$ or $I=(0,b]$ for some $b\in(0,1]$}
\end{array}\right..\]
\end{proposition}
\begin{proof}
If a closed set $A\In2^\IN$ has full measure $1$, then
it must be identical to the whole space $2^\IN$ and
hence it contains the constant zero sequence that can be computed.
This proves that $\P_{=1}\C_{2^\IN}\equivSW\C_1$.

An arbitrary non-empty closed set $A\In2^\IN$ can
be paired with the constant zero sequence $0^\omega$ to $B=\langle 0^\omega,A\rangle$
and this set has measure $0$. Hence, the computable map $A\mapsto B$ yields
the reduction $\C_{2^\IN}\leqSW\P_I\C_{2^\IN}$ if $0\in I$.
The other direction is obvious.
This proves $\P_I\C_{2^\IN}\equivSW\P_{\geq0}\C_{2^\IN}=\C_{2^\IN}$ if $0\in I$.

Now let $I=(0,b)$ and $J=(0,c)$ with $b,c\in(0,1]$. 
We claim 
\[\P_{(0,b)}\C_{2^\IN}\equivSW\P_{(0,c)}\C_{2^\IN}.\]
If $b\leq c$, then the direction $\leqSW$ is obvious. For the other direction, 
we need to map a given closed set $A\In2^\IN$ with positive measure $\mu(A)<c$
in a computable way to a closed set $B\In2^\IN$ with positive measure $\mu(B)<b\leq c$
such that we can recover a point of the original set $A$ from any point in $B$.
For this purpose we use the map $A\mapsto 0^nA$ that adds a suitable prefix $0^n$ of length $n$
to any point in $A$, where $n$ depends on $\frac{c}{b}$ and guarantees that $0^nA$ has a small
enough measure. This yields the desired reduction. The proof for intervals of type $I=(0,b]$ 
is analogous.  
\end{proof}

We mention that the second case also includes $\P_{=0}\C_{2^\IN}$, i.e., choice for non-empty
closed zero sets. 
While the results in this section show that {\em upper} bounds on the probability do not really lead
to meaningful distinctions, we will see in Section~\ref{sec:probability} that {\em lower} bounds can lead to such distinctions.

\section{Las Vegas Computability with Finitely Many Mind Changes}
\label{sec:PCN}

We recall that we want to call $f$ Las Vegas computable with finitely many mind changes
if $f\leqW\PC_{\IN\times 2^\IN}$. The purpose of this section is to get some further insights
into the $\PC_{\IN\times 2^\IN}$. We start with some comments on $\PC_\IN$.

\begin{lemma}[Probabilistic choice on natural numbers]
\label{lem:PCN}
We obtain
\[\PC_\IN=\P_{[0,\infty]}\C_\IN\equivSW\P_{[0,\infty)}\C_\IN\equivSW\C_\IN,\]
if $\IN$ is equipped with the counting measure or the geometric measure.
\end{lemma}
\begin{proof}
Let $\IN$ be endowed with the counting measure. Then $\PC_\IN=\P_{[0,\infty]}\C_\IN=\C_\IN$
and $\P_{[0,\infty)}\C_\IN$ is closed choice restricted to finite subsets $A\In\IN$.
Hence we get
\[\UC_\IN\leqSW\P_{[0,\infty)}\C_\IN\leqSW\P_{[0,\infty]}\C_\IN=\C_\IN,\]
where $\UC_\IN$ denotes choice for singletons.  In \cite[Proposition~3.8]{BGM12} we have proved
$\UC_\IN\equivSW\C_\IN$ and hence the equivalence follows.
If $\IN$ is endowed with the geometric measure (which is finite), then
$\PC_\IN=\P_{[0,\infty]}\C_\IN=\P_{[0,\infty)}\C_\IN=\C_\IN$.
\end{proof}

It follows from the Independent Choice Theorem~\ref{thm:products} that
\[\C_{\IN}\times\PC_{2^\IN}\leqW\C_{\IN}*\PC_{2^\IN}\leqW\PC_{\IN\times2^\IN}.\]
Here we assume that $\IN$ is endowed with the counting measure. 
We now want to prove that we can also get the inverse reductions in the above situation.

\begin{lemma}
\label{lem:PCN2Na}
$\PC_{\IN\times2^\IN}\leqW\PC_{2^\IN}*\C_\IN$.
\end{lemma}
\begin{proof}
Given a closed set $A\In\IN\times2^\IN$ of positive measure we can compute the sequence $(A_n)$ of sections
\[A_n:=\{p\in2^\IN:(n,p)\in A\}.\]
Since the measure $\mu:\AA_-(2^\IN)\to\IR$ is upper semi-computable by Lemma~\ref{lem:semi-computable}, we obtain that 
\[B:=\{\langle n,k\rangle\in\IN:\mu(A_n)\geq2^{-k}\}\]
is co-c.e.\ in the original set $A$. Since $A$ has positive measure, it follows that $B$ is non-empty
and hence we can use $\C_\IN$ to find a point $\langle n,k\rangle\in B$.
Given $\langle n,k\rangle\in B$ and the original input $A$, we can use $\PC_{2^\IN}$ to find
a point $p\in A_n$. Then $(n,p)\in A$. This proves the claim.
\end{proof}

Next we prove that $\PC_{2^\IN}$ commutes with $\C_\IN$.
Essentially, we exploit for this proof that $\PC_{2^\IN}$ is a fractal (fractals are defined after Theorem~\ref{thm:CN-elimination}).\footnote{Arno Pauly pointed out
that Lemma~\ref{lem:PCN2Nb} holds more generally for suitably defined uniform fractals $f$ in place of $\PC_{2^\IN}$ too.} 

\begin{lemma}
\label{lem:PCN2Nb}
$\PC_{2^\IN}*\C_\IN\leqW\C_\IN\times\PC_{2^\IN}$.
\end{lemma}
\begin{proof}
It suffices to prove that $f\leqW\PC_{2^\IN}*\C_{\IN}$ implies $f\leqW\C_{\IN}\times\PC_{2^\IN}$ for all $f:\In\IN^\IN\mto\IN^\IN$.
Let then $f\leqW\PC_{2^\IN}*\C_\IN$. Then upon input of $p\in\dom(f)$ there is a computation of a machine $M$ with finitely many mind changes
that produces finitely many partial outputs $v_0,...,v_k\in\IN^*$ before it finally produces an infinite output $q\in\IN^\IN$ 
that is a name for a set $A\In 2^\IN$ of positive measure. A name $r$ for a point in $A$ together with $q$ and $p$ finally
allow to compute a point in $f(p)$.
The basic idea is to replace the computation of $M$ with finitely many mind changes by an ordinary computation
that produces $v_0v_1...v_kq$ instead. 
The problem is that the latter sequence might not be a name of a set $A\In2^\IN$ of positive measure.
However, this can be rectified.

Firstly, there is a computable function $g:\IN^*\to\IN^*$ such that $g(v)=v$ if $v$ is a valid prefix of a name
of a set of positive measure and such that $g(v)$ is always a valid prefix of a name of a set of positive measure (this can be achieved
by replacing certain portions of negative information by dummy information). In this sense $g$ ``cleans-up'' the output.
A set of positive measure described by only finitely many open balls in its complement also automatically has a non-empty interior. 
Moreover, we can assume that $g$ is monotone, i.e., that $v\prefix w$ implies $g(v)\prefix g(w)$.
We assume that $(w_n)$ is a bijective effective standard enumeration of $2^*$.
Then there is a computable function $h:\In\IN^*\times\IN\to\IN$ such that $h(v,i)=j$, where $j$ is minimal with the property that $w_j 2^\IN$ 
is left uncovered by the negative information $v$ and $w_i\prefix w_j$, which is possible whenever an open subset of $w_i\IN^\IN$ is left uncovered by $v$.
Hence, $h$ finds an ``unspoiled region'' where the computation can continue.
Finally, there is a computable function $s:\IN^*\times\IN\to\IN^*$ that has the property that if $v$ describes a closed set $A\In2^\IN$,
then $s(v,i)$ describes the set $w_iA$. Hence $s$ ``shifts the output'' to a possibly unspoiled region.
We can also assume that $s$ is monotone in the first component, i.e., $w\prefix v$ implies $s(w,i)\prefix s(v,i)$.

We now describe an algorithm that transfers the original computation of machine $M$ with finitely many mind changes into a regular computation 
of an infinite output together with a sequence $(n_i)$ of natural numbers.
We start with $n_0$ such that $w_{n_0}$ is the empty word. Whenever the output of $M$ is extended to $w$, then we convert it to the ``cleaned-up'' output $g(w)$
and we write the number $n_0$ repeatedly to the output stream of numbers. When a first mind change happens, then the output $u_0:=g(v_0)$ has been produced so far.
In case of this event we compute $n_1:=h(u_0,n_0)$ (i.e., an unspoiled region, which exists since $u_0$ is cleaned-up). 
We continue to read the pieces of information $w$ produced by $M$ after the first mind change and we proceed writing $s(g(w),n_1)$ to the output 
(i.e., a cleaned-up version of the information that follows shifted to the unspoiled region)
and we write $n_1$ repeatedly to the stream of natural numbers until possibly the second mind change happens, in which case we continue inductively. 
In general, in between the $i$--th and the $(i+1)$--st mind change we have produced the output $u_{i}:=s(g(v_{i}),n_{i})$.
When the $(i+1)$--st mind change happens we compute $n_{i+1}:=h(u_i,n_i)$ and we start writing $s(g(w),n_{i+1})$ on the output 
and the number of $n_{i+1}$ into the stream of numbers from now on (for the partial output $w$ of $M$ that follows). 
Eventually (when $i=k+1$) no further mind change happens and we continue with the last step forever, writing $s(g(w),n_{k+1})$ on the output
tape and the number $n_{k+1}$ into the stream of numbers.

In this  phase after the last mind change $w$ will consist of longer and longer prefixes of $q$ and hence it will be already clean (i.e., $g(w)=w$ at this stage).
Altogether, we end up writing an output that constitutes a name of the set $w_{n_{k+1}}A$ and a sequence of numbers
$n_0,n_1,...,n_k,n_{k+1}$ with possible repetitions of each $n_i$ and infinitely many repetitions of the last $n_{k+1}$. 
With the help of $\C_\IN$ we can compute the value $n_{k+1}$ from this list. 
Given a point $r'\in w_{n_{k+1}}A$ and $n_{k+1}$ we can easily recover a point $r\in A$, which together with $p,q$ allows to find some point in $f(p)$. 
Altogether, this proves $f\leqW\C_\IN\times\PC_{2^\IN}$.
\end{proof}

Now Lemmas~\ref{lem:PCN2Na} and \ref{lem:PCN2Nb} together with Theorem~\ref{thm:products} yield
the following characterization.

\begin{corollary}
\label{cor:PCN2N}
$\PC_{2^\IN}\times\C_{\IN}\equivSW\C_{\IN}\times\PC_{2^\IN}\equivW\PC_{2^\IN}*\C_{\IN}\equivW\C_{\IN}*\PC_{2^\IN}\equivW\PC_{\IN\times2^\IN}$.
\end{corollary}

We obtain the following corollary that expresses this result in different terms using Proposition~\ref{prop:classes}
and Theorem~\ref{thm:probabilistic-computability}.

\begin{corollary}[Computability with finitely many mind changes]
\label{cor:PCN-composition}
Let $f$ be a multi-valued function on represented spaces.
Then the following are equivalent:
\begin{enumerate}
\item $f$ is Las Vegas computable with finitely many mind changes,
\item $f=g\circ h$ with some $g$ that is Las Vegas computable and some $h$ that is computable with finitely many mind changes,
\item $f=g\circ h$ with some $g$ that is computable with finitely many mind changes and some $h$ that is Las Vegas computable.
\end{enumerate}
\end{corollary}

This result actually justifies to call the $f$ below $\PC_{\IN\times2^\IN}$ Las Vegas computable
with finitely many mind changes.

Since the uniform measure on Cantor space is finite, it does not matter whether we define
$\PC_{2^\IN}$ using the interval $(0,\infty]$ or $(0,\infty)$. Likewise, it does not matter for $\IN$
by Lemma~\ref{lem:PCN}. The proof of Corollary~\ref{cor:PCN2N} also goes through in both cases.
Hence we obtain the following corollary, which says that also $\PC_{\IN\times2^\IN}$ can be defined
using $(0,\infty]$ or $(0,\infty)$.

\begin{corollary}
\label{cor:PCN2Nb}
$\PC_{\IN\times2^\IN}=\P_{(0,\infty]}\C_{\IN\times2^\IN}\equivSW\P_{(0,\infty)}\C_{\IN\times2^\IN}.$
\end{corollary}

We mention that also the equivalence class of $\PC_{\IN\times2^\IN}$ does not depend on whether
$\IN$ is equipped with the counting measure or the geometric measure. This is because the domains of $\P_{(0,\infty]}\C_{\IN\times2^\IN}$ are
identical in both cases and hence the multi-valued functions are identical.

\begin{lemma}
\label{lem:geometric-counting}
$\PC_{\IN\times2^\IN}$ does not depend on whether $\IN$ is equipped with the counting
measure or the geometric measure. 
\end{lemma}

\section{Changes of the Probability Space}
\label{sec:probability-space}

In this section we compare probabilistic choice for $[0,1]$ and $2^\IN$ as well
as probabilistic choice for $\IR$ and $\IN\times2^\IN$.

We recall that a {\em computable embedding} $f:X\to Y$ is a computable injective
function, such that the partial inverse $f^{-1}:\In Y\to X$ is computable too. 
If there is such a computable embedding $f$ such that $\range(f)$ is co-c.e.\ closed,
then $\C_X\leqSW \C_Y$ follows (see Corollary~4.3 in \cite{BBP12}).
Likewise, if there is a computable surjection $s:\In X\to Y$ with a co-c.e.\ closed domain $\dom(s)$,
then also $\C_Y\leqSW \C_X$ follows (see Proposition~3.7 in \cite{BBP12}).
We will implicitly use these ideas in the following and the proofs work for probabilistic
choice too with some assumptions on measure preservation.

We note that the binary representation 
\[\rho_2:2^\IN\to[0,1],p\mapsto\sum_{i=0}^\infty\frac{p(i)}{2^{i+1}}\]
is surjective, computable and measure-preserving. This yields immediately the reduction
$\P_I\C_{[0,1]}\leqSW\P_I\C_{2^\IN}$.
For the other direction we use the usual Smith-Volterra-Cantor set construction.

\begin{lemma}[Smith-Volterra-Cantor set]
For every computable $\varepsilon\in[0,1)$ there exists a computable embedding
$f_\varepsilon:2^\IN\to[0,1]$ such that
\[\lambda(f_\varepsilon(A))=\varepsilon\cdot\mu(A)\]
for every closed $A\In2^\IN$. Moreover, $\range(f_\varepsilon)$ is co-c.e.\ closed in this situation.
\end{lemma}
\begin{proof}
We consider the classical Smith-Volterra-Cantor set construction (see the $\varepsilon$--Cantor set in \cite[Lemma~18.9]{AB98}).
Given a computable $\varepsilon\in[0,1)$ we choose $\delta:=1-\varepsilon$ and given the unit interval $[0,1]$ we inductively
construct a sequence $(I_w)$ of closed intervals $I_w\In\IR$ indexed by binary words $w\in\{0,1\}^*$ as follows.
For the empty word $e$ and words $w\in\{0,1\}^{n-1}$, $n\geq1$ and symbols $c\in\{0,1\}$ we define
\begin{itemize}
\item $I_e:=[0,1]$,
\item $I_{wc}:=\left\{\begin{array}{ll}
         \ [a,a+\frac{b-a}{2}-\frac{\delta}{2^{2n}}] & \mbox{if $c=0$}\\
         \ [a+\frac{b-a}{2}+\frac{\delta}{2^{2n}},b] & \mbox{if $c=1$}
         \end{array}\right.$ where $[a,b]:=I_w$.
\end{itemize}
In other words, $I_{wc}$ is constructed from $I_w$ with $|w|=n-1$ by removing a middle piece of length $\frac{\delta}{2^{2n-1}}$
and $I_{w0}$ is the left half of the result, while $I_{w1}$ is the right half. The set 
\[C_\varepsilon:=\bigcap_{n=0}^\infty\bigcup_{w\in\{0,1\}^n}I_w\]
is called the Smith-Volterra-Cantor set and due to the construction we obtain
\[\lambda(C_\varepsilon)=1-\sum_{n=1}^\infty2^{n-1}\frac{\delta}{2^{2n-1}}=1-\delta=\varepsilon.\]
Now we define a computable function $f_\varepsilon:2^\IN\to[0,1]$ by
\[\{f_\varepsilon(p)\}:=\bigcap_{w\prefix p}I_w,\]
where the function value is meant to be the unique real in the given set (the value is unique by Cantor's Intersection
Theorem as $\diam(I_w)\leq2^{-|w|}$).
It is easy to see that $f_\varepsilon$ is computable due to the inductive nature of the construction and because $\varepsilon$ and hence
$\delta$ are computable. Moreover, $\range(f_\varepsilon)=f_\varepsilon(2^\IN)=C_\varepsilon$ and due to the symmetry of the construction we obtain 
$\lambda(f_\varepsilon(w2^\IN))=2^{-|w|}\varepsilon=\varepsilon\cdot\mu(w2^\IN)$.
Due to Lemma~\ref{lem:identity}, this implies $\lambda(f_\varepsilon(A))=\varepsilon\cdot\mu(A)$ for any closed $A\In2^\IN$. 
Since $2^\IN$ is computably compact, it follows that $f_\varepsilon(2^\IN)=C_\varepsilon$ is computably compact too by \cite[Theorem~3.3]{Wei03} and,
in particular, co-c.e.\ closed. This also implies that the partial inverse $f_\varepsilon^{-1}$ is computable (see for instance \cite[Corollary~6.7]{BBP12}).
\end{proof} 

Altogether, we obtain the following result.

\begin{proposition}[Cantor space and the unit interval]
\label{prop:probability-cantor}
Let $I$ be an interval and let $\varepsilon\in[0,1)$ be computable. Then we obtain
\[\P_I\C_{[0,1]}\leqSW\P_I\C_{2^\IN}\leqSW\P_{\varepsilon I}\C_{[0,1]}.\]
\end{proposition}

For the interval $I=(0,\infty]$ we can just choose $\varepsilon=\frac{1}{2}$ and we obtain the
following result that was already proved in \cite[Corollary~19]{BP10}.

\begin{corollary}
\label{cor:interval-cantor}
$\PC_{[0,1]}\equivSW\PC_{2^\IN}$.
\end{corollary}

In case of the spaces $\IN\times2^\IN$ and $\IR$ we can obtain a stronger result.

\begin{proposition}[Real numbers]
\label{prop:reals}
Let $I$ be an interval. Then we obtain
\[\P_I\C_\IR\equivSW\P_I\C_{\IN\times2^\IN}.\]
Here $\IN$ is equipped with the counting measure.
\end{proposition}
\begin{proof}
For the proof of $\P_I\C_\IR\leqSW\P_I\C_{\IN\times2^\IN}$ we use the function
\[f:\IN\times2^\IN\to\IR,(n,p)\mapsto\left\{\begin{array}{ll}
      \frac{1}{2}n+\rho_2(p) & \mbox{if $n$ is even}\\
      -\frac{1}{2}(n+1)+\rho_2(p) & \mbox{if $n$ is odd}
\end{array}\right.,\]
which is defined with the help of the binary representation $\rho_2$. This function $f$
is computable, surjective and measure-preserving (since the Lebesgue measure $\lambda$ is translation-invariant) 
and hence we obtain the desired reduction.

For the proof of $\P_I\C_{\IN\times2^\IN}\leqSW\P_I\C_\IR$ we use the function
$f_{\frac{1}{2}}$ from the Smith-Volterra-Cantor set construction and we define
\[g:\IN\times2^\IN\to\IR,(n,p)\mapsto 3n+2\cdot f_{\frac{1}{2}}(p).\]
The function $g$ is injective and even a computable embedding and it satisfies
\[\lambda(g(A))=2\lambda(f_{\frac{1}{2}}(A))=\mu(A)\]
for any closed set $A\In\IN\times2^\IN$, where $\mu$ is the measure on $\IN\times2^\IN$.
Finally, $\range(g)=3\IN+\range(2\cdot f_\frac{1}{2})$ is clearly a co-c.e.\ closed set. This is because
$2^\IN$ is computably compact and hence $\range(f_\frac{1}{2})$ is computably compact too and $\range(g)$
is the union of clearly separated copies of $\range(2\cdot f_\frac{1}{2})$ and hence it is co-c.e.\ closed.
Altogether, this proves $\P_I\C_{\IN\times2^\IN}\leqSW\P_I\C_\IR$.
\end{proof}

We give an example that shows that this stronger type of result cannot be achieved for the unit interval $[0,1]$
and Cantor space $2^\IN$, hence Proposition~\ref{prop:probability-cantor} is in a certain sense optimal.
The reason for this difference are the different connectedness properties of $[0,1]$ and $2^\IN$.
While the first one does not have two disjoint closed subsets of measure $\frac{1}{2}$, the second one does.

\begin{proposition}
\label{prop:example-unit-cantor}
$\P_{\geq\frac{1}{2}}\C_{[0,1]}\lW\P_{\geq\frac{1}{2}}\C_{2^\IN}$. 
\end{proposition}
\begin{proof}
The positive part of the reduction follows from Proposition~\ref{prop:probability-cantor}.
In order to prove the strictness, we claim that $\C_2\leqSW\P_{\geq\frac{1}{2}}\C_{2^\IN}$ and $\C_2\nleqW\P_{\geq\frac{1}{2}}\C_{[0,1]}$.

The first part is easy to see as $2^\IN$ can be subdivided into $A_i:=i2^\IN$ for $i\in\{0,1\}$ and 
$\mu(A_0)=\mu(A_1)=\frac{1}{2}$. Now any subset $C\In\{0,1\}$ is computably mapped to $A_C=\bigcup_{i\in C}A_i$ and
given some $q\in A_C$ one can directly recover an $i$ with $q\in A_i$. This yields a computable reduction   $\C_2\leqSW\P_{\geq\frac{1}{2}}\C_{2^\IN}$.

In order to prove the negative claim, we assume for a contradiction that we have $\C_2\leqW\P_{\geq\frac{1}{2}}\C_{[0,1]}$.
Then there are computable $H,K$ such that $H\langle\id,FK\rangle$ is a realizer of $\C_2$ whenever $F$ is a realizer of $\P_{\geq\frac{1}{2}}\C_{[0,1]}$.
We recall that we can assume that $H$ is defined on the minimal required domain and that we use the signed-digit representation of $[0,1]$.
Let $p$ be a name of $\{0,1\}$. Then $K(p)$ is a name of a closed set $A\In[0,1]$ with $\lambda(A)\geq\frac{1}{2}$.
Let $A_i\In A$ be the set of all points that have only names $q$ such that $H\langle p,q\rangle$ is mapped to (a name of) $i\in\{0,1\}$.
We write $H\langle p,q\rangle=i$ in this situation. Let $A_2\In A$ be a set of all those points that have names $q$ with $H\langle p,q\rangle=0$ as
well as names $q$ with $H\langle p,q\rangle=1$. Then $A=A_0\cup A_1\cup A_2$ is a disjoint union, where one of the $A_i$'s might be empty.
Let $N$ be the set of names of the points in $A$, which is compact for the signed-digit representation. 
Since $H$ is uniformly continuous\footnote{Strictly speaking, we use the following property stronger than uniform continuity on $K$: 
if $f:X\to Y$ is a continuous function on metric spaces $(X,d_X)$ and $(Y,d_Y)$ and $K\In X$ is compact, then
for every $\varepsilon>0$ there exists a $\delta>0$ such that for all $x\in X$ and all $y\in K$ we obtain
that $d_X(x,y)<\delta$ implies $d_Y(f(x),f(y))<\varepsilon$. This can be proved analogously to the fact
that $f|_K$ is uniformly continuous. When we refer to ``uniform continuity'' of a function $f:X\to Y$ on $K$, then we mean this property.} 
on the compact set $\{p\}\times N$, there is a finite prefix $w\prefix p$ such that $H\langle w\IN^\IN,q\rangle$ is a singleton for every
name $q$ of a point in $A$. Since $K$ is continuous, we can assume that $w$ is long enough such that
$K(w\IN^\IN)$ only contains names of sets $A'$ with $\lambda(A'\setminus A)\leq\frac{(1-\lambda(A))}{3}$.
There also exists an $i\in\{0,1\}$ such that $\lambda(A_i)\leq\frac{1}{2}\lambda(A)$.
Let now $p'$ be a name of $C:=\{i\}$ that extends $w$, i.e., $w\prefix p'$.
Then $K(p')$ is a name of a set $A'$ with $\lambda(A')\geq\frac{1}{2}$. We claim that $A'\cap A\In A_i$, for otherwise there is an 
$x\in A'\cap A$ with a name $q$ such that $H\langle p',q\rangle=1-i\not\in C$ and there is a realizer $F$ of $\P_{\geq\frac{1}{2}}\C_{[0,1]}$ 
such that $FK(p')=q$ in contradiction to the assumption and the choice of $C$. Hence, if $\lambda(A)<1$, then we obtain
\[\lambda(A')\leq\lambda(A_i)+\lambda(A'\setminus A)\leq\frac{1}{2}\lambda(A)+\frac{1-\lambda(A)}{3}<\frac{1}{2}\]
in contradiction to the assumption. If, on the other hand, $\lambda(A)=1$, then $A=[0,1]$ and by continuity of $H$
we obtain $d(A_0,A_1):=\inf\{|a-b|:a\in A_0,b\in A_1\}>0$. This implies $\lambda(A_2)>0$ and hence we can 
even assume $\lambda(A_i)<\frac{1}{2}$ for some $i\in\{0,1\}$. Similarly to above, this yields the contradiction
$\lambda(A')\leq\lambda(A_i)+\lambda(A'\setminus A)<\frac{1}{2}$.
\end{proof}

This proves that Proposition~\ref{prop:probability-cantor} cannot be extended to the case $\varepsilon=1$.

\section{Weak Weak K\H{o}nig's Lemma}
\label{sec:WWKL}

It is easy to see that $\PC_{2^\IN}$ is essentially equivalent to Weak Weak K\H{o}nig's Lemma as it is known
from reverse mathematics \cite{Sim09}.
By $\Tr$ we denote the set of binary trees $T\In2^*$, represented via their characteristic functions.
Then  Weak K\H{o}nig's Lemma $\WKL:\In\Tr\mto2^\IN,T\mapsto[T]$ is the problem
to map a binary tree $T$ to set $[T]$ of its infinite paths.
The domain $\dom(\WKL)$ consists of all infinite binary trees.
It has been proved in \cite[Corollary~83 and Theorem~8.5]{BG11} that $\C_{2^\IN}\equivSW\WKL$ (see also \cite{GM09,BBP12,BGM12}).
The proof is essentially based on the fact that the map
\[[\,]:\Tr\to\AA_-(2^\IN),T\mapsto[T]\]
that maps infinite binary trees to closed subsets of Cantor space (with respect to negative information) is computable and has a multi-valued computable inverse.
Analogously to Weak K\H{o}nig's Lemma, we can define the computational counterpart of Weak Weak K\H{o}nig's Lemma, as it is used in reverse mathematics \cite{Sim09}.

\begin{definition}[Weak Weak K\H{o}nig's Lemma]
{\em Weak Weak K\H{o}nig's Lemma} is the problem
\[\WWKL:\In\Tr\mto2^\IN,T\mapsto[T]\]
restricted to the set $\dom(\WWKL)=\{T:\mu([T])>0\}$ of those trees $T$, whose set
of infinite paths $[T]$ has positive measure.
\end{definition}

Weak Weak K\H{o}nig's Lemma was already studied in the Weihrauch lattice in \cite{BP10}
and the following result was noticed. This observation can be proved using the computable map $T\mapsto[T]$ above
which also preserves the respective measure conditions.

\begin{proposition}[Weak Weak K\H{o}nig's Lemma]
\label{prop:WWKL}
$\WWKL\equivSW\PC_{2^\IN}\equivSW\PC_{[0,1]}$.
\end{proposition}

We mention that by Corollary~\ref{cor:products-pairing} we obtain the following.

\begin{corollary}
\label{cor:WWKL-product}
$\WWKL*\WWKL\equivW\WWKL$.
\end{corollary}

Theorem~\ref{thm:probabilistic-computability} yields the following characterization.

\begin{corollary}[Las Vegas computability]
\label{cor:Las-Vegas-WWKL}
The following are equivalent to each other:
\begin{enumerate}
\item $f\leqW\WWKL$,
\item $f$ is Las Vegas computable.
\end{enumerate}
\end{corollary}

It is important to note that $\WWKL$ can be separated from $\WKL$. 
This was proved in \cite[Theorem~20]{BP10} and independently in \cite[Proposition~4.2]{DDH+12}.

\begin{corollary}[Weak and Weak Weak K\H{o}nig's Lemma]
\label{cor:WWKL-WKL}
$\WWKL\lW\WKL$.
\end{corollary}

We will prove a more general result in Section~\ref{sec:probabilistic-degrees}, see Corollary~\ref{prop:WKL-probabilistic}.
Corollary~\ref{cor:WWKL-WKL} can also be rephrased using classes of functions as follows.

\begin{corollary}
Every Las Vegas computable multi-valued function is non-deter\-ministically computable,
but there are non-deterministically computable multi-valued functions that are not
Las Vegas computable.
\end{corollary}

To emphasize multi-valuedness is important here, since all single-valued non-deter\-ministically
computable functions (on computable metric spaces) are automatically computable, as proved in \cite[Corollary~8.8]{BG11}.

Dorais et al.\ \cite{DDH+12} also introduced a quantitative version $\varepsilon\dash\WWKL$ of Weak Weak K\H{o}nig's Lemma
that requires a probability above some threshold $\varepsilon$. We define this version here
and we will study it starting from Section~\ref{sec:probability}.

\begin{definition}[Quantitative Weak Weak K\H{o}nig's Lemma]
For $\varepsilon\in\IR$ we denote by
$\varepsilon\dash\WWKL$ the restriction of 
Weak Weak K\H{o}nig's Lemma $\WWKL$ to the set
$\dom(\varepsilon\dash\WWKL)=\{T:\mu([T])>\varepsilon\}$.
\end{definition}

From the aforementioned results it is clear that we obtain the following corollary.

\begin{corollary}
\label{cor:epsilon-WWKL-Cantor}
$\varepsilon\dash\WWKL\equivSW\P_{>\varepsilon}\C_{2^\IN}$ for all $\varepsilon\in\IR$.
\end{corollary}

\section{Jumps of Weak Weak K\H{o}nig's Lemma}
\label{sec:jumps}

In this section we want to mention some observations on the jump of Weak Weak K\H{o}nig's Lemma.
We will also see that there is a certain trade-off between the complexity of the underlying spaces $X$ in $\PC_X^{(n)}$ and the number of jumps $n$.
We start with mentioning that with iterated jumps of $\WWKL$ one actually climbs up the Borel hierarchy.
We recall that a multi-valued map $f:\In X\mto Y$ on Polish spaces $X$ and $Y$ is called {\em $\SO{n}$--measurable} if preimages 
\[f^{-1}(U):=\{x\in X:f(x)\cap U\not=\emptyset\}\]
of open sets $U\In Y$ are $\SO{n}$--sets in the Borel hierarchy relatively to $\dom(f)$ (see \cite{Bra05} for more details).

\begin{proposition}
\label{prop:WWKL-measurable}
$\WWKL^{(n)}$ is $\SO{n+2}$--measurable, but not $\SO{n+1}$--measurable with respect to the Borel hierarchy for all $n\in\IN$.
\end{proposition}
\begin{proof}
Since $\WWKL^{(n)}\leqSW\WKL^{(n)}\leqSW\lim^{(n)}$ it is clear that $\WWKL^{(n)}$ is $\SO{n+2}$--measurable.
This is because $\lim^{(n)}$ is $\SO{n+2}$--measurable and hence so is $\WWKL^{(n)}$ by \cite[Proposition~7.5]{Bra05}.
On the other hand, $\C_2^{(n)}\leqSW\WWKL^{(n)}$ (as one can easily show) and $\C_2^{(n)}$ is not $\SO{n+1}$--measurable.
Hence $\WWKL^{(n)}$ is not $\SO{n+1}$--measurable by \cite[Proposition~7.5]{Bra05}.
We sketch the proof that $\C_2^{(n)}$ is not $\SO{n+1}$--measurable, which can be proved by induction. 
It is clear that $\C_2$ is not continuous, which proves the case $n=0$. 
For $n=1$ we note that $\C_2'\equivSW\CL_2$ by \cite[Theorem~9.4]{BGM12} where $\CL_2$ denotes the cluster point problem of $\{0,1\}$, also called
the infinite pigeonhole principle. Since $\CL_2^{-1}\{0\}$ is the set of binary sequences that contain infinitely many zeros, 
which is known to be $\PO{2}$--complete (see, for instance, \cite[Exercise~II.23.1]{Kec95}),
it is not a $\SO{2}$--set and hence $\C_2'\equivSW\CL_2$ is not $\SO{2}$--measurable.
We note that for a convergent sequence $(p_i)$ in $\{0,1\}^\IN$ we obtain
\[\lim_{i\to\infty}p_i(k)=0\iff(\exists j)(\forall i\geq j)\;p_i(k)=0\iff(\forall j)(\exists i\geq j)\;p_i(k)\not=1.\]
This implies that each further application of a limit adds exactly one quantifier (the above equivalence allows to choose
whether it is an existential or a universal one).
From this it follows by induction that the preimage $(\CL_2^{(n)})^{-1}(\{0\})$ is $\PO{n+2}$--complete and hence
not a $\SO{n+2}$ set
(using suitable complete sets of higher levels of the Borel hierarchy)
and hence $\C_2^{(n+1)}\equivSW\CL_2^{(n)}$ is not $\SO{n+2}$--measurable.
\end{proof}

In particular, we obtain the following by \cite[Proposition~7.5]{Bra05}.

\begin{corollary}
\label{cor:WWKL-hierarchy}
$\WWKL^{(n)}\lW\WWKL^{(n+1)}$ for all $n\in\IN$.
\end{corollary}

We also introduce discrete jumps of Weihrauch degrees.
In \cite{BBP12} the discrete limit map $\lim_\Delta:\In\IN^\IN\to\IN^\IN,\langle p_0,p_1,...\rangle\mapsto\lim_{i\to\infty}p_i$ 
was studied, which is the limit map with respect to the discrete topology on $\IN^\IN$, i.e., $\lim_\Delta$ is the restriction of $\lim$
to eventually constant sequences. This leads to the {\em discrete jump} $\delta_X^\Delta:=\delta_X\circ\lim_\Delta$ of a representation
and analogously to the {\em discrete jump} $f^\Delta:\In(X,\delta_X^\Delta)\mto(Y,\delta_Y)$ of a multi-valued function.
It is easy to see that we obtain $f^\Delta\equivSW f\stars\lim_\Delta$. In \cite[Fact~3.7]{BGM12} it was proved that 
$\lim_\Delta\equivSW(\id\times\C_\IN)$, i.e., $\lim_\Delta$ is strongly equivalent to the cylindrification of $\C_\IN$.
We can use this concept to express the following result.
 
\begin{theorem}[Discrete jump]
\label{thm:discrete-jump}
$\WWKL^\Delta\equivSW\PC_{2^\IN}^\Delta\equivSW\PC_\IR\equivSW\PC_{\IN\times2^\IN}$.
\end{theorem}
\begin{proof}
It follows from Proposition~\ref{prop:reals} that $\PC_\IR\equivSW\PC_{\IN\times2^\IN}$
and from Proposition~\ref{prop:WWKL} that $\PC_{2^\IN}^\Delta\equivSW\WWKL^\Delta$.

A close inspection of the proof of Lemma~\ref{lem:PCN2Na} shows that the number $n$ that occurs in the pair $\langle n,k\rangle$ of that proof can be encoded in the
set $A_n$, by replacing it by $C_n:=01^n0A_n$. Then an application of $\PC_{2^\IN}$ to $C_n$ yields a point $p$ from which 
a point in the original set $A$ can be reconstructed without any direct access to the original input. 
This proves 
\[\PC_{\IN\times2^\IN}\leqSW\PC_{2^\IN}\stars(\id\times\C_\IN)\equivSW\PC_{2^\IN}\stars\lim\nolimits_\Delta\equivSW\PC_{2^\IN}^\Delta.\]
On the other hand, exactly the same proof as the proof of Lemma~\ref{lem:PCN2Nb} shows
\[\PC_{2^\IN}\stars\lim\nolimits_\Delta\leqSW\C_\IN\times\PC_{2^\IN}.\]
Finally, Proposition~\ref{prop:products} implies $\C_\IN\times\PC_{2^\IN}\leqSW\PC_{\IN\times2^\IN}$, which completes the proof.
\end{proof}

We mention that by Corollary~\ref{cor:products-pairing} we obtain the following.

\begin{corollary}
\label{cor:WWKL-Delta-product}
$\WWKL^\Delta*\WWKL^\Delta\equivW\WWKL^\Delta$.
\end{corollary}

If we combine Theorem~\ref{thm:discrete-jump} with Corollary~\ref{cor:PCN2N} and Proposition~\ref{prop:WWKL}, then we obtain the following corollary.

\begin{corollary}
\label{cor:WWKL-Delta-WWKL-CN}
$\WWKL^\Delta\equivW\WWKL*\C_\IN$.
\end{corollary}

This yields the following characterization of Las Vegas computability with finitely many mind changes.

\begin{corollary}[Las Vegas computability with finitely many mind changes]
\label{cor:Las-Vegas-WWKL-Delta}
The following are equivalent to each other:
\begin{enumerate}
\item $f\leqW\WWKL^\Delta$,
\item $f$ is Las Vegas computable with finitely many mind changes.
\end{enumerate}
\end{corollary}

Theorem~\ref{thm:discrete-jump} together with 
$\lim_\Delta\leqSW\lim$ yields the positive content of the following corollary.
The fact that the reduction is strict follows from $\PC_{\IN\times2^\IN}\leqW\lim$, which holds by \cite[Theorem~8.7]{BBP12} (an intuitive explanation of this fact is that there is a limit computation that always selects the left most path in a given infinite tree~${T\In\IN\times\{0,1\}^*}$), 
which means
that $\PC_{\IN\times2^\IN}$ is $\SO{2}$--measurable (by \cite[Proposition~7.5]{Bra05}), whereas 
$\WWKL'$ is not $\SO{2}$--measurable by Proposition~\ref{prop:WWKL-measurable}.

\begin{corollary}
\label{cor:N2N}
$\PC_{\IN\times2^\IN}\lSW\PC_{2^\IN}'\equivSW\WWKL'$.
\end{corollary}

In \cite{BGM12} the cluster point problem of a space was studied and it was proved
that the cluster point problem is the jump of the closed choice problem of the same space.
This result straightforwardly generalizes to the probabilistic setting.
We will call the problem to find a cluster point of a sequence that has a set of cluster
points of a certain measure the probabilistic cluster point problem.

\begin{definition}[Probabilistic cluster point problem]
Let $X$ be a computable metric space that is equipped with a Borel measure $\mu$
and let $I$ be some interval. Then we call
\[\P_I\CL_X:\In X^\IN\mto X,(x_n)\mapsto\{x\in X:\mbox{$x$ is a cluster point of $(x_n)$}\}\]
the {\em probabilistic cluster point problem} with measure in $I$,
where $\dom(\P_I\CL_X)$ is the set of all sequences $(x_n)$ in $X$, with a non-empty set of cluster points $C$
that satisfies $\mu(C)\in I$.
\end{definition}

We use similar abbreviations as for $\P_I\C_X$, for instance $\PCL_X:=\P_{>0}\CL_X:=\P_{(0,\infty]}\CL_X$ etc.
Now we obtain the following general result.

\begin{theorem}[Probabilistic cluster point problem]
Let $X$ be a computable metric space that is equipped with a Borel measure and let $I$ be an interval.
Then
\[(\P_I\C_X)'\equivSW\P_I\CL_X.\]
\end{theorem}
\begin{proof}
We just adapt results that have been provided in \cite{BGM12}. If we denote by $\Low_X:\In X^\IN\mto\AA_-(X)$
the surjective map that maps each sequence $(x_n)$ to the set $\Low_X(x_n)$ of its cluster points, then we obtain
\[\P_I\CL_X=\P_I\C_X\circ \Low_X.\]
By \cite[Proposition~9.2]{BGM12} we have $\Low_X\leqSW\lim$ and together with \cite[Theorem~5.14]{BGM12}
we obtain $\P_I\CL_X\leqSW(\P_I\C_X)'$. The other direction $(\P_I\C_X)'\leqSW\P_I\CL_X$ follows
by \cite[Corollary~9.5]{BGM12}, which states that the jump $(\Low_X^{-1})'$ of the multi-valued inverse of $\Low_X$
is computable.
\end{proof}

Together with Proposition~\ref{prop:WWKL} and the fact that jumps are monotone with respect to $\leqSW$
we obtain the following corollary.

\begin{corollary}
$\WWKL'\equivSW\PCL_{2^\IN}\equivSW\PCL_{[0,1]}$.
\end{corollary}

\section{Changes of the Probability Values}
\label{sec:probability}

In this section we discuss the dependency of probabilistic choice on different lower bounds on the probability.
Intuitively, it should make choice easier if the measure of the set that one chooses from increases.
Indeed, in some cases this is strictly so. For instance, one easily obtains
\[\frac{1}{2}\mbox{-}\WWKL\lSW\frac{1}{3}\mbox{-}\WWKL\lSW\frac{1}{4}\mbox{-}\WWKL\lSW...\]
While the reductions $\leqSW$ are obvious, the strictness of these reductions follows since 
$\C_n\leqSW\frac{1}{n+1}\mbox{-}\WWKL$ and $\C_n\nleqSW\frac{1}{n}\mbox{-}\WWKL$.
The latter is true for mere counting reasons, which show that there cannot be $n$ disjoint subsets of $2^\IN$ all
with measure $>\frac{1}{n}$. In fact, $\#\frac{1}{n}\dash\WWKL=n-1$ and $\#\C_n=n$ and
Proposition~\ref{prop:cardinality} yields the result.
As a consequence one obtains $\varepsilon$-$\WWKL\lSW\WWKL$ for all $\varepsilon>0$,
which was also proved by Dorais et al.\ in \cite[Proposition~4.7]{DDH+12}.

Here we generalize the above observation in several respects: for one we prove that the strictness of the
reduction even occurs for ordinary Weihrauch reducibility $\leqW$, secondly we separate the probabilistic choice
principles for arbitrary probabilities (not just for fractions of the form $\frac{1}{n}$) and lastly we prove the result for both spaces, the unit interval $[0,1]$
and Cantor space $2^\IN$.
Instead of $\C_n$ as above we use a slightly more general variant of choice. For non-negative integers $a,b$ we denote by $\C_{a,b}$ 
the closed choice operation for the set $b=\{0,1,...,b-1\}$, restricted to subsets $C\In b$ with cardinality $|C|\geq a$.
In other words, 
\[\C_{a,b}:=\P_{\geq a}\C_b.\] 
Equivalent problems have been studied under the name $\LLPO_{b,b-a}$ by Mylatz~\cite{Myl06}.
He also classified the exact relation of the problems $\LLPO_{n,m}$ and $\LLPO_{k,l}$ to each other in terms of number theoretic properties of $n,m,k$ and $l$~\cite[Satz~15]{Myl06}.
Other finite choice principles with restricted cardinality have already been studied by Pauly and Le Roux in \cite{LP13}. 
We point out that the negative part of the following proof is very similar to the proof of Proposition~\ref{prop:example-unit-cantor} (and to a lesser extent
to the proof of Theorem~\ref{thm:AC-*-WWKL}).

\begin{theorem}[Probability dependency]
\label{thm:probability-dependency}
Let $\varepsilon,\delta\in[0,1]$ and $X=2^\IN$ or $X=[0,1]$. Then
\[\P_{>\varepsilon}\C_{X}\leqW\P_{>\delta}\C_{X}\iff\P_{>\varepsilon}\C_{X}\leqSW\P_{>\delta}\C_{X}\iff\varepsilon\geq\delta.\]
\end{theorem}
\begin{proof}
``$\Longleftarrow$'' these reductions are clear and follow from Proposition~\ref{prop:interval-monotonicity}.\\
``$\TO$'' Let $\varepsilon<\delta$. Then there are positive integers $a<b$ with
$\varepsilon<\frac{a}{b}\leq\delta$. 
By Proposition~\ref{prop:probability-cantor} it suffices to show
$\C_{a,b}\leqW\P_{>\varepsilon}\C_{[0,1]}$ and $\C_{a,b}\nleqW\P_{>\delta}\C_{2^\IN}$,
since this implies $\P_{>\varepsilon}\C_{X}\nleqW\P_{>\delta}\C_{X}$ for both $X=[0,1]$ and $X=2^\IN$.

In order to prove the first statement, we select $b$ consecutive 
disjoint closed intervals $I_0,...,I_{b-1}\In[0,1]$ with rational endpoints
of equal length $l=\lambda(I_i)$ for all $i\in B=\{0,1,...,b-1\}$ with
$\frac{\varepsilon}{a}<l<\frac{1}{b}$.
The function that maps a subset $C\In B$ to $A_C:=\bigcup_{i\in C}I_i$
is computable and if $C$ is of cardinality $|C|\geq a$, then $\lambda(A_C)\geq al>\varepsilon$.
Since given a point $x\in A_C$ one can easily recover the unique number $i\in B$ of the interval $I_i$ with $x\in I_i$,
one obtains the reduction $\C_{a,b}\leqSW\P_{>\varepsilon}\C_{[0,1]}$.

Let us now assume for a contradiction that $\C_{a,b}\leqW\P_{>\delta}\C_{2^\IN}$.
Then there are computable $H,K$ such that $H\langle\id,FK\rangle\vdash \C_{a,b}$, whenever
$F\vdash \P_{>\delta}\C_{2^\IN}$ holds. 
Let $p$ be a name of the entire set $B=\{0,1,...,b-1\}$. Then $K(p)$ is a name of a closed set $A\In2^\IN$
with measure $\mu(A)>\delta$. Now $H\langle p,q\rangle$ is a (name of a) point $i\in B$ for every $q\in A$. 
For simplicity we write $H\langle p,q\rangle=i$ in this situation. Let
\[A_i:=\{q\in A:H\langle p,q\rangle=i\}\]
for all $i\in B$. 
By definition the sets $A_0,...,A_{b-1}$ are pairwise disjoint.
By a version of the Pigeonhole Principle there must be a set $C\In B$ of cardinality $|C|=a$ such that $A_C:=\bigcup_{i\in C}A_i$
has measure $\mu(A_C)\leq\frac{a}{b}\mu(A)$.

Since $H$ is uniformly continuous on the compact set $\{p\}\times A$, it follows that there
is some finite prefix $w\prefix p$ such that $H\langle w\IN^\IN,q\rangle$ is (a name of) a singleton for every $q\in A$.
Moreover, since $K$ is continuous we can assume that $w$ is long enough such that $K(w\IN^\IN)$ only contains names of sets $A'\In2^\IN$
with $\mu(A'\setminus A)\leq\frac{a}{b}(1-\mu(A))$. 
Now there is a name $p'$ of the set $C$ with $w\prefix p'$ and hence $K(p')$ is a name of a set $A'$ as above.
It is clear that $A'\cap A\In A_C$: for if  $q\in A'\cap A$, then there is a realizer $F$ of $\P_{>\delta}\C_{2^\IN}$
such that $FK(p')=q$ and hence $H\langle p,q\rangle=H\langle p',q\rangle=H\langle p',FK(p')\rangle\in C$, which implies $q\in A_C$.
Finally, the measure of $A'$ satisfies
\[\mu(A')\leq\mu(A_C)+\mu(A'\setminus A)\leq\frac{a}{b}\leq\delta\]
in contradiction to the requirement $\mu(A')>\delta$.
\end{proof}

In particular, we obtain the following result.

\begin{corollary}
\label{cor:probability-WWKL}
$\varepsilon$-$\WWKL\leqW\delta$-$\WWKL\iff\varepsilon\geq\delta$ for $\varepsilon,\delta\in[0,1]$.
\end{corollary}

This result was independently proved by Dorais et al.\ \cite[Proposition~4.7]{DDH+12}.
Since the proof of Theorem~\ref{thm:probability-dependency} includes the case $\frac{a}{b}=\delta$, we obtain the following.

\begin{corollary}
\label{cor:C-n-WWKL}
$\C_n\nleqW\frac{1}{n}\dash\WWKL$ for all $n\geq1$.
\end{corollary}

Besides $\varepsilon\dash\WWKL$ we can also consider a $*$--version of this principle that we define next.
Essentially, the definition is
$*\dash\WWKL:=\bigsqcup_{n\in\IN} 2^{-n}\dash\WWKL$, which can be understood as a uniform version of the logical statement~``$(\forall n)\;2^{-n}\dash\WWKL$.''
We phrase the principle slightly more precisely.

\begin{definition}
We define $*\dash\WWKL:\In\IN\times\Tr\mto2^\IN$ by
\[*\dash\WWKL(n,T):=2^{-n}\dash\WWKL(T),\]
where $\dom(*\dash\WWKL):=\{(n,T)\in\IN\times\Tr:\mu([T])>2^{-n}\}$.
\end{definition}

That is the input to $*\dash\WWKL$ is a pair $(n,T)$, where $n$ is a natural number and 
$T$ is a tree such that the set of infinite paths $[T]$ of $T$ satisfies $\mu([T])>2^{-n}$.
The output is an infinite path $p\in[T]$.
It follows from Proposition~\ref{prop:products} that $*\dash\WWKL$ is idempotent.
Similarly, we can define a lower counterpart $(1-*)\dash\WWKL:=\bigsqcap_{n\in\IN}(1-2^{-n})\dash\WWKL$,
which can be understood as corresponding to the logical statement ``$(\exists n)\;(1-2^{-n})\dash\WWKL$.''
Also in this case we repeat the definition for clarity.

\begin{definition}
We define $(1-*)\dash\WWKL:\In\Tr^\IN\mto2^\IN$ by
\[(1-*)\dash\WWKL(T_n)_n:=\bigsqcup_{n\in\IN}(1-2^{-n})\dash\WWKL(T_n),\]
where $\dom((1-*)\dash\WWKL):=\{(T_n)_n\in\Tr^\IN:(\forall n\in\IN)\;\mu([T_n])>1-2^{-n}\}$.
\end{definition}

Thus, the input to $(1-*)\dash\WWKL$ is a sequence $(T_n)_n$ of trees with $\mu([T_n])>1-2^{-n}$
and the output is an infinite path $p\in[T_n]$ of one of these trees $T_n$ together with the information $n$
to which tree the path belongs.
It is clear that Corollary~\ref{cor:probability-WWKL} implies the following.

\begin{corollary}
\label{cor:*WWKL-WWKL}
$(1-*)\dash\WWKL\lW\varepsilon\dash\WWKL\lW*\dash\WWKL\leqSW\WWKL$ for every $\varepsilon\in(0,1)$.
\end{corollary}

We want to show that the latter reduction is strict too.
We use $\PCC_{[0,1]}$, which is closed choice on $[0,1]$ restricted
to connected sets of positive measure. In other words, this is choice restricted to proper
intervals, which was already considered in \cite{BG11a} under the name $\C_I^-$
and in \cite{BLP12} under the name $\ConC_1^-$.
In \cite[Proposition~3.8]{BG11a} it was proved that $\PCC_{[0,1]}\leqW\C_\IN$ holds.

We first prove that $\PCC_{[0,1]}$ is join-irreducible. We mention that due to distributivity of the Weihrauch lattice
$g$ is {\em join-irreducible} 
(in the sense introduced in Section~\ref{sec:preliminaries}) if and only if
$g\leqW\bigsqcup_{i=0}^\infty f_i$ implies that there exists $i\in\IN$ with $g\leqW f_i$.
We prove a slightly more general result that we apply to other problems than $\PCC_{[0,1]}$ at a later stage.

\begin{lemma}
\label{lem:join-irreducible}
Every restriction $\C_{[0,1]}|_\CC$ of closed choice to a set $\CC$ of closed subsets $A\In[0,1]$ 
with $[0,1]\in\CC$ is join-irreducible.
\end{lemma}
\begin{proof}
Let $(f_i)_i$ be a sequence of multi-valued functions and
let $f:=\bigsqcup_{i=0}^\infty f_i$. 
Let us assume that $\C_{[0,1]}|_\CC\leqW f$ holds.
Then there are computable $H,K$ such that $H\langle\id,FK\rangle$ is a realizer of $\C_{[0,1]}|_\CC$
for every realizer $F$ of $f$. Now let $p$ be a name of $[0,1]$.
Then $K(p)$ is a name of a pair $(n,x)$, such that $x$ is an input to $f_n$. 
Since $K$ is continuous, there is a finite prefix $w\prefix p$
such that $K(w\IN^\IN)$ only contains names of pairs $(n,x)$ with the same fixed $n$.
Now there is a computable function $L$ that transforms any name $p'$ of a closed set $A\In[0,1]$ 
into a name $q=L(p')$ of the same set that starts with $w$, i.e., such that $w\prefix q$.
This is because $w$ contains no negative information that overlaps with $[0,1]$.
Hence, the functions $H, KL$ witness the reduction $\C_{[0,1]}|_\CC\leqW f_n$.
\end{proof}

We get the following immediate corollary.

\begin{corollary}
\label{cor:PCC-irreducible}
$\PCC_{[0,1]}$ is join-irreducible.
\end{corollary}

Now we are prepared to prove the following.

\begin{proposition}
\label{prop:PCC-WWKL}
$\PCC_{[0,1]}\nleqW*\dash\WWKL$.
\end{proposition}
\begin{proof}
Let us assume for a contradiction that $\PCC_{[0,1]}\leqW*\dash\WWKL$ holds.
Since $\PCC_{[0,1]}$ is join-irreducible we can conclude by Corollary~\ref{cor:PCC-irreducible}, there is some $n$ such that
$\PCC_{[0,1]}\leqW2^{-n}\dash\WWKL$.
Now we obtain
\[\C_{2^n}\leqW\K_\IN\equivSW\C_2^*\leqW\PCC_{[0,1]}\leqW2^{-n}\dash\WWKL\]
in contradiction to Corollary~\ref{cor:C-n-WWKL}.
Here $\K_\IN$ denotes choice for finite subsets of $\IN$ that are given together with an upper bound of the set
and $\K_\IN\equivSW\C_2^*$ has been proved in \cite[Proposition~10.9]{BGM12}.
The reduction $\C_2^*\leqW\PCC_{[0,1]}$ has been proved in \cite[Proposition~7.2]{BLP12}.
\end{proof}

We note that this proof also yields another proof of the strictness of $\C_2^*\lW\PCC_{[0,1]}$ since $\C_2^*\leq*\dash\WWKL$ (compare the remark
after \cite[Proposition~7.2]{BLP12}).
Since we also obtain $\PCC_{[0,1]}\leqW\PC_{[0,1]}\equivW\WWKL$, we arrive at the following corollary of Proposition~\ref{prop:PCC-WWKL}.

\begin{corollary}
\label{cor:*-WWKL}
$*\dash\WWKL\lW\WWKL$.
\end{corollary}

\section{The Lebesgue Density Lemma}
\label{sec:LDL}

In the following we will need a simple version of the Lebesgue Density Theorem, which we will call the 
Lebesgue Density Lemma, and for that purpose we will classify its Weihrauch degree.
The classical Lebesgue Density Theorem (see \cite[5.8(ii)]{Bog07}), which is a special case of the Lebesgue Differentiation Theorem
for measurable sets, says that for every measurable set $A\In\IR^n$
\[\lim_{\varepsilon\to0}\frac{\lambda(A\cap B(x,\varepsilon))}{\lambda(B(x,\varepsilon))}=1\]
for almost all $x\in A$ (where $B(x,\varepsilon)$ denotes the ball around $x$ with radius $\varepsilon$).
We will use a special case of this theorem in Cantor space that is in for-all-exists form (see \cite[Theorem~1.2.3]{DH10} for a direct proof of this special case).

\begin{lemma}[Lebesgue Density Lemma]
\label{lem:LDL}
For every closed $A\In 2^\IN$ with $\mu(A)>0$ and every $k\in\IN$ there exists a word $w\in\{0,1\}^*$
such that 
\[\frac{\mu(A\cap w2^\IN)}{2^{-|w|}}\geq1-\frac{1}{2^k}.\]
\end{lemma}

We consider the following multi-valued function as a representative of the Lebesgue Density Lemma 
in the Weihrauch lattice:
\[\LDL:\In\AA_{-}(2^\IN)\times\IN\mto\{0,1\}^*,(A,k)\mapsto\left\{w\in\{0,1\}^*:\frac{\mu(A\cap w2^\IN)}{2^{-|w|}}\geq1-\frac{1}{2^k}\right\}\]
with $\dom(\LDL)=\{(A,k):\mu(A)>0\}$.
It is easy to see that the Lebesgue Density Lemma is equivalent to $\C_\IN$.

\begin{theorem}[Lebesgue Density Lemma] 
\label{thm:LDL-CN}
$\LDL\equivSW\C_\IN$.
\end{theorem}
\begin{proof}
We first prove $\LDL\leqSW\C_\IN$. By $(w_i)$ we denote some effective standard enumeration of $\{0,1\}^*$.
Given a closed set $A\In2^\IN$ with $\mu(A)>0$ and $k\in\IN$ the Lebesgue Density Lemma~\ref{lem:LDL} guarantees that 
\[B:=\left\{i\in\IN:\frac{\mu(A\cap w_i2^\IN)}{2^{-|w_i|}}\geq1-\frac{1}{2^k}\right\}\]
is non-empty and since the measure $\mu:\AA_-(2^\IN)\to\IR$ is upper semi-computable by Lemma~\ref{lem:semi-computable} it follows
that $B$ is co-c.e.\ closed in $A$. Hence, $\C_\IN$ can determine some point $i\in B$ which yields
the desired result $w_i$.

Now we prove $\C_\IN\leqSW\LDL$. We recall that by $\UC_\IN$ we denote choice for singletons $\{n\}$.
In fact, since $\UC_\IN\equivSW\C_\IN$ by \cite[Proposition~3.8]{BGM12}
it suffices to prove $\UC_\IN\leqSW\LDL$. Hence, given a singleton $\{n\}\In\IN$ by an enumeration of its complement,
we need to find the number $n$. Given $\{n\}$ by an enumeration of its complement, we can compute (negative information on) the closed set 
\[A_n:=0^n1\{0,1\}^\IN\cup\{0^\omega\}\In\{0,1\}^\IN\]
and we obtain $\mu(A_n)>0$. Then $\LDL$ will produce upon input of $(A_n,2)$ a word $w\in\{0,1\}^*$ such that
$\mu(A_n\cap w2^\IN)> 2^{-|w|-1}$. In order to ensure this condition, the word $w$ has to have prefix $0^n1$, 
which yields the number $n$.
\end{proof}

By Corollary~\ref{cor:WWKL-Delta-WWKL-CN} we have $\WWKL^\Delta\equivW\WWKL*\C_\IN$.
We can now factorize $\WWKL^\Delta$ also using $\varepsilon\dash\WWKL$ for arbitrarily large $\varepsilon<1$.
As a preparation we prove the following lemma.

\begin{lemma}
\label{lem:epsilon-WWKL-Delta}
$\WWKL\leqW\varepsilon\dash\WWKL*\C_\IN$ for all $\varepsilon\in[0,1)$.
\end{lemma}
\begin{proof}
Let $\varepsilon\in[0,1)$ and let $k\in\IN$ be such that $1-\frac{1}{2^k}>\varepsilon$. 
By Theorem~\ref{thm:LDL-CN} it suffices to prove $\WWKL\leqW\varepsilon\dash\WWKL*\LDL$.
Given a binary tree $T$ with a set $A$ of infinite paths of positive measure, we apply $\LDL$ to $(A,k)$
in order to obtain a $w$ such that $\frac{\mu(A\cap w2^\IN)}{2^{-|w|}}>1-\frac{1}{2^k}>\varepsilon$.
Hence, the subtree $T_w$ of $T$ that starts in node $w$ (i.e., $u\in T_w\iff wu\in T$) has a set $A_w$ of infinite paths of measure
$\mu(A_w)>\varepsilon$ and hence $\varepsilon\dash\WWKL$ yields an infinite path $p_w$ in $A_w$.
Given this path $p_w$ and $w$ we can compute $p=wp_w\in A$.
\end{proof}

Now we can prove the following main result of this section.

\begin{theorem}
\label{thm:WWKL-Delta-WWKL-CN}
$\WWKL^\Delta\equivW\varepsilon\dash\WWKL*\C_\IN$ for all $\varepsilon\in[0,1)$.
\end{theorem}
\begin{proof}
Let $\varepsilon\in[0,1)$.
Since $\varepsilon\dash\WWKL\leqSW\WWKL$, we obtain by Corollary~\ref{cor:WWKL-Delta-WWKL-CN}
\[\varepsilon\dash\WWKL*\C_\IN\leqW\WWKL*\C_\IN\leqW\WWKL^\Delta.\]
On the other hand, Corollary~\ref{cor:WWKL-Delta-WWKL-CN} and Lemma~\ref{lem:epsilon-WWKL-Delta} yield
\[\WWKL^\Delta\leqW\WWKL*\C_\IN\leqW\varepsilon\dash\WWKL*\C_\IN*\C_\IN\leqW\varepsilon\dash\WWKL*\C_\IN.\]
The last reduction follows since $\C_\IN$ is closed under composition, which in turn follows from the obvious closure under composition of the class of functions computable with finitely many mind changes~\cite[Corollary~7.6]{BBP12}.
\end{proof}

This result can also be interpreted such that on real numbers probabilistic choice does not depend on the value
of the probability.

\begin{theorem}
\label{thm:WWKL-Delta-PCR}
$\PC_\IR\equivW\P_{>\varepsilon}\C_\IR$ for all $\varepsilon\geq0$.
\end{theorem}
\begin{proof}
Let us suppose that the claim is true for $\varepsilon=2$. Then we obtain for all $\varepsilon\in[0,1]$ by Proposition~\ref{prop:interval-monotonicity}
\[\PC_\IR\equivW\P_{>2}\C_\IR\leqW\P_{>\varepsilon}\C_\IR\leqW\PC_\IR.\]
This means that it follows that the claim also holds for $\varepsilon\in[0,1]$. What remains is to prove the claim for $\varepsilon>1$ (which includes the case $\varepsilon=2$). 
Hence, let $\varepsilon>1$.
Then there exists $\delta\in(0,1)$ and $n\in\IN$ with $n\geq1$ such that $\delta\cdot n>\varepsilon$.
We note that $\C_\IN\equivSW\P_{=n}\C_\IN$ (which can be proved analogously to $\C_\IN\equivSW\UC_\IN=\P_{=1}\C_\IN$).
With Theorems~\ref{thm:discrete-jump}, \ref{thm:WWKL-Delta-WWKL-CN}, \ref{thm:products}, Propositions~\ref{prop:interval-monotonicity} and \ref{prop:reals} 
and Corollary~\ref{cor:epsilon-WWKL-Cantor} we obtain
\[\PC_\IR\equivW\WWKL^\Delta\equivW\delta\dash\WWKL*\C_\IN\equivW\P_{>\delta}\C_{2^\IN}*\P_{=n}\C_\IN\leqW\P_{>\delta\cdot n}\C_{2^\IN\times\IN}\leqW\P_{>\varepsilon}\C_\IR.\]
The inverse reduction is clear by Proposition~\ref{prop:interval-monotonicity}.
\end{proof}

This result is in sharp contrast to Theorem~\ref{thm:probability-dependency}. While probabilistic choice on Cantor space sensitively depends
on lower bounds on the probability, probabilistic choice on the Euclidean space does not. Intuitively, this is because Cantor space is compact,
whereas Euclidean space offers ``enough space'' to enlarge the measure of sets.

\section{Probability Amplification}
\label{sec:probability-amplification}

Theorem~\ref{thm:probability-dependency} shows that the technique of probability amplification, 
which is well-known from the theory of randomized algorithms \cite{MR95} 
fails for Las Vegas computability over infinite objects. The reason is that we are dealing with infinite
computations and if we run two instances of a probabilistic algorithm with different guesses in parallel, then
we need to decide at some finite time which output we are going to choose. This is simply not possible in general.
The positive content of probability amplification can however be captured in an algebraic way.
In order to express it precisely, we introduce the parallel sum of two degrees.

\begin{definition}[Parallel Sums]
Let $f:\In X\mto Y$ and $g:\In W\mto Z$ be multi-valued functions. Then 
we define the {\em parallel sum} $f+g:X\times W\mto Y\times Z$ by 
\[(f+g)(x,w):=(f(x)\times\range(g))\cup(\range(f)\times g(w))\]
for all $(x,w)\in\dom(f+g):=\dom(f)\times\dom(g)$. 
\end{definition}

The parallel sum $f+g$ captures an operation that takes inputs $x,w$ for both $f$ and $g$
and it produces a pair $(y,z)$ such that $y\in f(x)$ or $z\in g(w)$, i.e., only one of the
two results is guaranteed to be correct.\footnote{We warn the reader that the parallel sum is not monotone for (strong) Weihrauch reducibility and hence it
cannot be considered as an operation on the Weihrauch lattice.} 
We note that for $f,g$ with computable points in the range one obtains $f+g\leqW f\sqcap g$.
The concept of a parallel sum is closely related to the concept of a fraction as introduced in \cite{BLP12}.
We can now express probability amplification with sums as follows.

\begin{proposition}[Probability amplification]
\label{prop:probability-amplification}
Let $X$ and $Y$ be represented spaces with $\sigma$--finite Borel probability measures $\mu_X$ and $\mu_Y$, respectively and
let $a,b\in[0,1]$ and $c:=1-(1-a)(1-b)$. Then
\[\P_{>a}\C_X+\P_{>b}\C_Y\leqSW\P_{>c}\C_{X\times Y}.\]
\end{proposition}
\begin{proof}
Given closed sets $A\In X$ and $B\In Y$ with $\mu_X(A)>a$ and $\mu_Y(B)>b$ 
we can compute $C:=(A\times Y)\cup(X\times B)$ and we obtain $(\mu_X\otimes\mu_Y)(C)>c$.
This yields the reduction.
\end{proof}

In case that $X$ is a space that has a measure preserving pairing mechanism, we can replace $X, Y$ and $X\times Y$ by $X$ in this result
using Corollary~\ref{cor:products-pairing}. In particular, we obtain the following.

\begin{corollary}
Let $a,b\in[0,1]$ and $c:=1-(1-a)(1-b)$. Then
\[a\dash\WWKL+b\dash\WWKL\leqSW c\dash\WWKL.\]
\end{corollary}

We also note that in the situation of Proposition~\ref{prop:probability-amplification} with $X=Y$ one single mind change allows
us to identify the successful one among the two parallel computations, i.e.,
$\P_{>a}\C_X\sqcap\P_{>b}\C_X\equivW\P_{>\max(a,b)}\C_X\leqW\C_2*(\P_{>a}\C_X+\P_{>b}\C_X)$.

\section{Majority Vote}
\label{sec:majority}

In this section we will prove that any suitable single-valued function $f$ below any jump $(\frac{1}{2}\dash\WWKL)^{(n)}$ is computable.
The idea is that a simple majority vote after an exhaustive search will yield the result if more than half of the random advices do the job. 
We will consider single-valued functions $f:X\to Y$ to computable metric spaces $Y$. The majority vote technique
works for these spaces since the consistency of approximations can be recognized. We will make this statement more precise.
If $(X,\delta)$ is a represented space, then we denote by $\FF(X)$ the set of finite subsets of $X$ that is represented in the canonical way by
$\delta_{\FF(X)}$, which is defined by
\[\delta_{\FF(X)}\langle n,p_0,...,p_n\rangle:=\{\delta(p_0),...,\delta(p_n)\}.\]
Moreover, we recall that every represented space $(X,\delta)$ induces a dual represented space $(\OO(X),\delta^\circ)$,
where $\OO(X)$ is the topology of $X$ (i.e., the final topology of $\delta$) and $\delta^\circ(p):=X\setminus\psi_-(p)$
is the representation of open subsets via their characteristic functions to Sierpi\'nski space.
In the following lemma we consider $\IN^\IN$ as a represented space via the identity $\id_{\IN^\IN}$ as representation.

\begin{lemma}[Cauchy representation]
\label{lem:Cauchy}
Let $X$ be a computable metric space. The (suitably defined) Cauchy representation $\delta$ of $X$ satisfies
\begin{enumerate}
\item $\Delta:\OO(\IN^\IN)\to\OO(X), U\mapsto \delta(U)$ is computable,
\item $C:=\{W\in\FF(\IN^*):\bigcap_{w\in W}\delta(w\IN^\IN)\not=\emptyset\}$ is c.e.
\end{enumerate}
We say that a set $W\in\FF(\IN^*)$ is {\em consistent}, if $W\in C$.
\end{lemma}
\begin{proof}
Let $(X,d,\alpha)$ be a computable metric space (see \cite{Wei00} for definitions).
Then we can define a version of the Cauchy representation $\delta$ of $X$ by $\delta(p):=\lim_{n\to\infty}\alpha (p(n))$
with
\[\dom(\delta):=\{p\in\IN^\IN:(\exists x\in X)(\forall i\in\IN)\;d(x,\alpha p(i))<2^{-i}\}.\]
This representation $\delta$ is computably equivalent to other standard versions of the Cauchy representation of $X$ and we obtain
\[\delta(w\IN^\IN)=\bigcap_{i=0}^{|w|-1}B(\alpha w(i),2^{-i})\]
for every $w\in\IN^*$. Secondly, for $U=\bigcup_{w\in W}w\IN^\IN$ with $W\In\IN^*$ we obtain
$\delta(U)=\bigcup_{w\in W}\delta(w\IN^\IN)$.
Since finite intersections and countable unions of open sets are computable, we can conclude that $\Delta$ is computable.
Since
\[\{\UU\in\FF(\OO(X)):\bigcap\UU\not=\emptyset\}\]
is c.e.\ open and $\Delta$ is computable, it follows that $C$ is c.e.
\end{proof}

Now we can prove our main result on majority votes. We point out that the multi-valued function $g$ in the following result
need not be computable. 

\begin{theorem}[Majority vote]
\label{thm:majority}
Let $X$ be a represented space, let $Y$ be a computable metric space and let $f:X\to Y$ be a single-valued function. If 
\[f\leqW\frac{1}{2}\dash\WWKL\circ g\] 
for some $g$, then $f$ is computable.
\end{theorem}
\begin{proof}
Let $(Y,d,\alpha)$ be a computable metric space.
Without loss of generality, we can assume that $\delta$ is the Cauchy representation of $Y$ according to Lemma~\ref{lem:Cauchy}.
Let $f\leqW\frac{1}{2}\dash\WWKL\circ g$ and let $g:\In Z\mto\Tr$. Then there are computable $H,K$, such that
$H\langle\id,GK\rangle$ is a realizer of $f$ whenever $G$ is a realizer of $\frac{1}{2}\dash\WWKL\circ g$.
Given a name $p$ of a point $x\in\dom(f)$, $K(p)$ is a name of a point $z\in\dom(g)$ such that every tree $T\in g(z)$
satisfies $\mu([T])>\frac{1}{2}$ and every infinite path $q\in[T]$ yields a name $H\langle p,q\rangle$ of $f(x)$, i.e., $\delta H\langle p,q\rangle=f(x)$.
We need to prove that there is a computable realizer $F$ of $f$. Upon input of $p$ one can use $H\langle p,q\rangle$
for varying $q$ in order to obtain such a realizer by majority vote. For every $p$ that is a name of a point $x\in\dom(f)$ as above we can
compute a monotone function $h_p:\{0,1\}^*\to\IN^*$ that approximates $q\mapsto H\langle p,q\rangle$.  
Since every $A=[T]$ for binary trees $T$ is a compact set, the function $H$ is uniformly continuous on
$\{p\}\times A$. That means that for every $k\in\IN$ there is an $n\in\IN$ such that for every $q\in A$ it holds that $|h_p(q|_n)|\geq k+2$. 
On the other hand, $\mu(A)>\frac{1}{2}$ and hence there is a finite set $W\in\FF(\{0,1\}^*)$ that consists of more than half of the words in $\{0,1\}^n$
and such that $\{h_p(w):w\in W\}$ is consistent and $|h(w)|\geq k+2$ for all $w\in W$.
Since consistency is a c.e.\ property by Lemma~\ref{lem:Cauchy}, we can find for our given $k$ a suitable $n$ and a corresponding
set $W$ by exhaustive search. As soon as we have found it, we compute an approximation $\alpha(i)\in\bigcap_{w\in W}\delta(h_p(w)\IN^\IN)$ 
of $f(x)$, which is possible since $\delta$ is computably open by Lemma~\ref{lem:Cauchy}.
We claim that $d(\alpha(i),f(x))<2^{-k}$. We note that if $W$ consists of half of the words in $\IN^n$, then
due to the measure condition at least one of the $h_p(w)$ for $w\in W$ has to be a prefix of a correct $\delta$--name $q$ of $f(x)$. 
Hence, for this $w$ we have $f(x)\in\delta(h_p(w)\IN^\IN)$ and hence we obtain for $a:=\alpha(h_p(w)(k+1))$
\[d(\alpha(i),f(x))\leq d(\alpha(i),a)+d(a,f(x))\leq 2^{-k-1}+2^{-k-1}=2^{-k}.\]
If we proceed with the above algorithm for $k=0,1,2,...$ and each fixed given input $p$, then 
we obtain a computable realizer $F$ of $f$ with respect to the representation $\delta$ on the output side.
\end{proof}

In general we obtain $h^{(n)}$ from $h:\In X\mto Y$ by replacing the representation $\delta$ of $X$
by its $n$--fold jump $\delta^{(n)}$ on the input side. Since $h^{(n)}\equivSW h\circ\delta^{(n)}$,
we obtain the following corollary of Theorem~\ref{thm:majority}.

\begin{corollary}[Majority vote]
\label{cor:majority}
Let $X$ be a represented space, let $Y$ be a computable metric space and let $f:X\to Y$ be a single-valued
function. If 
\[f\leqW\frac{1}{2}\dash\WWKL^{(n)}\] 
for some $n\in\IN$, then $f$ is computable.
\end{corollary}

Theorem~\ref{thm:majority} and Corollary~\ref{cor:majority} automatically also hold true for all $\varepsilon>\frac{1}{2}$
instead of $\frac{1}{2}$.
It is easy to see that $\LPO\leqW\varepsilon\dash\WWKL'$ for every $\varepsilon<\frac{1}{2}$ and $\LPO$ is single-valued 
(with the computable metric space $\{0,1\}$ on the output side).
Hence Theorem~\ref{thm:majority} and Corollary~\ref{cor:majority} cannot be generalized to the case of $\varepsilon<\frac{1}{2}$.

Corollary~\ref{cor:majority} analogously holds for $\frac{1}{2}\dash\WWKL^\Delta$ instead of $\frac{1}{2}\dash\WWKL'$.
Since $\LPO\leqW\C_\IN\leqW\WWKL^\Delta$ by Theorem~\ref{thm:WWKL-Delta-WWKL-CN}, we obtain the following.

\begin{corollary}
$\frac{1}{2}\dash\WWKL^\Delta\lW\WWKL^\Delta$.
\end{corollary}

This shows that Theorem~\ref{thm:WWKL-Delta-WWKL-CN} can in general not be improved to the statement that
$\WWKL^\Delta$ is equivalent to $(\varepsilon\dash\WWKL)^\Delta$.
Since $\LPO\leqW\WWKL^\Delta$ is single-valued, we also get the following corollary.

\begin{corollary}
$\WWKL^\Delta\nleqW\frac{1}{2}\dash\WWKL^{(n)}$ for all $n\in\IN$.
\end{corollary}

\section{Probabilistic Degrees}
\label{sec:probabilistic-degrees}

In this section we would like to capture the non-uniform content of probabilistic computability. 
Intuitively, we want to call a Weihrauch degree {\em probabilistic}, if it can be computed with some random advice,
irrespectively of any failure recognition mechanisms that Las Vegas machines come equipped with.
A suitable notion of random advice has already been introduced and studied by the first author and Arno Pauly in \cite{BP10}.
We repeat the definition for our setting.

\begin{definition}[Probabilistic degrees]
Let $(X,\delta_X)$, $(Y,\delta_Y)$ be represented spaces. A multi-valued function $f:\In X\mto Y$ is called
{\em probabilistic}, if there exists a computable function $F:\In\IN^\IN\times2^\IN\to\IN^\IN$ 
such that $\mu(\{r\in2^\IN:\delta_YF(p,r)\in f\delta_X(p)\})>0$ for all $p\in \dom(f\delta_X)$.
A Weihrauch degree is called {\em probabilistic}, if it has a probabilistic member.
\end{definition}

We emphasize that the condition in this definition implies that the sets $A_p:=\{r\in2^\IN:\delta_YF(p,r)\in f\delta_X(p)\}$ 
have to be measurable, but they are not required to be closed and they do not need to depend on $p$ in any uniform way (in contrast to the
sets $S_p$ in Definition~\ref{def:Las-Vegas}).
It follows from results below that Cantor space $2^\IN$ could be equivalently replaced by Baire space $\IN^\IN$ in the above definition. 
The following characterization of probabilistic degrees follows from \cite[Theorem~11]{BP10}.
We also give a direct proof.

\begin{proposition}[Probabilistic degrees]
\label{prop:probabilistic}
A multi-valued function $f$ on represented spaces is probabilistic, if and only if $f\leqW g$ for some $g:\In\IN^\IN\mto2^\IN$ 
such that $\mu_{2^\IN}(g(p))>0$ for all $p\in\dom(g)$. 
\end{proposition}
\begin{proof}
Let $f:\In X\mto Y$ be a multi-valued function on represented spaces $(X,\delta_X)$ and $(Y,\delta_Y)$
and let $g:\In\IN^\IN\to2^\IN$ be as stated above.
Let $f\leqW g$ be witnessed by computable $H,K$. 
We define $A_p:=gK(p)$ for all $p\in D:=\dom(f\delta_X)$.
Without loss of generality, we can assume that $H$ has minimal domain, i.e., $\dom(H)=\{\langle p,r\rangle:p\in D,r\in A_p\}$.
We define $F:\In\IN^\IN\times2^\IN\to\IN^\IN$ by $F(p,r):=H\langle p,r\rangle$.
Then $F$ is computable and by assumption we have $\mu_{2^\IN}(A_p)>0$ 
and $\delta_Y F(p,r)=\delta_YH\langle p,r\rangle\in f\delta_X(p)$ for all $p\in D$ and $r\in A_p$.
Hence, $f$ is probabilistic. If, on the other hand, $f$ is probabilistic, then there is a computable
$F:\In\IN^\IN\times2^\IN\to\IN^\IN$ such that $A_p:=\{r\in2^\IN:\delta_YF(p,r)\in f\delta_X(p)\}$
satisfies $\mu_{2^\IN}(A_p)>0$ for all $p\in D:=\dom(f\delta_X)$. If we define
$g:\In\IN^\IN\mto2^\IN$ by $g(p):=A_p$ for all $p\in D$, then we obtain $f\leqW g$. 
\end{proof}

It is clear that it follows from this proposition that probabilistic degrees are closed downwards with respect to Weihrauch reduction.

\begin{proposition}
\label{prop:downwards-probabilistic}
If $f\leqW g$ and $g$ is probabilistic, then $f$ is probabilistic.
\end{proposition}

The case of Baire space can be reduced to Cantor space in the non-uniform setting due to the following lemma.

\begin{lemma}[Embedding of Baire into Cantor space]
\label{lem:Baire-Cantor}
The following map is a computable embedding with a measure preserving inverse:
\[\iota:\IN^\IN\to2^\IN,p\mapsto 1^{p(0)}01^{p(1)}01^{p(2)}... .\]
\end{lemma}
\begin{proof}
It is clear that $\iota$ and its partial inverse are computable. For $w\in\IN^*$ we obtain
\[\mu_{2^\IN}(\iota(w\IN^\IN))=2^{-\sum_{i=0}^{|w|-1}(w(i)+1)}=\prod_{i=0}^{|w|-1}2^{-w(i)-1}=\mu_{\IN^\IN}(w\IN^\IN).\]
By Lemma~\ref{lem:identity} this proves that $\mu_{2^\IN}(\iota(A))=\mu_{\IN^\IN}(A)$ holds for all measurable $A\In\IN^\IN$.
\end{proof}

We point out that this embedding does not have a closed range and it does not preserve closedness.
Hence it cannot be used in a uniform setting. 
The following lemma provides a fully uniform reduction in the opposite direction for sets.
We use the {\em signum function} $\sgn:\IN\to\IN$, defined by $\sgn(0):=0$ and $\sgn(n+1)=1$ for all $n\in\IN$ and 
its extension $\sgn:\IN^\IN\to2^\IN$, defined by $\sgn(p)(n):=\sgn(p(n))$ for all $p\in\IN^\IN$ and $n\in\IN$.

\begin{lemma}[Signum]
\label{lem:power-Cantor-Baire}
The map 
\[J:2^{(2^\IN)}\to2^{(\IN^\IN)},A\mapsto\sgn^{-1}(A)\]
has the property that $\mu_{\IN^\IN}(J(A))=\mu_{2^\IN}(A)$ for all 
measurable $A\In 2^\IN$ and its restriction $J:\AA_-(2^\IN)\to\AA_-(\IN^\IN)$ to closed sets is computable.
\end{lemma}
\begin{proof}
Since $\sgn:\IN^\IN\to2^\IN$ is computable, it follows that $J$ maps measurable sets 
to measurable sets and closed sets to closed sets and that the restriction of $J$ to closed sets 
is computable. The claim on the measure can be proved by induction. We claim that $\mu_{\IN^\IN}(J(w2^\IN))=\mu_{2^\IN}(w\IN^\IN)$
holds for all $w\in\{0,1\}^*$. This is clear for the empty word $w$. Suppose it holds for a given word $w$. Then we obtain
\[\mu_{\IN^\IN}(J(1w2^\IN))=\sum_{i=1}^\infty2^{-i-1}\mu_{\IN^\IN}(J(w2^\IN))=\frac{1}{2}\mu_{2^\IN}(w2^\IN)=\mu_{2^\IN}(1w2^\IN)\]
and likewise $\mu_{\IN^\IN}(J(0w2^\IN))=\mu_{2^\IN}(0w2^\IN)$. Hence the claim follows by structural induction.
This implies that $\mu_{\IN^\IN}(J(A))=\mu_{2^\IN}(A)$ for all measurable $A\In2^\IN$ by Lemma~\ref{lem:identity}.
\end{proof}

Lemma~\ref{lem:Baire-Cantor} and \ref{lem:power-Cantor-Baire} show that we
could equivalently use functions $g:\In\IN^\IN\mto\IN^\IN$ in Proposition~\ref{prop:probabilistic}.
The following result, which follows from Proposition~\ref{prop:probabilistic} shows that being probabilistic is a necessary criterion for being 
probabilistically computable in any sense that we consider here.

\begin{theorem}[Las Vegas computability and probabilistic degrees]
\label{thm:probabilistic}
If there is a~$g$ with ${f\leqW\PC_{\IN^\IN}\circ g}$, then $f$ is probabilistic.
\end{theorem}

By Lemma~\ref{lem:power-Cantor-Baire} and due to the fact that the signum function $\sgn:\IN^\IN\to2^\IN$ is computable,
we obtain $\PC_{2^\IN}\leqSW\PC_{\IN^\IN}$.
By a slight variant of this result we also obtain $\PC_{\IN\times2^\IN}\leqSW\PC_{\IN^\IN}$ and hence
\[\PC_\IN\leqSW\PC_{\IN\times2^\IN}\leqSW\PC_{\IN^\IN}.\] Altogether, we get the following corollary.

\begin{corollary}[Probabilistic degrees]
\label{cor:probabilistic}
If $f\leqW\PC_R^{(n)}$ for some $n\in\IN$ and $R$ is among $\IN, 2^\IN, \IN\times2^\IN$ or $\IN^\IN$, then $f$ is probabilistic.
\end{corollary}

The core observation in this context is that Weak K\H{o}nig's lemma is not probabilistic~\cite[Theorem~20]{BP10}.

\begin{proposition}
\label{prop:WKL-probabilistic}
$\WKL$ is not probabilistic.
\end{proposition}

The proof is essentially based on an earlier result of Jockusch and Soare (see \cite[Theorem~5.3]{JS72}) which shows that the set
of those points in Cantor space from which one can compute a separating set for two given disjoint c.e.\ sets that are computably inseparable has measure $0$.
Since the separation problem is equivalent to Weak K\H{o}nig's Lemma by \cite[Theorem~6.7]{GM09}, one obtains Proposition~\ref{prop:WKL-probabilistic}.

As a consequence it follows with Proposition~\ref{prop:downwards-probabilistic} that everything above $\WKL$ is also not probabilistic.
This applies, in particular, to the limit map $\lim$.
Corollary~\ref{cor:probabilistic} and Proposition~\ref{prop:WKL-probabilistic} yield the following.

\begin{corollary}
\label{cor:WKL-probabilistic}
$\WKL\nleqW\PC_{\IN^\IN}^{(n)}$ and, in particular, $\WKL\nleqW\WWKL^{(n)}$ for every $n\in\IN$.
\end{corollary}

An important classical result that yields further insights into non-uniform randomized computations is the
Theorem of Sacks (see \cite{Sac63} and for recent presentations see \cite[Theorem~5.1.12]{Nie09}, \cite[Corollary~8.12.2]{DH10}).

\begin{theorem}[Sacks 1963]
\label{thm:Sacks}
Let $A\In2^\IN$ be a set such that $\mu(A)>0$ and let $q\in\IN^\IN$ be such that $q\leqT r$ for every $r\in A$.
Then $q$ is computable.
\end{theorem}

In a certain sense Theorem~\ref{thm:majority} captures the uniform content of the Theorem of Sacks. 
An early predecessor of the Theorem of Sacks is the Theorem of de Leeuw, Moore, Shannon and Shapiro \cite{dLMSS56}
that makes a corresponding statement with respect to c.e.\ sets (see also \cite[Theorem~8.12.1]{DH10}).

\begin{theorem}[de Leeuw, Moore, Shannon and Shapiro 1956]
\label{thm:dLMSS}
Let $A\In2^\IN$ be a set such that ${\mu(A)>0}$ and let $B\In\IN$ be such that $B$ is c.e.\ in $r$ for every $r\in A$.
Then $B$ is c.e.
\end{theorem}

The Theorem of de Leeuw, Moore, Shannon and Shapiro also holds true in a relativized version that states that
if $\mu(A)>0$ and $B$ is c.e.\ in $p\oplus r$ for every $r\in A$, then $B$ is c.e.\ in $p$.
Here we use this relativized version in order to prove the following result that shows that certain single-valued probabilistic functions 
map computable inputs to computable outputs.\footnote{We thank Mathieu Hoyrup for pointing out that the Theorem of de Leeuw, Moore, Shannon and Shapiro~\ref{thm:dLMSS}
can be used to generalize our initial version of Theorem~\ref{thm:single-valued-probabilistic}, which was originally only formulated for certain metric spaces $Y$. At the same time its use simplified the proof.}

In fact, using the notion of representation reducibility introduced by Joseph Miller \cite{Mil04},
one can even express a stronger result. We recall that for represented spaces $(X,\delta_X)$ and $(Y,\delta_Y)$ and $x\in X$ and $y\in Y$ we say
that $y$ is {\em representation reducible} to $x$, if there exists a partial computable function $g:\In X\to Y$ such that $g(x)=y$.
In symbols this is denoted by $y\leqr x$.\footnote{This can also be rephrased such that $\delta_Y^{-1}(y)$ is Medvedev reducible to $\delta_X^{-1}(x)$.
Miller proved that for computable metric spaces $X$ and $Y$ this is equivalent to $\delta_Y^{-1}(y)$ being Muchnik reducible to $\delta_X^{-1}(x)$,
see \cite[Corollary~4.3]{Mil04}.}
We recall that a $T_0$--space $X$ with countable basis $(U_i)_i$ has a standard representation $\delta_X$ given
by $\delta_X(p)=x:\iff\{n\in\IN:n+1\in\range(p)\}=\{n\in\IN:x\in U_n\}$ (see \cite{Wei00}). In other words,
$p$ is a name of $x$ if it encodes a list of all basic properties $U_n$ of $x$.

\begin{theorem}[Single-valued probabilistic degrees]
\label{thm:single-valued-probabilistic}
Let $X$ be a represented space and let $Y$ be a $T_0$\nobreakdash-space with countable base and standard representation.
If a single-valued function $f:X\to Y$ is probabilistic, then $f$ maps computable inputs to computable outputs.
In fact, even $f(x)\leqr x$ holds for all $x\in X$ in this case.
\end{theorem}
\begin{proof}
Let $\delta_X$ be the representation of $X$ and let $\delta_Y$ be a standard representation of $Y$ with respect
to a countable base $(U_i)_i$.
Let $f:X\to Y$ be single-valued and probabilistic. 
Then there is a computable function $F:\In\IN^\IN\times2^\IN\to\IN^\IN$ such that the sets
$A_p:=\{r\in2^\IN:\delta_YF(p,r)= f\delta_X(p)\}$ satisfy
$\mu(A_p)>0$ for all $p\in D:=\dom(f\delta_X)$.
We fix some $\delta_X$--name $p$ of some $x\in X$.
Then $F(p,r)$ is a $\delta_Y$--name of $f(x)$, 
i.e., $F(p,r)$ is an enumeration of $B:=\{n\in\IN:f(x)\in U_n\}$.
Hence, $B$~is~c.e.\ in $p\oplus r$ for every $r\in A_p$. 
By the relativized version of the Theorem of de Leeuw, Moore, Shannon and Shapiro~\ref{thm:dLMSS}
we obtain that $B$~is~c.e.\ in $p$. This means that $\delta_Y^{-1}(f(x))$ is Muchnik
reducible to $\delta_X^{-1}(x)$ and hence $f(x)\leqr x$ by \cite[Corollary~4.3]{Mil04}.
In particular, $f$ maps computable inputs to computable outputs.
\end{proof}

\section{Zeros of Continuous Functions with Sign Changes}
\label{sec:IVT}

In this section we would like to prove that there is no Las Vegas algorithm for computing zeros of continuous functions with changing signs 
(not even one with additional finitely many mind changes).
There is, however, a probabilistic algorithm of second order for finding such zeros.

By $\IVT$ we denote the Intermediate Value Theorem, i.e., the problem: given a continuous function $f:[0,1]\to\IR$ that changes its sign (i.e., $f(0)\cdot f(1)<0$),
find an $x\in [0,1]$ with $f(x)=0$. In \cite[Theorem~6.2]{BG11a} it was proved that $\IVT\equivSW\ConC_{[0,1]}$,
where $\ConC_{[0,1]}$ denotes choice for closed sets $A\In[0,1]$ restricted to connected sets (i.e., closed intervals).
We now prove that $\ConC_{[0,1]}$ cannot be reduced to probabilistic choice $\PC_{[0,1]}$.
For that purpose we use a finite extension construction to obtain a name of a connected closed set that is mapped
to a set of measure zero (given a potential reduction).

\begin{proposition}
\label{prop:CC-PC}
$\ConC_{[0,1]}\nleqW\PC_{[0,1]}$.
\end{proposition}
\begin{proof}
By Proposition~\ref{prop:WWKL} it suffices to show $\ConC_{[0,1]}\nleqW\WWKL$.
Let us assume for a contradiction that $\ConC_{[0,1]}\leqW\WWKL$. 
Then there are computable $H,K$ such that $H\langle \id,FK\rangle$ is a realizer of $\ConC_{[0,1]}$
for every realizer $F$ of $\WWKL$. 
Without loss of generality, we can assume as usual that $H,K$ have minimal domains. 
Let now $p_0$ be a name of the unit interval $I_0:=[0,1]$.
Then $K(p_0)$ is a name of a tree $T_0$ such that the set of infinite paths $A_0=[T_0]$ satisfies $\mu(A_0)>0$.
Since $H$ is uniformly continuous on the compact set $\{p_0\}\times A_0$, it follows that there is a finite prefix $w_0\prefix p_0$
such that $H\langle w_0\IN^\IN,q\rangle$ produces its output in $[0,1]$ with precision smaller than $\frac{1}{3}$
(i.e., any two outputs $x,x'$ named in that set satisfy $|x-x'|<\frac{1}{3}$),
uniformly for all names $q$ of points in $A_0$. 
Since $K$ is continuous, we can also assume that $w_0$ is long enough such that $K(w_0\IN^\IN)$ only contains
names of trees with closed sets $B\In2^\IN$ of infinite paths with $\mu(B\setminus A_0)\leq\frac{1}{2}(1-\mu(A_0))$.
Now there are names $r_{0}$ and $r_{1}$ of the intervals
$J_0:=[0,\frac{1}{3}]$ and $J_1:=[\frac{2}{3},1]$, respectively such that $w_0\prefix r_0$ and $w_0\prefix r_1$.
Hence $K(r_0)$ and $K(r_1)$ are names of infinite trees $S_0$ and $S_1$ with sets of infinite paths $B_0=[S_0]$
and $B_1=[S_1]$, respectively. We claim that $A_0\cap B_0\cap B_1=\emptyset$. For if $q\in A_0\cap B_0\cap B_1$,
then $H\langle r_0,q\rangle$ and $H\langle r_1,q\rangle$ have to be names of points $x_0,x_1\in[0,1]$, respectively
with $|x_0-x_1|<\frac{1}{3}$ according to the choice of $w_0$. On the other hand, there is a realizer $F$ of $\WWKL$
with $F(r_0)=F(r_1)=q$ and thus $x_0\in J_0=[0,\frac{1}{3}]$
and $x_1\in J_1:=[\frac{2}{3},1]$, which is impossible. 
Hence $A_0\cap B_0\cap B_1=\emptyset$ and there is an $i\in\{0,1\}$ with $\mu(B_{i}\cap A_0)\leq\frac{1}{2}\mu(A_0)$
and we obtain
\[\mu(B_{i})=\mu(B_{i}\cap A_0)+\mu(B_{i}\setminus A_0)\leq\frac{1}{2}\mu(A_0)+\frac{1}{2}(1-\mu(A_0))=\frac{1}{2}.\]
We now choose $I_1:=J_{i}$, $p_1:=r_{i}$ and $A_1:=B_{i}$ and we continue the construction inductively. 
If at stage $n>0$ the interval $I_n$ with $\diam(I_n)=\frac{1}{3^n}$, the sequence $p_n$ and a set $A_n$ with $\mu(A_n)\leq\frac{1}{2^{n}}$ have been determined, 
then we continue as follows.
There exists a prefix $w_n\prefix p_n$ such that $H$ on $w_n$ guarantees precision $\frac{1}{3^{n+1}}$ 
and $K$ guarantees measure $\mu(B\setminus A_n)\leq\frac{1}{2}(\frac{1}{2^{n}}-\mu(A_{n}))$ in an analogous way as above.
We select the left and right third $J_0$ and $J_1$ of $I_n$ together with corresponding names $r_0$ and $r_1$, respectively,
such that $w_n\prefix r_0$ and $w_n\prefix r_1$. By $B_0, B_1$ we denote the sets named by $K(r_0)$, $K(r_1)$, respectively.
As above, $A_n\cap B_0\cap B_1=\emptyset$ and there exists $i\in\{0,1\}$ such that $\mu(B_i\cap A_n)\leq\frac{1}{2}\mu(A_n)$.
Analogously to above, we obtain $\mu(B_i)\leq\frac{1}{2}\mu(A_n)+\frac{1}{2}(\frac{1}{2^{n}}-\mu(A_n))=\frac{1}{2^{n+1}}$.
For the next stage we choose $I_{n+1}:=J_i$, $p_{n+1}:=r_i$ and $A_{n+1}:=B_i$. 

Altogether, this construction yields a strictly increasing sequence $w_n$ of prefixes of names $p_n$ of closed intervals $I_n$ with $\diam(I_n)=\frac{1}{3^{n}}$
that are mapped to names $K(p_n)$ of closed sets $A_n$ with $\mu(A_n)\leq\frac{1}{2^{n}}$.
These names $p_n$ converge to a name $p$ of a singleton interval $\{x\}\In[0,1]$
and by continuity of $K$ this name $p$ is mapped to a name $K(p)$ of a tree $T$ with a set $A=[T]$ of infinite paths such that $\mu(A)=0$.
This is a contradiction to the assumption. 
\end{proof}

It is easy to see that the inverse reduction is also not possible.
This is because $\ConC_{[0,1]}$ has some computable outputs for any
computable input and $\PC_{[0,1]}$ does not 
(for instance a universal Martin-L\"of test yields examples of co-c.e.\ closed subsets $A\In[0,1]$ of positive measure without computable points).

\begin{proposition}
\label{prop:PC-CC}
$\PC_{[0,1]}\nleqW\ConC_{[0,1]}$.
\end{proposition}

This can also be proved in a topological way that does not refer to computable inputs and outputs and such a proof will
follow from Proposition~\ref{prop:C2xAC-PCC}.
Altogether we obtain that connected and probabilistic choice on $[0,1]$ are incomparable.

\begin{corollary}[Connected and probabilistic choice]
\label{cor:CC-PC}
$\ConC_{[0,1]}\nW\PC_{[0,1]}$.
\end{corollary}

This means, in particular,  that there is no Las Vegas algorithm that computes zeros of continuous functions with changing sign.
Next we want to show that it does not help to have finitely many mind changes additionally to a Las Vegas algorithm or, 
in other words, there is also no Las Vegas algorithm over the probability space $\IN\times2^\IN$.
This follows from Corollary~\ref{cor:CC-PC} using the following choice elimination principle
that was proved in \cite{LP13} by Le Roux and Pauly and is based on the Baire category technique introduced in \cite{BG11a}.

\begin{theorem}[Discrete choice elimination]
\label{thm:CN-elimination}
$f\leqW g*\C_\IN$ implies $f\leqW g$ for total fractals $f$.
\end{theorem}

We recall that a {\em fractal} $f$ is a multi-valued function on represented spaces such that there exists
a $g:\In\IN^\IN\mto\IN^\IN$ with non-empty domain and such that $g|_A\equivW f$ for every clopen $A\In\IN^\IN$ 
for which $A\cap\dom(g)$ is non-empty. Moreover, $f$ is called a {\em total fractal} if $g$
can be chosen to be total. It is implicit in the proof of \cite[Proposition~4.9]{BG11a} that $\ConC_{[0,1]}$
is a total fractal and we include the proof here for completeness. 

\begin{lemma}
\label{lem:CC-fractal}
$\ConC_{[0,1]}$ is a total fractal.
\end{lemma}
\begin{proof}
In \cite[Proposition~3.6]{BG11a} it was proved that $\ConC_{[0,1]}\equivW\B_I$, where 
\[\B_I:\In\IR_<\times\IR_>\mto\IR,(a,b)\mapsto [a,b]\]
is the boundedness principle that maps a left real number $a$ and a right real number $b$ 
with $a\leq b$ to the closed interval $[a,b]$, i.e., $\dom(\B_I)=\{(a,b)\in\IR_<\times\IR_>:a\leq b\}$.
It is easy to see that $\dom(\B_I)$ is co-c.e.\ closed. 
We assume that $\IR_<$ and $\IR_>$ are represented by $\rho_<$ and $\rho_>$, which are the
representations of reals by strictly increasing and decreasing sequences of rational numbers, respectively. 
We represent $\IR$ by the usual Cauchy representation $\rho$.
Let $G:\In\IN^\IN\mto\IN^\IN$ be the function that maps every pair $\langle p,q\rangle\in\IN^\IN$
to all names of $r$ with $\rho_<(p)\leq\rho(r)\leq\rho_>(q)$. By definition this $G$ is 
partial, but since the domain is co-c.e.\ closed, one can easily extend it to an equivalent total
$F:\In\IN^\IN\mto\IN^\IN$: as soon as a prefix $\langle p|_{n+1},q|_{n+1}\rangle$ of the input is inconsistent, since
the represented sequences of rational numbers are either not increasing in case of $p$ or not decreasing in case of $q$
or since the rational number given by $p(n)$ is not smaller than the one given by $q(n)$, $F$ just maps $\langle p,q\rangle$
to the value $G\langle p',q'\rangle$ where $p'$ and $q'$ are some canonical consistent extensions of $p|_{n}$ and $q|_{n}$, respectively. 
This guarantees that $F\equivW G\equivW \B_I\equivW\ConC_{[0,1]}$. 
Moreover, it is easy to see that $F|_A\equivW F$ for every non-empty clopen $A\In\IN^\IN$. 
It suffices to consider $A$ of the form $A=w\IN^\IN$ with $w\in\IN^*$.
Given a prefix $w$ of a $[\rho_<,\rho_>]$--name of an interval $[a,b]$, the interval is
only described by $w$ up to some rational numbers $c,d$ with $c<a\leq b<d$.
Now any other given interval $[a',b']$ can be mapped by a computable affine transformation
$T:\IR\to\IR$ to an interval $[T(a'),T(b')]$ with $c<T(a')\leq T(b')<d$ and given a
point $x\in[T(a'),T(b')]$ one can easily recover a point $T^{-1}(x)\in[a',b']$. 
This proves $F_{w\IN^\IN}\equivW F$.   
Altogether, this shows that $\ConC_{[0,1]}$ is a total fractal.
\end{proof}

Using Corollary~\ref{cor:CC-PC} we obtain the following result.

\begin{theorem}
\label{thm:CC-PCR}
$\ConC_{[0,1]}\nW\PC_\IR$.
\end{theorem}
\begin{proof}
Since $\PC_{[0,1]}\leqW\PC_\IR$ we obtain $\PC_\IR\nleqW\ConC_{[0,1]}$ by Proposition~\ref{prop:PC-CC}.
Now we prove $\ConC_{[0,1]}\nleqW\PC_\IR$. 
For one we have $\PC_\IR\equivW\PC_{2^\IN}*\C_\IN\equivW\PC_{[0,1]}*\C_\IN$ due to Corollaries~\ref{cor:PCN2N}, \ref{cor:interval-cantor}
and Proposition~\ref{prop:reals}.
If we assume for a contradiction $\ConC_{[0,1]}\leqW\PC_\IR$, then $\ConC_{[0,1]}\leqW\PC_{[0,1]}*\C_\IN$
follows and hence $\ConC_{[0,1]}\leqW\PC_{[0,1]}$ follows with Theorem~\ref{thm:CN-elimination} 
on discrete choice elimination since $\ConC_{[0,1]}$ is a total fractal by Lemma~\ref{lem:CC-fractal}. 
This contradicts Proposition~\ref{prop:CC-PC}.
\end{proof}

On the other hand, it is easy to see that there is a probabilistic algorithm of second order that
can compute zeros of functions with sign change. We first give an intuitive description of this
algorithm before we formulate the result using choice:

\begin{enumerate}
\item A continuous function $f:[0,1]\to\IR$ with $f(0)\cdot f(1)<0$ is given as input.
\item Guess a binary sequence or, equivalently, a bit $b\in\{0,1\}$ and a point $x\in[0,1]$.
\item Interpret the guess $b=1$ such that the zero set $f^{-1}\{0\}$ contains no open intervals
        and use the trisection method to compute a zero $z\in[0,1]$ with $f(z)=0$ in this case (disregarding $x$).
\item Interpret the guess $b=0$ such that the zero set $f^{-1}\{0\}$ does contain an open interval
        and check whether $f(x)=0$ in this case. Stop after finite time if this test fails and output $x$ otherwise.
\end{enumerate}

This is not a Las Vegas algorithm, since the failure of the algorithm in case of $b=0$ cannot be recognized computably.
However, the algorithm succeeds with a positive probability in any case since $x$ is disregarded in case $b=0$ and
there is a set of successful guesses of positive measure in case $b=1$.
Additionally, even in case $b=0$ it is not too difficult to recognize failure, even though it is not computable.
We prove that this algorithm is probabilistic of second order and for simplicity we express this using choice again.

\begin{proposition}
\label{prop:ConC-PC}
$\ConC_{[0,1]}\leqW\PC_{[0,1]}'$.
\end{proposition}
\begin{proof}
By Proposition~\ref{prop:WWKL} it suffices to prove $\ConC_{[0,1]}\leqW\WWKL'$.
Given a name $p$ of a closed interval $I\In[0,1]$ we compute a sequence $(T_n)_n$
of trees that converges to a tree $T$ with a set $A:=[T]\In\{0,1\}^\IN$ of infinite paths with $\mu(A)>0$. 
Firstly, we can compute a tree $T_I\In0\{0,1\}^*$ such that the binary representation $\rho_2$ maps $[T_I]$ to $\frac{1}{2}I$.
By $m_n$ we denote the measure of the approximation of $I$ that is determined by $p|_n$, the prefix
of $p$ of length $n$. Without loss of generality, we can assume that 
we can compute $m_n$ as a positive rational number.
Now we can compute a sequence $(T_n)_n$ of trees
\[T_n:=T_I\cup1\{0,1\}^k\]
where $k$ is maximal with $m_n<2^{-k}$. Then the sequence $(T_n)_n$ clearly 
converges to a tree $T$ with
\[ [T]=\left\{\begin{array}{ll}
  [T_I] & \mbox{if $I$ is not a singleton}\\
  {[T_I]} \cup 1\{0,1\}^\IN & \mbox{otherwise}
\end{array}\right.\]
Let us denote by $K$ the computable map that maps $p$ to a sequence of names of the trees $(T_n)_n$.
Given an infinite path $q\in[T]$ and the original name $p$ we can reconstruct a point $x\in I$ as follows:
if the path $q$ starts with $q(0)=0$, then we compute $x=2\cdot\rho_2(q)$ and if the path $q$
starts with $q(0)=1$, then we know that $I$ must be a singleton and we use $p$ to compute $x$ with $I=\{x\}$.
This describes a computable function $H$ such that $H\langle\id,FK\rangle$ is a realizer of $\ConC_{[0,1]}$
whenever $F$ is a realizer of $\WWKL'$.
\end{proof}

One can also ask whether Proposition~\ref{prop:ConC-PC} can be strengthened to the statement $\ConC_{[0,1]}\leqSW\PC_{[0,1]}'$.
However, for mere cardinality reasons this is not possible.
Since there cannot be an uncountable number of pairwise disjoint sets $A\In[0,1]$ of positive measure
it follows that $\#\PC_{[0,1]}^{(n)}=|\IN|$ (see Proposition~\ref{prop:cardPICR}), while obviously $\#\ConC_{[0,1]}=|\IR|$.
We obtain the following consequence of Proposition~\ref{prop:cardinality}.

\begin{corollary}
\label{cor:CC-PC2}
$\ConC_{[0,1]}\nleqSW\PC_{[0,1]}^{(n)}$ for all $n\in\IN$.
\end{corollary}

Another interpretation of our results is that Weak Weak K\H{o}nig's Lemma does not compute the Intermediate
Value Theorem (nor the other way around), but the jump of Weak Weak K\H{o}nig's Lemma does. 

\begin{corollary}[Intermediate Value Theorem and Weak Weak K\H{o}nig's Lemma]
\label{cor:IVT-WWKL}
$\IVT\nW\WWKL$ and $\IVT\leqW\WWKL'$.
\end{corollary}

We note that the results of this section also yield another proof of the fact that $\WKL\nleqW\WWKL$ (see Corollary~\ref{cor:WKL-probabilistic}), 
since $\IVT\leqW\WKL$.

\section{Robust Division}
\label{sec:RDIV}

In this section we would like to prove that there is a Las Vegas algorithm for robust division
\[\RDIV:[0,1]\times[0,1]\mto[0,1],(x,y)\mapsto\left\{\begin{array}{ll}
  \{\frac{x}{\max(x,y)}\} & \mbox{if $y\not=0$}\\
   {[0,1]} & \mbox{otherwise}
\end{array}\right.\]
Robust division can be used for solving linear equations and inequalities in compact domains
and it has been defined and studied by Arno Pauly \cite{Pau10,Pau11}.
For instance it is easy to see that robust division can be used to find solutions of 
linear equations $ax=b$ for $a,b\in\IR$ in a compact domain.
Robust division is related to all-or-unique choice that we define next.

\begin{definition}[All-or-unique choice]
Let $X$ be a represented space. Then by $\AUC_X$ we denote the {\em all-or-unique choice} operation
of $X$, which is $\C_X$ restricted to
\[\dom(\AUC_X):=\{X\}\cup\{\{x\}:x\in X\},\]
i.e., the entire set or singletons.
\end{definition}

One readily verifies the following result (see \cite[Proposition~5.2.1.3]{Pau11}).

\begin{proposition}[Robust division]
\label{prop:RDIV-AUC}
$\RDIV\equivSW\AUC_{[0,1]}$.
\end{proposition}

Now we will further study the relation of all-or-unique choice $\AUC_{[0,1]}$ to probabilistic choice.

\begin{theorem}[All-or-unique and probabilistic choice]
\label{thm:AC-PCC}
$\AUC_{[0,1]}\lW\PCC_{[0,1]}$.
\end{theorem}
\begin{proof}
Given a name $p$ of a set $A\In[0,1]$, which is either the whole interval or a singleton,
we compute the name $q$ of a proper closed interval $I\In[0,1]$ as follows:
as long as $p$ does not contain any negative information (i.e., $p$ is still compatible with a name
of the full interval), we just copy $p$ to $q$. If the first negative information in $p$ appears 
at position $n$, then we continue to read $p$ until we know the singleton $A=\{x\}$ given by it
up to precision $2^{-n}$. At that point we have rational numbers $a,b$ with $x\in[a,b]$ and
$b-a\leq2^{-n}$ and we just extend the output $q$ to a name of the interval $I=[a,b]$.
This describes a computable function $K$ that maps $p$ to $q$.
The output produced by $K$ is a name for an interval of the form $I=[a,b]$,
where $I=[0,1]$ if and only if $p$ is a name of $[0,1]$.
Given a name of a point $y\in I$ and the original input $p$, we can recover a point $x\in A$ as follows:
we read $p$ and produce an approximation of $y$ up to precision $2^{-n}$ as long as
$p|_n$ does not contain any negative information. In the moment where we find some negative
information in $p$, we stop using $y$ and we just compute an output $x$ with $A=\{x\}$ 
by inspection of $p$. 
Note that this is always possible, since the approximation of $y$ that we have produced so far can 
always be extended to $x$.
This describes a computable function $H$ such that
$H\langle\id,FK\rangle$ is a realizer of $\AUC_{[0,1]}$ whenever $F$ is a realizer of $\PCC_{[0,1]}$. 

It is clear that $\PCC_{[0,1]}\nleqW\AUC_{[0,1]}$: while $\AUC_{[0,1]}\leqW\LPO$ can be computed
with one mind change,  $\C_2^*\leqW\PCC_{[0,1]}$ (which holds by \cite[Proposition~7.2]{BLP12})
implies that $\PCC_{[0,1]}$ cannot be computed
with any finite number of mind changes.
\end{proof}

In other words, we have proved that robust division can be reduced to Weak Weak K\H{o}nig's Lemma.

\begin{corollary}
\label{cor:RDIV-WWKL}
$\RDIV\lW\WWKL$.
\end{corollary}

This means that there is a Las Vegas algorithm for robust division.
Now one can ask whether there is a Las Vegas algorithm for robust division with a fixed positive success probability.
We will show that this is not the case and we start with an observation that follows from Lemma~\ref{lem:join-irreducible}.

\begin{corollary}
\label{cor:AC-irreducible}
$\AUC_{[0,1]}$ is join-irreducible.
\end{corollary}

Now we can conclude that no fixed positive success probability is sufficient for robust division.

\begin{theorem}
\label{thm:AC-*-WWKL}
$\AUC_{[0,1]}\nleqW*\dash\WWKL$.
\end{theorem}
\begin{proof}
Let us assume for a contradiction that $\AUC_{[0,1]}\leqW*\dash\WWKL$.
Since $\AUC_{[0,1]}$ is join-irreducible by Corollary~\ref{cor:AC-irreducible}, we obtain
that there exists an $n\in\IN$ with $\AUC_{[0,1]}\leqW2^{-n}\dash\WWKL$.
Let $H,K$ be computable functions such that $H\langle\id,FK\rangle$ is a realizer
of $\AUC_{[0,1]}$ for every realizer $F$ of $2^{-n}\dash\WWKL$.
Let $p$ be a name of $[0,1]$, which is mapped to a name $K(p)$ of a tree $T$
with a set $A=[T]$ of infinite paths such that $\mu(A)>2^{-n}$.
Now we consider $2^{n}+1$ distinct points $x_0,x_1,....,x_{2^{n}}\in[0,1]$
and let $m\in\IN$ be such that $2^{-m}<\min\{|x_i-x_j|:i,j\in\{0,...,2^{n}\},i\not=j\}$.
Since $H$ is uniformly continuous on the compact set $\{p\}\times A$, 
there is some prefix $w\prefix p$ such that all names in $H\langle w\IN^\IN,q\rangle$ determine their results
with precision better than $2^{-m-1}$ for all $q\in A$.
Since $K$ is continuous, we can assume that $w$ is long enough such that
$K(w\IN^\IN)$ contains only names of trees $S$ with sets $B=[S]$ of infinite paths
such that $\mu(B\setminus A)\leq2^{-n}(1-\mu(A))$. 
Now we consider extensions $p_0,p_1,...,p_{2^{n}}$ of $w$ which are names
of the singletons $\{x_0\}$, $\{x_1\}$, ..., $\{x_{2^{n}}\}$, respectively. 
Then $K(p_0)$,...,$K(p_{2^{n}})$ are names of trees $T_0,...,T_{2^n}$ with sets
of infinite paths $A_0,...,A_{2^n}$, respectively. Since $\mu(A_i)>2^{-n}$ for all $i=0,...,2^n$,
it is clear that there are distinct $k,j\in\{0,...,2^n\}$ such that $A_j\cap A_k\cap A\not=\emptyset$,
since otherwise we obtain for some $i=0,...,2^n$
\[\mu(A_i)=\mu(A_i\cap A)+\mu(A_i\setminus A)\leq\frac{1}{2^n+1}\mu(A)+2^{-n}(1-\mu(A))<2^{-n}\]
in contradiction to the assumption. 
Let now $q\in A_j\cap A_k\cap A$. Then $H\langle p_j,q\rangle$ is a name of $x_j$ and
$H\langle p_k,q\rangle$ is a name of $x_k$ and hence $|x_j-x_k|<2^{-m}$ according to the choice
of $w$. This is in contradiction to the definition of $m$. 
\end{proof}

Together with Theorem~\ref{thm:AC-PCC} this gives us an alternative proof of Proposition~\ref{prop:PCC-WWKL}.
Theorem~\ref{thm:AC-*-WWKL} also implies $\AUC_{[0,1]}\nleqW\C_2^*$, which was proved
in a different way by Arno Pauly in \cite[Theorem~5.2.1.4]{Pau11}.

Similarly as before Corollary~\ref{cor:CC-PC2} we can ask whether Theorem~\ref{thm:AC-PCC} can be strengthened to the statement $\AUC_{[0,1]}\leqSW\PCC_{[0,1]}$.
However, again for mere cardinality reasons this is not possible. As above we note that $\#\PC_{[0,1]}^{(n)}=|\IN|$, 
while obviously $\#\AUC_{[0,1]}=|\IR|$.
We obtain the following consequence of Proposition~\ref{prop:cardinality}.

\begin{corollary}
$\AUC_{[0,1]}\nleqSW\PC_{[0,1]}^{(n)}$ for all $n\in\IN$.
\end{corollary}

\section{Nash Equilibria}
\label{sec:NASH}

In this section we would like to prove (based on results of Arno Pauly) that there is a Las Vegas algorithm to compute Nash equilibria.
We recall from \cite{Pau10,NRTV07} that a pair $A,B\in\IR^{m\times n}$ of $m\times n$--matrices
is called a {\em bi-matrix game}. Any vector $s=(s_1,...,s_m)\in\IR^m$ with $s_i\geq 0$ for all $i=1,...,m$ and
$\sum_{j=1}^ms_j=1$ is called a {\em mixed strategy}. By $S^m$ we denote the set of these mixed strategies of dimension $m$.
Then a {\em Nash equilibrium} is a pair $(x,y)\in S^n\times S^m$ of strategies such that
\begin{enumerate}
\item $x^{\rm T}Ay\geq w^{\rm T}Ay$ for all $w\in S^n$ and
\item $x^{\rm T}By\geq x^{\rm T}Bz$ for all $z\in S^m$.
\end{enumerate}
John F.\ Nash \cite{Nas51} proved that for any bi-matrix game there exists a Nash equilibrium. 
By $\NASH_{n,m}:\IR^{m\times n}\times\IR^{m\times n}\mto\IR^n\times\IR^m$ we denote the corresponding problem 
\[\NASH_{n,m}(A,B):=\{(x,y)\in\IR^n\times\IR^m:(x,y)\mbox{ is a Nash equilibrium for $(A,B)$}\}\]
of finding a Nash-equilibrium for an $m\times n$ bi-matrix game and
by $\NASH:=\bigsqcup_{n,m\in\IN}\NASH_{n,m}$ we denote the coproduct of all such games for finite $m,n\in\IN$.
By \cite[Theorem~28]{Pau10} it follows that $\NASH$ is strongly idempotent, i.e., $\NASH\times\NASH\leqSW\NASH$.
Like in \cite[Theorem~24]{Pau10} we will use a variant of the well-known {\em matching pennies game} (see \cite{NRTV07}), 
which has a unique Nash equilibrium, in order to prove the following result.

\begin{lemma}
\label{lem:NASH-cylinder}
$\NASH$ is a cylinder.
\end{lemma}
\begin{proof}
Since $\NASH$ is strongly idempotent and $\id_{\IN^\IN}\equivSW\id_{[0,1]}$, it suffices to prove $\id_{[0,1]}\leqSW\NASH$. 
Given some input $a\in[0,1]$, we can compute the bi-matrix game $(A,B)$ given by
\[A:=\left(\begin{matrix} 1 & -1\\ -1 & a \end{matrix}\right),\;B:=\left(\begin{matrix} -1 & 1\\ 1 & -1 \end{matrix}\right).\]
We claim that the unique Nash equilibrium $(x,y)$ of the game $(A,B)$ is given by $x=(x_1,x_2):=(\frac{1}{2},\frac{1}{2})$ and 
$y:=(y_1,y_2)$ with $y_1:=\frac{1+a}{3+a}$ and $y_2:=1-y_1$.
This yields the desired reduction, since $a$ can be recovered from the unique output $(x,y)=\NASH_{2,2}(A,B)$ by
$a=\frac{2y_1}{1-y_1}-1$. It remains to prove the claim, which amounts to check that the above pair $(x,y)\in S^2\times S^2$
is the unique pair that satisfies
\begin{enumerate}
\item $(y_1-y_2)x_1+(ay_2-y_1)x_2\geq (y_1-y_2)w_1+(ay_2-y_1)w_2 $ and
\item $(x_2-x_1)y_1+(x_1-x_2)y_2\geq(x_2-x_1)z_1+(x_1-x_2)z_2$ 
\end{enumerate}
for all $w=(w_1,w_2),z=(z_1,z_2)\in S^2$. If we consider the case $y_1,y_2\not\in\{0,1\}$, then the only
way to satisfy the second constraint (2) is by balancing both addends, i.e., $x_2-x_1=x_1-x_2$, which yields $x_1=x_2=\frac{1}{2}$.
This is because an unbalanced pair $x_2-x_1\not=x_1-x_2$ would always allow to increase the weight $y_1$ or $y_2$ of the
larger component of the pair, which is possible if both weights $y_1$ and $y_2$ are smaller than $1$.
The corresponding modified pair of weights $z_1,z_2$ would then violate (2).
Likewise, if we consider the case $x_1,x_2\not\in\{0,1\}$ balancing the addends in the first constraint (1) yields $y_1-y_2=ay_2-y_1$, which implies $y_1=\frac{1+a}{3+a}$ after the
substitution $y_2=1-y_1$. Now we still need to consider the case where we allow $x_1,x_2,y_1,y_2\in\{0,1\}$.
For instance, if $y_2=1$, then $y_1=0$ and the second constraint (2) can only be satisfied if $x_1-x_2\geq x_2-x_1$, which means if $x_1\geq x_2$. 
In this case the first constraint can only be satisfied if $-1=y_1-y_2\geq ay_2-y_1=a$, which is impossible for $a\in[0,1]$.  
Likewise, the other cases with values in $\{0,1\}$ can be ruled out. What remains is the above unique Nash equilibrium $(x,y)$.
\end{proof}

Arno Pauly proved that the problem $\NASH$ is Weihrauch equivalent to the idempotent closure $\AUC_{[0,1]}^*$ of all-or-unique choice on the unit interval (see \cite{Pau10}).

\begin{theorem}[Nash equilibria, Arno Pauly 2010]
$\NASH\equivSW\AUC_{[0,1]}^*$.
\end{theorem}

The proof of $\NASH\equivW\RDIV^*$ can be found in  \cite[Corollary~40]{Pau10}. Moreover,
the equivalence $\RDIV^*\equivSW\AUC_{[0,1]}^*$ holds by Proposition~\ref{prop:RDIV-AUC}. 
It is easy to see that $\AUC_{[0,1]}^*$ is a cylinder and hence we obtain by Lemma~\ref{lem:NASH-cylinder} that even strong Weihrauch equivalence
holds as stated in the previous result.
Since $\WWKL$ is idempotent, we obtain the following immediate conclusion of Theorem~\ref{thm:AC-PCC}.
 
\begin{corollary}
\label{cor:NASH-WWKL}
$\NASH\leqW\WWKL$.
\end{corollary}

This means that there is a Las Vegas algorithm for computing Nash equilibria.
In terms of $\RDIV$ this algorithm is quite involved and can be found in \cite{Pau10}, while we have shown how
$\RDIV$ can be computed in Las Vegas style (see the proof of Theorem~\ref{thm:AC-PCC} and Corollary~\ref{cor:RDIV-WWKL}). 
An obvious question is whether Corollary~\ref{cor:NASH-WWKL} can be improved 
to the reduction $\NASH\leqW\PCC_{[0,1]}$, which would mean that there is a Las Vegas algorithm,
whose random guesses can always be organized in a connected interval. 
We will see that this is not the case and we will even obtain $\NASH\nleqW\ConC_{[0,1]}$.
We start with the following result, which improves \cite[Proposition~7.1]{BLP12}.

\begin{proposition}
\label{prop:C2xAC-PCC}
$\C_2\times\AUC_{[0,1]}\nleqW\ConC_{[0,1]}$.
\end{proposition}
\begin{proof}
Let us assume for a contradiction that $\C_2\times\AUC_{[0,1]}\leqW\ConC_{[0,1]}$.
Then there are computable $H,K$ such that $H\langle\id,FK\rangle$ is a realizer of $\C_2\times\AUC_{[0,1]}$
whenever $F$ is a realizer of $\ConC_{[0,1]}$.
Without loss of generality, we can assume that we represent $[0,1]$ with the following signed-digit representation:
\[\rho:\In\{-1,0,1\}^\IN\to[0,1],p\mapsto\sum_{n=0}^\infty p(n)2^{-n}.\]
Let now $p$ be a name of $\{0,1\}\times[0,1]$. Then $K(p)$ is the name of an interval $I$
and $H$ is uniformly continuous on the compact set $\{p\}\times N$, where $N$ is the set of all names of points in $I$
(this set is compact, since we are using a signed-digit representation).
\begin{enumerate}
\item[(a)]
Hence there is a number $n\in\IN$ such that all points in $H\langle p|_n\IN^\IN,q|_n\IN^\IN\rangle$ with $q\in N$
have a fixed first discrete component in $\{0,1\}$. 
Since we use the signed-digit representation $\rho$ for $[0,1]$, each $q|_n=q(0)...q(n-1)$ determines a point $x\in I$ 
up to precision $2^{-n+1}$, more precisely, if $x,y\in I$ are points with $|x-y|\leq2^{-n+1}$, then $x,y$ have names $r,s$, respectively, with $r|_n=s|_n$.
\item[(b)]
There is also a number $k\geq n$ such that all second components of the points $H\langle p|_k\IN^\IN,q|_k\IN^\IN\rangle$ with $q\in N$
determine values in $[0,1]$, which are identical up to precision $2^{-n-2}$; more precisely, if $r,s\in\pi_2 H\langle p|_k\IN^\IN,q|_k\IN^\IN\rangle$ for some $q\in N$,
then $r,s$ are names of points $x,y\in[0,1]$ with $|x-y|<2^{-n-2}$.
\end{enumerate}
We now choose $2^{n+2}$ equi-distant points $x_1,...,x_{2^{n+2}}\in[0,1]$ including the endpoints $0,1$.
In particular, $|x_i-x_j|>2^{-n-2}$ for all $i,j$ with $i\not=j$.
Let now $q_a,q_b$ be names of the two endpoints of $I$
and let $a,b\in [0,1]$ be the value of the second components of $H\langle p,q_a\rangle$
and $H\langle p,q_b\rangle$, respectively. We fix some $x\in[0,1]$. Let $A_x=\{0,1\}\times\{x\}$. 
Then due to continuity of $K$ there is a name $p_i$ of $A_x$ for every $i\in\IN$
such that $p|_{k+i}$ is a prefix of $p_i$ and $K(p_i)$ is a name of an interval $I_i$ such that $\sup_{y\in I_i}\dist_I(y)<2^{-i}$ (where $\dist_I(y):=\inf\{|z-y|:z\in I\}$).
We note that $p_i\to p$ for $i\to\infty$.
Let us assume that $I_i\not\In I^\circ$ for all $i$ (where $I^\circ$ denotes the interior of $I$). 
Then there is a sequence of points $(y_i)$ with $y_i\in I_i$ and a corresponding sequence
of names $(q_i)$ such that $q_i\to q_a$ or $q_i\to q_b$ for $i\to\infty$. Without loss of generality, we assume $q_i\to q_a$.
Due to continuity of $H$ we obtain
\[H\langle p,q_a\rangle=H\left\langle\lim_{i\to\infty}p_i,\lim_{i\to\infty}q_i\right\rangle=\lim_{i\to\infty}H\langle p_i,q_i\rangle.\]
Since the second component of $H\langle p,q_a\rangle$ is a name for $a$ and the second components of all the $H\langle p_i,q_i\rangle$
are names for $x$, we obtain $x=a$. Hence, in the general case we have $x\in\{a,b\}$.
In other words, if $x\not\in\{a,b\}$, then $A_x$ has a name $p_x$ with prefix $p|_k$ such that $K(p_x)$ is a name of an interval $I_x$ with $I_x\In I^\circ$.
Among the $2^{n+2}$ points $x_i$ there are at least $2^{n+1}$, which are different from $a,b$.
Let us assume, without loss of generality, that the points $x_1,...,x_{2^{n+1}}$ are all different from $a,b$.
Now we claim that
\begin{enumerate}
\item $I_{x_i}\cap I_{x_j}=\emptyset$ for different $i,j\in\{1,...,2^{n+1}\}$,
\item $\lambda(I_{x_i})>2^{-n}$ for all $i\in\{1,...,2^{n+1}\}$.
\end{enumerate}
Together this is clearly a contradiction since 
\[\sum_{i=1}^{2^{n+1}}\lambda(I_{x_i})>2^{n+1}2^{-n}>1=\lambda([0,1]).\]
We first prove (1). Let us assume that $z\in I_{x_i}\cap I_{x_j}$ for $i\not=j$ and $q$ is a name of $z$.
It follows that the second components of $H\langle p_{x_i},q\rangle$ and $H\langle p_{x_j},q\rangle$ are names of $x_i$ and $x_j$, respectively,
and since $p_{x_i}$ and $p_{x_j}$ have the prefix $p|_k$ in common, we obtain $|x_i-x_j|<2^{-n-2}$ in contradiction
to the choice of these points $x_i,x_j$.
We now prove (2). Let $i\in\{1,...,2^{n+1}\}$.
Due to continuity of $K$ we can choose some $m\geq k$ such that all intervals $J$ named in $K(p_{x_i}|_m\IN^\IN)$
satisfy $\sup_{y\in J}\dist_{I_{x_i}}(y)<2^{-n-1}$ and $J\In I$ (the latter is possible since $I_{x_i}\In I^\circ$). Now the two sets $\{0\}\times\{x_i\}$ and $\{1\}\times\{x_i\}$
have names $p_0,p_1$ that share the common prefix $p_{x_i}|_m$.
Let $J_0\In I$ and $J_1\In I$ be the intervals named by $K(p_0)$ and $K(p_1)$, respectively and let $x\in J_0$ and $y\in J_1$.
If $|x-y|\leq2^{-n+1}$, then there are names $q,r\in N$ of $x,y$, respectively, such that $q|_n=r|_n$ and hence 
$H\langle p_0,q\rangle$ and $H\langle p_1,r\rangle$ must name identical first components, which is a contradiction to the choice of $p_0,p_1$.
Hence we obtain $\inf\{|x-y|:x\in J_0,y\in J_1\}\geq2^{-n+1}$. But this implies 
\[\lambda(I_{x_i})=\sup\{|x-y|:x,y\in I_{x_i}\}>2^{-n+1}-2\cdot2^{-n-1}=2^{-n}.\]
This completes the proof.
\end{proof}

Since $\C_2\times\AUC_{[0,1]}\leqW\AUC_{[0,1]}^*$ we obtain the following corollary.

\begin{corollary}
\label{cor:AC-CC}
$\AUC_{[0,1]}^*\nleqW\ConC_{[0,1]}$.
\end{corollary}

This means that a method to compute zeros of continuous functions with changing signs
cannot help to compute Nash equilibria.  

\begin{corollary}
\label{cor:NASH-IVT}
$\NASH\nleqW\IVT$.
\end{corollary}

We mention that Proposition~\ref{prop:C2xAC-PCC} together with
$\C_2\times\AUC_{[0,1]}\leqW\PC_{[0,1]}$ also yields another 
proof of Proposition~\ref{prop:PC-CC}.
Since $\C_2\leqW\AUC_{[0,1]}\leqW\PCC_{[0,1]}\leqW\ConC_{[0,1]}$ we also obtain the following corollary of Proposition~\ref{prop:C2xAC-PCC}.

\begin{corollary}
$\AUC_{[0,1]}$, $\PCC_{[0,1]}$ and $\ConC_{[0,1]}$ are not idempotent.
\end{corollary}

The first fact can also easily be deduced from the number of mind changes required and the 
latter fact was already proved in \cite[Theorem~7.3]{BLP12} in a slightly different way.
We mention that the situation also yields another instance of a difference between suprema and products, since we obtain
\[\C_2\sqcup\AUC_{[0,1]}\lW\C_2\times\AUC_{[0,1]}.\]

In particular, Corollary~\ref{cor:AC-CC} implies that $\AUC_{[0,1]}^*\nleqW\PCC_{[0,1]}$.
As a final result in this section we would like to clarify the inverse relation between $\PCC_{[0,1]}$ and $\AUC_{[0,1]}^*$.
The separation can be achieved using the concept of a level (as introduced by Hertling \cite{Her96b,Her96}), which is preserved downwards by Weihrauch reducibility 
(i.e., if $f\leqW g$, then the level of $f$ is less or equal to the level of $g$).
Since $\AUC_{[0,1]}^*\leqW\LPO^*$, it follows that $\AUC_{[0,1]}^*$ has at most the level of $\LPO^*$, which is $\omega$ (the first transfinite ordinal),
while $\PCC_{[0,1]}$ has no level, since its entire domain consists of points of discontinuity. 
This implies $\PCC_{[0,1]}\nleqW\LPO^*$ and altogether we obtain the following result.

\begin{corollary}
$\PCC_{[0,1]}\nW\AUC_{[0,1]}^*$.
\end{corollary}

In particular, we obtain the following corollary.

\begin{corollary}
\label{cor:IVT-NASH}
$\IVT\nW\NASH$.
\end{corollary}

\begin{figure}[htb]
\begin{center}
\begin{tikzpicture}[scale=.5,auto=left,every node/.style={fill=black!15},y=-2cm]
\input{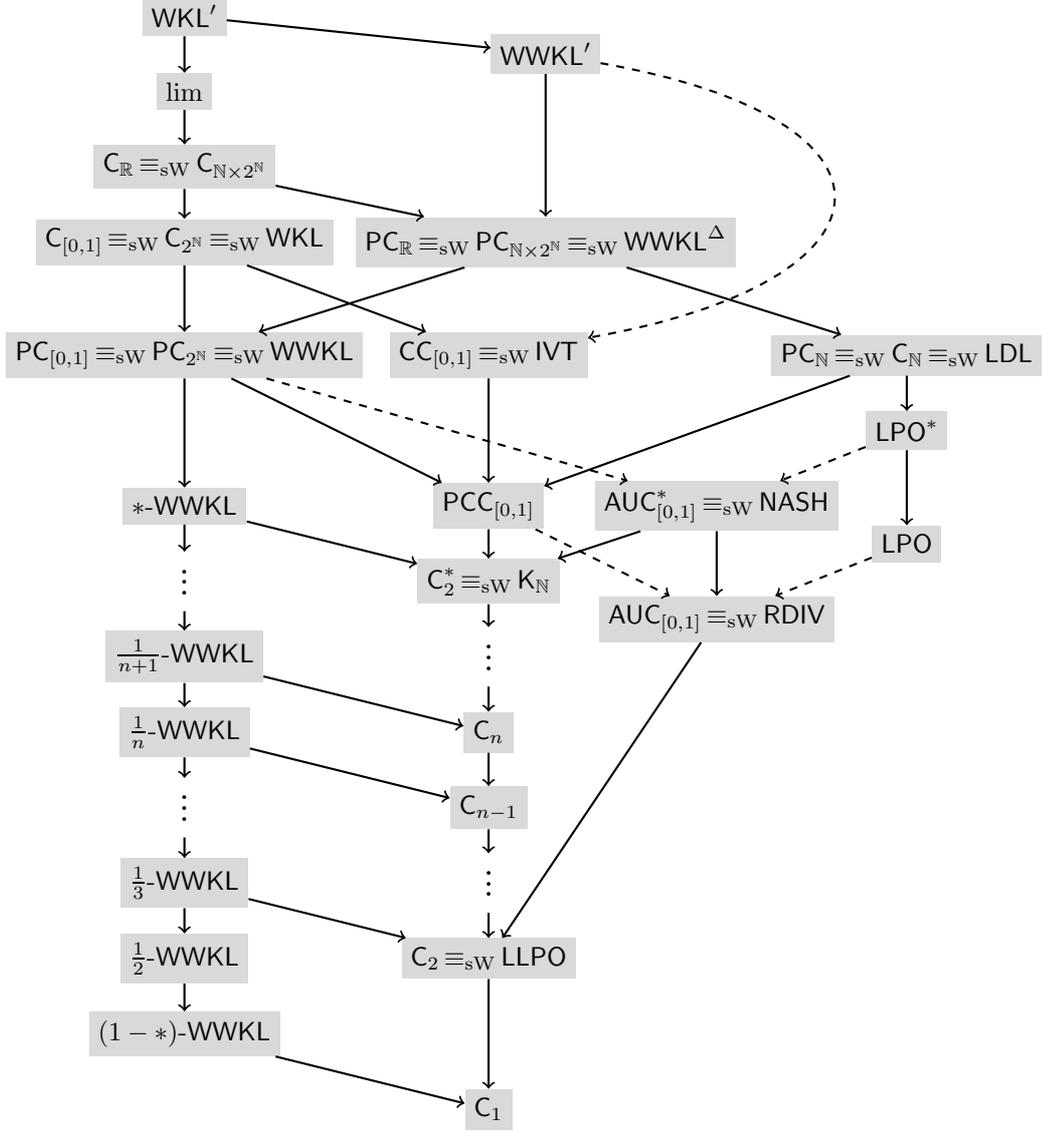}
\end{tikzpicture}
\caption{(Probabilistic choice in the Weihrauch lattice). All solid lines indicate strong Weihrauch reductions $\leqSW$ against the direction of the arrow, i.e., if
$f\leqSW g$, then the arrow points from $g$ to $f$ (which corresponds to the direction of logical implication).
All dashed lines indicate that we only have ordinary Weihrauch reductions $f\leqW g$ in those cases. The diagram is
 complete (with regards to ordinary Weihrauch reducibility) up to transitivity.}
\label{fig:diagram}
\end{center}
\end{figure}

\section{Conclusions}
\label{sec:conclusions}

In Figure~\ref{fig:diagram} we illustrate the fragment of the Weihrauch lattice that we have studied in this paper.

We emphasize that many separations presented in this paper hold for purely topological reasons. That is, we 
get analogous results if we replace (strong) Weihrauch reducibility by its topological counterpart that is defined
analogously, but with continuous $H,K$ instead of computable $H,K$. 
In particular, Propositions~\ref{prop:example-unit-cantor}, \ref{prop:CC-PC}, and \ref{prop:C2xAC-PCC}, as well as Theorems~\ref{thm:probability-dependency} and \ref{thm:AC-*-WWKL}, 
hold analogously for the topological variant of Weihrauch reducibility.

There are numerous structural questions that we have not addressed or answered in our study.
We mention some examples:

\begin{enumerate}
\item Is $\WWKL'$ closed under composition?
\item Or is $\WWKL'*\WWKL'\equivW\WWKL''$?
\item Is $\WWKL'\leqW\PC_{\IN^\IN}$?
\end{enumerate}

The techniques used to prove $\WKL'*\WKL'\equivW\WKL''$ in \cite{BGM12} cannot be directly transferred to the case of $\WWKL$,
since the proofs crucially exploit that $\WKL$ is a cylinder, which $\WWKL$ is not according to Corollary~\ref{cor:cylinder}.

Of course, it would be very interesting to find out whether further concrete problems (besides determining zeros or computing
Nash equilibria) admit Las Vegas algorithm or other types of randomized algorithms. 
In one forthcoming paper we will study the Vitali Covering Theorem and other results from measure theory from this perspective
and in a second paper we will investigate the relation between Martin-L\"of randomness and Weak Weak K\H{o}nig's Lemma.

d to prove $\WKL'*\WKL'\equivW\WKL''$ in \cite{BGM12} cannot be directly transferred to the case of $\WWKL$,
since the proofs crucially exploit that $\WKL$ is a cylinder, which $\WWKL$ is not according to Corollary~\ref{cor:cylinder}.

Of course, it would be very interesting to find out whether further concrete problems (besides determining zeros or computing
Nash equilibria) admit Las Vegas algorithm or other types of randomized algorithms. 
In one forthcoming paper we will study the Vitali Covering Theorem and other results from measure theory from this perspective
and in a second paper we will investigate the relation between Martin-L\"of randomness and Weak Weak K\H{o}nig's Lemma.

\section*{Acknowledgments}

We would like to thank the anonymous referees for their detailed and helpful remarks.
We are also grateful to Mathieu Hoyrup and Arno Pauly for their comments on an earlier draft of this article.

\bibliographystyle{plain}
\bibliography{C:/Users/Vasco/Dropbox/Bibliography/lit}

\end{document}